\documentclass[a4paper]{article}

\usepackage{authblk}
\usepackage{maths_symbols}

\newcommand{\pr}[0]{\ensuremath{\textup{pr}}}
\newcommand{\pos}[0]{\ensuremath{\textup{pos}}}
\newcommand{\obs}[0]{\ensuremath{\textup{obs}}}

\newcommand{\opt}[0]{\ensuremath{\textup{opt}}}

\newcommand{\G}[0]{\ensuremath{\mathcal{G}}}
\newcommand{\X}[0]{\ensuremath{\mathfrak{X}}}

\begin{document}

\title{Optimal low-rank posterior covariance approximation in linear Gaussian inverse problems on Hilbert spaces}

\author[]{Giuseppe Carere\thanks{~giuseppe.carere@uni-potsdam.de, ORCID ID: 0000-0001-9955-4115} }
\author[]{Han Cheng Lie\thanks{~han.lie@uni-potsdam.de, ORCID ID: 0000-0002-6905-9903}}
\affil[]{~Institut f\"ur Mathematik, Universit\"at Potsdam, Potsdam OT Golm 14476, Germany}
\renewcommand\Affilfont{\small}

\date{}
\maketitle

\begin{abstract}
	For linear inverse problems with Gaussian priors and Gaussian observation noise, the posterior is Gaussian, with mean and covariance determined by the conditioning formula. The covariance is the central object for uncertainty quantification, as it encodes the variability of the posterior distribution and thus the uncertainty in the posterior mean estimate. Using the Feldman--Hajek theorem, we analyse the prior-to-posterior update and its low-rank approximation for infinite-dimensional Hilbert parameter spaces and finite-dimensional observations. We show that the posterior distribution differs from the prior on a finite-dimensional subspace, and construct low-rank approximations to the posterior covariance, while keeping the mean fixed. Since in infinite dimensions, not all low-rank covariance approximations yield approximate posterior distributions which are equivalent to the posterior and prior distribution, we characterise the low-rank covariance approximations which do yield this equivalence, and their respective inverses, or `precisions'. For such approximations, a family of measure approximation problems is solved by identifying the low-rank approximations which are optimal for various losses simultaneously. These loss functions include the family of R\'enyi divergences, the Amari $\alpha$-divergences for $\alpha\in(0,1)$, the Hellinger metric and the Kullback--Leibler divergence. Our results extend those of Spantini et al.\ (SIAM J.\ Sci.\ Comput.\ 2015) to Hilbertian parameter spaces, and provide theoretical underpinning for the construction of low-rank approximations of discretised versions of the infinite-dimensional inverse problem, by formulating discretisation independent results.
\end{abstract}


\section{Introduction}
\label{sec:introduction}

The class of Bayesian inverse problems with linear forward models and Gaussian priors plays a special role in the context of Bayesian statistical inference. For example, this class of linear Gaussian inverse problems appears naturally in the Laplace approximation of posteriors for nonlinear statistical inverse problems, and the classical Kalman filter can be understood as an iterative solution method for a sequence of linear Gaussian inverse problems. 
A particularly attractive feature of the class of linear Gaussian inverse problems is the availability of a closed-form solution, in the case where the parameter space is a separable Hilbert space. In this case, given a linear forward model $G$ with codomain $\mathbb{R}^{n}$ for some $n\in\mathbb{N}$, a realisation $y$ of the $\mathbb{R}^{n}$-valued data random variable 
\begin{equation*}
 Y=GX+\zeta,\quad \zeta\sim \mathcal{N}(0,\mathcal{C}_\obs),
\end{equation*}
and given a Gaussian prior $\mu_\pr=\mathcal{N}(m_\pr,\mathcal{C}_\pr)$ for the unknown parameter $X$, the solution $\mu_\pos$ to the Bayesian inverse problem is a Gaussian measure $\mathcal{N}(m_\pos,\mathcal{C}_\pos)$. The posterior mean $m_\pos$ and the posterior covariance $\mathcal{C}_\pos$ can be computed explicitly:
 \begin{align*}
	 m_\pos = m_\pr+ \mathcal{C}_{\pos}G^*\mathcal{C}_{\obs}^{-1}(y-Gm_\pr),\quad	 \mathcal{C}_{\pos} = \mathcal{C}_{\pr} - \mathcal{C}_{\pr} G^*(\mathcal{C}_{\obs}+G\mathcal{C}_{\pr} G^*)^{-1}G\mathcal{C}_{\pr},
\end{align*}
see e.g.\ \cite[Example 6.23]{Stuart2010}. It should be noted that $\mathcal{C}_\pos$ does not depend on the realisation $y$ of $Y$.

The availability of closed-form solutions to linear Gaussian inverse problems endows these problems with structure that makes them interesting objects to study in the context of measure approximation problems.
Measure approximation problems have become ubiquitous in modern statistical inference, often because one cannot sample exactly from the probability measure of interest, e.g.\ for computational cost reasons, or because one has only partial information about the measure of interest. 
In the context of Bayesian inverse problems, we can also consider measure approximation problems as a way to analyse the Bayesian prior-to-posterior update.

Computational studies of Bayesian inverse problems on high- but finite-dimensional parameter spaces show that the data is often `informative'---i.e., that the posterior differs from the prior---only on a subspace of much lower dimension than the dimension of the parameter space; see e.g.\ \cite{Flath2011}. In \cite{Cui2014}, a similar subspace is called a `likelihood-informed subspace'. Since the posterior is obtained by reweighting the prior by the likelihood, this likelihood-informed subspace is determined by how the concentration of the likelihood interacts with the concentration of prior.
In the case of linear Gaussian inverse problems, the concentration of the likelihood and the concentration of the prior are described by the eigenpairs of the Hessian of the negative log-likelihood, referred to simply as `Hessian' in this work, and the eigenpairs of the prior precision. These ideas are used in \cite{Flath2011} to identify low-rank approximations of the posterior covariance matrix. In \cite{Spantini2015}, the optimality of these posterior covariance approximations with respect to a family of spectral loss functions is shown. In particular, the leading generalised eigenvectors of the matrix pencil defined by the Hessian and the prior precision build a hierarchy of nested low-dimensional subspaces on which the posterior differs from the prior. 
If only a few directions in the parameter space need to be stored to be able to approximate the posterior distribution well, then this can be done before observing the data. Indeed, these directions are independent of the data, because the prior precision and the Hessian are independent of the data. In high but finite-dimensional parameter spaces, this leads to considerable computational and storage savings. Nowadays, the latter is important since read and write operations from memory often form the bottleneck in modern computational hardware, c.f.\ \cite{Ordentlich2025}.

So far, the existence of optimal low-rank approximations and likelihood-informed subspaces for linear Gaussian prior-to-posterior updates has only been proven for posterior distributions on finite-dimensional parameter spaces. Such low-rank approximations are exploited in \cite{Bui-Thanh2012b,Bui-Thanh2013} to obtain computationally tractable uncertainty quantification in high-dimensional inverse problems. In these works it is noticed that the spectral decay of the Hessian of a discretised and linearised version of an inverse problem seems independent of the discretisation dimension. As a consequence, also the spectral decay of the prior-preconditioned Hessian is independent of the discretisation dimension. This observation is central in the effort of making the resolution of the inverse problem scalable. In order to provide theoretical underpinning for this behaviour, it is fundamental to formulate the approximation procedure centered around the prior-preconditioned Hessian directly on the native infinite-dimensional space. 

Analysing low-rank posterior covariance approximations in infinite dimensions adheres to the principle proposed in \cite{Stuart2010} of studying Bayesian inverse problems in their native function space formulation, while deferring the discretisation to the very end. A constructive formulation of optimal low-rank posterior covariance approximations in infinite dimensions can then be used to show that constructed discretised counterparts are consistent under increasingly fine discretisations levels, for a given inverse problem of interest. This is in parallel to showing the consistency of discretisations of the full posterior, see for example \cite{Sanz-Alonso2024}. Consistent discretisations of optimal low-rank posterior covariance approximations are dimension robust, meaning that the posterior covariance approximation error in a discretised setting does not blow up under discretisation refinement but converges to the optimal approximation error in the infinite-dimensional setting.

We therefore formulate and solve the low-rank posterior covariance approximation problem of \cite{Spantini2015}, which provides a finite-dimensional analysis, instead in infinite-dimensional separable Hilbert spaces. While \cite[Example 6.23]{Stuart2010} provides a formulation of the linear Gaussian inverse problem in infinite dimensions, in the generalisation of the optimal low-rank posterior covariance approximation analysed by \cite[Section 2]{Spantini2015} certain challenges appear.

\subsection{Challenges in infinite dimensions}
In the finite-dimensional context of \cite{Spantini2015}, the above equation updating $\mathcal{C}_\pr$ to $\mathcal{C}_\pos$ provides a starting point for the approximation procedure. Also the corresponding equation which updates $\mathcal{C}_\pr^{-1}$ to $\mathcal{C}_\pos^{-1}$, c.f.\ \cite[eq.\ (6.13a)]{Stuart2010}, provides a starting point for the approximation. These are called the `prior precision' and `posterior precision' respectively. An operator pencil involving $\mathcal{C}_\pr^{-1}$ is central in the result of \cite[Theorem 2.3]{Spantini2015}. When the prior distribution is nondegenerate, these can be interpreted as full-rank matrices, that is, finite-dimensional, hence bounded, linear operators. In infinite dimensions, $\mathcal{C}_\pr^{-1}$ and $\mathcal{C}_\pos^{-1}$ are no longer bounded, and they are not even defined on the entire parameter space. In fact, $\ran{\mathcal{C}_\pr^{1/2}}$, the range of the self-adjoint square root of $\mathcal{C}_\pr$, is called the `Cameron--Martin space' of the prior distribution and contains the domain of $\mathcal{C}_\pr^{-1}$. This space is a proper subspace of the parameter space in infinite dimensions, and with probability 1, draws from the prior distribution do not belong to this space. This makes the required analysis of approximations based on the prior-to-posterior precision update more delicate. 

Another complication of the infinite-dimensional setting is that, unlike in the finite-dimensional setting, not all approximations of the posterior mean and covariance result in approximate posterior measures that are equivalent to the exact posterior distribution, even if they have the same support; see the Feldman--Hajek theorem, recalled below in \Cref{thm:feldman--hajek}. Here, `equivalent' means that the approximate posterior has a density with respect to the exact posterior distribution, and vice versa. Since the prior and the posterior distributions are equivalent for linear Gaussian inverse problems with finite-dimensional data, approximate posteriors which are not equivalent to the exact posterior are also not equivalent to the prior distribution. In fact, nonequivalent Gaussian measures on infinite-dimensional spaces are necessarily mutually singular. That is, they assign full measure to disjoint measurable sets, which is an undesirable property for the approximate posterior and exact posterior/prior distribution to have. Thus, an understanding of which approximate updates of the prior covariance lead to equivalent approximation posterior distributions, with probability 1 with respect to the data $Y$, is needed to construct approximate posterior measures equivalent to the exact posterior.

A third complication is that the analysis of the finite-dimensional setting in \cite{Spantini2015} relies on certain finite-dimensional results and concepts. 
For example, in approximating the posterior covariance, a certain loss functional is used to measure the closeness of the approximate posterior covariance to the exact posterior covariance. The coercivity of this loss functional is used to prove some results in the finite-dimensional setting. However, in our infinite-dimensional formulation, the analogous coercivity statement does not hold.
Also, Fr\'echet differentiability of this loss functional, useful for finding extreme points of the loss, cannot be deduced in the same way as in the finite-dimensional case, as the latter case relies on the finite-dimensional result of \cite[Theorem 1.1]{Lewis1996}.

\subsection{Contributions}

This work provides a rigorous analysis for the infinite-dimensional version of the Bayesian prior-to-posterior covariance and precision updates and constructs optimal low-rank approximations thereof. We assume a linear Gaussian inverse problem in which the parameter space is a possibly infinite-dimensional separable Hilbert space, the observation space is finite-dimensional and the prior is nondegenerate and has mean zero. We identify optimal Gaussian approximations to the true posterior, keeping the mean fixed, using low-rank measure approximation problems. Our results extend the results of \cite{Spantini2015} that are developed for finite-dimensional parameter spaces, to the case where the parameter space is possibly infinite-dimensional. This shows a certain dimension independence of the results of \cite{Spantini2015}. In related work, see \cite{PartII}, we study low-rank posterior mean approximation, and give some insight on joint posterior mean and covariance approximation. We highlight the main contributions of this paper.

The first main contribution is \Cref{prop:bayesian_feldman_hajek}. In particular:
\begin{itemize}
	\item 
		It formulates three operators and their relation in infinite dimensions. The three operators are important in the approximation procedure, and are given by the prior-preconditioned Hessian, the \emph{posterior}-preconditioned Hessian and the posterior covariance preconditioned with the prior precision. As the relations are one-to-one, these operators contain the same information. It was already noted in previous works in finite dimensions, see e.g.\ \cite[Proposition 10]{Jagalur-Mohan2021} and \cite[Section 3.4.1]{Huan2024}, that these operators and various other transformations of them contain the same information and are the central object for studying the quality of the finite-dimensional low-rank posterior approximation.
	\item
		It gives one-to-one relations between the above three operators and the Hilbert--Schmidt operator which mixes the prior and posterior covariance in the Feldman--Hajek theorem, given by
		\begin{align}
			\label{eqn:informal_fh_operator}
			R(\mathcal{C}_\pos\Vert\mathcal{C}_\pr) \coloneqq (\mathcal{C}_\pr^{-1/2}\mathcal{C}_\pos^{1/2})(\mathcal{C}_\pr^{-1/2}\mathcal{C}_\pos^{1/2})^* - I.
		\end{align}
	
		The Feldman--Hajek theorem gives necessary and sufficient conditions for equivalence of Gaussian measures, and the connection with the three operators given here, shows that these operators essentially all quantify the amount of similarity and equivalence between the prior and the posterior distribution. This provides intuitive motivation for the importance of this family of operators in the study of optimal posterior approximation.
	\item
		It shows that this family of operators can be diagonalised in the common Cameron--Martin space of the prior and the posterior. In particular, this implies that the diagonalisations of the above family of operators have interpretations as operator pencils as in the finite-dimensional case.
	\item
		It shows that the prior and posterior distribution differ only on a finite-dimensional subspace, which is a subspace of the Cameron--Martin space of both the prior and posterior. Its dimension equals the rank of the Hessian of the negative log-likelihood, or equivalently, the rank of the forward model $G$. 
\end{itemize}

The second main contribution is given by \Cref{lemma:properties_equivalence} and \Cref{prop:range_of_K}. They are stated for an arbitrary Gaussian measure $\mu_{1}=\mathcal{N}(m_1,\mathcal{C}_1)$ with $\mathcal{C}_1$ injective. Among the low-rank updates of the covariance $\mathcal{C}_1$ and the precision $\mathcal{C}_1^{-1}$, these results characterise those low-rank updates which satisfy an equivalence property, namely that when keeping the mean fixed, they correspond to approximate distributions that are equivalent to $\mu_1$. Furthermore, these results also characterise the approximate precisions and covariances which correspond to respectively the low-rank covariance and precision updates satisfying this equivalence property. 
In the Bayesian context, \Cref{lemma:properties_equivalence} and \Cref{prop:range_of_K} also show that in infinite dimensions, not all updates of the prior covariance of the form considered in \cite{Spantini2015} satisfy this equivalence property, and not all updates of the prior covariance that do satisfy this property can be constructed as the inverse of an update of the prior precision considered in \cite{Spantini2015}. Our results give a necessary and sufficient condition on the range of the low-rank updates, under which such updates of the prior covariance do in fact satisfy the equivalence property. This provides a tool to inflate or deflate the covariance of a Gaussian measure while retaining access to Radon--Nikodym derivatives, e.g.\ to deflate prior covariance or inflate posterior covariance.

The third main contribution is to solve a family of Gaussian measure approximation problems in which we approximate the posterior covariance and keep the mean fixed, for example at the exact posterior mean. Using the main result mentioned above and other technical developments, the solutions are obtained in \Cref{thm:opt_covariance_and_precision} and \Cref{cor:optimal_covariance_for_amari_and_hellinger}, and can be stated informally as follows.

\begin{theorem*}[Informal version of \Cref{thm:opt_covariance_and_precision} and \Cref{cor:optimal_covariance_for_amari_and_hellinger}]
Let $r\leq n$, let $m\in\H$, and let $\mathcal{C}_\pr$ and $\mathcal{C}_\pos$ denote the prior and posterior covariances respectively. Let $R(\mathcal{C}_\pos\Vert\mathcal{C}_\pr)$ be the Feldman--Hajek operator \eqref{eqn:informal_fh_operator} and let $R(\mathcal{C}_\pos\Vert\mathcal{C}_\pr)_r$ be a truncated rank-$r$ singular value decomposition of $R(\mathcal{C}_\pos\Vert\mathcal{C}_\pr)$. The following holds.
\begin{enumerate}
	\item 
		\label{item:informal_1}
		The covariance
		$\mathcal{C}^\opt_r = \mathcal{C}_\pr^{1/2}(I-[-R(\mathcal{C}_\pos\Vert\mathcal{C}_\pr)_r])\mathcal{C}_\pr^{1/2}$
			minimises, over all covariances $\mathcal{C}_\pr-KK^*>0$ with linear $K:\R^r\rightarrow\ran{\mathcal{C}_\pr}\subset\H$, the loss
		\begin{align*}
			D\left(\ \mathcal{N}(m,\mathcal{C}_\pos)\ \Vert\ \mathcal{N}(m,\mathcal{C}_\pr-KK^*)\ \right),
		\end{align*}
		for all $D(\cdot\Vert\cdot)$ in a set of divergences, which includes the forward and reverse Kullback--Leibler divergence, the R\'enyi divergences of order $\rho$ with $\rho\in(0,1)$, the Amari-$\alpha$ divergences with $\alpha\in(0,1)$, and the Hellinger distance.
	\item With $(\lambda_i,w_i)_i\subset (-1,0]\times\ran{\mathcal{C}_\pr^{1/2}}$ the eigenpairs of $R(\mathcal{C}_\pos\Vert\mathcal{C}_\pr)$, we have
		\begin{align*}
			\mathcal{C}^\opt_r = \mathcal{C}_\pr - \sum_{i=1}^{r}(-\lambda_i)(\mathcal{C}_\pr^{1/2})w_i\otimes (\mathcal{C}_\pr^{1/2} w_i).
		\end{align*}
	\item The solution $\mathcal{C}^\opt_r$ is unique if and only if the following holds: $\lambda_{r+1}=0$ or $\lambda_r<\lambda_r+1$. 
	\item 
		The inverse of $\mathcal{C}^\opt_r$ is given by the unbounded operator
		\begin{align*}
			\mathcal{P}^\opt_r \coloneqq \mathcal{C}_\pr^{-1} + \sum_{i=1}^{r}\left(\frac{-\lambda_i}{1+\lambda_i}\right) (\mathcal{C}_\pr^{-1/2}w_i)\otimes (\mathcal{C}_\pr^{-1/2}w_i),
		\end{align*}
		which is optimal among all precisions $\mathcal{C}_\pr^{-1}+UU^*$ with linear $U:\R^r\rightarrow\H$, when measuring the approximation quality of the corresponding covariances $(\mathcal{C}_\pr^{-1}+UU^*)^{-1}$ using the losses considered in \cref{item:informal_1}.
\end{enumerate}
\end{theorem*}

The losses considered above are all spectral loss functions in the sense that when applied to any two Gaussian measures, they depend only on the spectrum of a suitable operator. The range condition $\ran{K}\subset\ran{\mathcal{C}_\pr}$ in \cref{item:informal_1} ensures that the resulting approximate posterior obtained by approximating the covariance and keeping the mean fixed is equivalent to the exact posterior. A precise statement on the optimal solutions for the covariance approximation problem, which includes an expression for the corresponding minimal loss, is provided in \Cref{thm:opt_covariance_and_precision} and \Cref{cor:optimal_covariance_for_amari_and_hellinger}. 

In \cite{Spantini2015}, the optimal low-rank covariance approximation problem is analysed and solved in the finite-dimensional setting, in which $\ran{\mathcal{C}_\pr}=\H$ holds and the range condition $\ran{K}\subset\ran{\mathcal{C}_\pr}$ is always satisfied. The solution in \cite{Spantini2015} is formulated in terms of the eigenpairs of the prior-preconditioned Hessian instead of the Feldman--Hajek operator $R(\mathcal{C}_\pos\Vert\mathcal{C}_\pr)$ defined in \eqref{eqn:informal_fh_operator}, but as mentioned above, these operators are transformations of each other by \Cref{prop:bayesian_feldman_hajek}. The solution given in the infinite-dimensional version of \Cref{thm:opt_covariance_and_precision} reduces in the finite-dimensional setting to the solution found by \cite{Spantini2015}. For inverse problems with an infinite-dimensional parameter, it follows that \emph{consistent} optimal posterior covariance approximations constructed via the approach of \cite{Spantini2015} are then dimension robust. Thus, besides obtaining an analysis of optimal low-rank covariance approximation of Gaussian posteriors on Hilbert spaces, our results also provide theoretical underpinning of the approximation method proposed in finite dimensions in \cite{Spantini2015}.

\subsection{Related literature}

Low-rank approximation of posterior covariance operators for linear Gaussian inverse problems posed on finite-dimensional parameter spaces is studied in \cite{Flath2011}. In particular, \cite[eq.\ (5)]{Flath2011} presents a formula for a low-rank approximation of the posterior covariance that exploits spectral decay in the Hessian of the negative log-likelihood, and \cite[eq.\ (4)]{Flath2011} indicates that the error of this low-rank approximation is related to the tail of the spectrum of the prior-preconditioned Hessian of the negative log-likelihood. 

In \cite{Spantini2015}, a precise formulation of the low-rank posterior covariance approximation problem is given and rigorously analysed, for linear Gaussian inverse problems on finite-dimensional parameter spaces. The low-rank approximation for the posterior covariance proposed in \cite{Flath2011} is shown to be an optimal solution for a family of spectral loss functions that include as special cases the Kullback--Leibler divergence and Hellinger distance between Gaussians with the same mean but different covariances. 
This approach is further developed for goal-oriented linear Gaussian inverse problems in \cite{Spantini2017}. Dimension reduction methods for linear Gaussian inverse problems using projections of the data are studied using generalised eigenvalue problems in \cite{Giraldi2018}. The Kullback--Leibler divergence and mutual information are used to quantify the error of the approximating measures. 

Dimension reduction for Bayesian inverse problems with possibly nonlinear forward models and non-Gaussian priors appears to have been first analysed in \cite{Zahm2022}. Joint dimension reduction of parameter and data is studied in \cite{Baptista2022}, for possibly nonlinear forward models and non-Gaussian priors. The results of \cite{Zahm2022} are further improved in \cite{Li2024a,Li2025}, which derived error bounds in terms of Amari $\alpha$-divergences 
\begin{equation*}
	D_{\am,\alpha}(\nu\Vert \mu)\coloneqq \frac{1}{\alpha(\alpha-1)}\biggr(\int \left(\frac{\mathrm{d}\nu}{\mathrm{d}\mu}\right)^\alpha \mathrm{d}\mu-1\biggr) ,
\end{equation*}
for probability measures $\nu$ and $\mu$ such that $\nu\ll\mu$ and $0<\alpha\leq 1$; see \cite[eq.\ (7)]{Li2024a}.
The above-cited works consider only the setting of finite-dimensional parameter spaces, and do not consider infinite-dimensional parameter spaces. While \cite{Spantini2015} provides explicit formulas for the approximation errors, \cite{Zahm2022,Baptista2022,Li2024a,Li2025} provide only error bounds. In infinite dimensions, \cite{Cui2016b} proposes a method for sampling the posterior based on the infinite-dimensional likelihood-informed subspace, and identifies the prior-preconditioned Hessian as the fundamental object. However, a rigorous treatment of optimality is not present.

In \cite{Pinski2015}, Kullback--Leibler approximation of probability measures on infinite-dimensional Polish spaces using Gaussians is studied from the calculus of variations perspective. The main results of this work concern existence of minimisers and convergence of a proposed minimisation scheme for identifying the best approximation in a class of approximating Gaussian measures. In our setting, the posterior is already Gaussian, and the approximation classes we consider differ from those in \cite{Pinski2015}.

In \cite[Section 3]{Agapiou2017}, importance sampling for linear Gaussian inverse problems posed on separable Hilbert spaces is considered. The main result is to identify two types of intrinsic dimension, such that if both dimensions are finite, then absolute continuity of the posterior with respect to the prior holds, and thus importance sampling is possible. Also \cite{Alexanderian2016} considers the setting of linear Gaussian inverse problems on separable Hilbert spaces. In this work, the aim is to analyse the Kullback-Leibler divergence from prior to posterior for optimal experimental design. The focus of our work is not to determine whether importance sampling is possible or study optimal experimental design, but rather to identify low-dimensional structure in the Bayesian prior-to-posterior update.

\subsection{Outline}

We introduce key notation in \Cref{sec:notation} below. In \Cref{sec:formulation}, we begin by recalling the infinite-dimensional formulation of the linear Gaussian inverse problem, and formulate the posterior covariance and posterior precision approximation classes that define our measure approximation problems. In \Cref{sec:equivalence_and_divergences_between_gaussian_measures} we recall the Feldman--Hajek theorem, which characterises when two Gaussian measures are equivalent, and recall expressions for the Kullback--Leibler divergence and R\'{e}nyi divergence of two equivalent Gaussians. We state the first main result of this paper, \Cref{prop:bayesian_feldman_hajek}, which identifies the generalised eigenpairs of the three operator pencils mentioned in the introduction and identifies the finite-dimensional subspace on which the posterior differs from the prior. In \Cref{sec:optimal_approximation_covariance}, we consider measure approximation problems where the posterior covariance is approximated and identify solutions in \Cref{thm:opt_covariance_and_precision} and \Cref{cor:optimal_covariance_for_amari_and_hellinger}. Auxiliary results are presented in \Cref{sec:theoretical_facts}, and proofs of the results in this work can be found in \Cref{sec:proofs_of_results}.

\subsection{Notation}
\label{sec:notation}

Let $\H$ be a separable Hilbert space over $\R$, i.e.\ a linear space endowed with an inner product $\langle \cdot,\cdot\rangle$ which induces a complete topology and norm $\norm{\cdot}$. Let $(e_i)_{i}$ be an orthonormal basis (ONB) of $\H$, where $i$ ranges over a countable index set because $\H$ is separable.
Let also $\mathcal{K}$ be a separable Hilbert space over $\R$. By $\B(\H,\mathcal{K})$, $\B_0(\H,\mathcal{K})$, and $\B_{00}(\H,\mathcal{K})$, we denote the vector spaces of linear operators with domain $\H$ and codomain $\mathcal{K}$ that are bounded, compact, and finite-rank respectively, endowed with the operator norm $\norm{\cdot}$.
We define a finite-rank operator to be an operator that is bounded and has finite-dimensional range.
By $\B_{00,r}(\H,\mathcal{K})$ we denote the set of finite-rank operators that have rank at most $r\in\N$.
This set is not a vector space since the rank is not preserved under linear combinations.
If $\mathcal{K}=\H$ then we omit the second argument in the spaces above, e.g.\ $\mathcal{B}(\H)\coloneqq\mathcal{B}(\H,\H)$. 
We write $L_1(\H)$ and $L_2(\H)$ to denote the vector spaces of trace-class and Hilbert--Schmidt operators, and $\norm{\cdot}_{L_1(\H)}$ and $\norm{\cdot}_{L_2(\H)}$ to denote their respective norms. We also equip $L_2(\H)$ with the Hilbert--Schmidt inner product $\langle\cdot,\cdot\rangle_{L_2(\H)}$.

For $T\in\B(\H,\mathcal{K})$, we denote the adjoint of $T$ by $T^\ast\in\mathcal{B}(\mathcal{K},\H)$.
The space $\B(\mathcal{H})_\R$ denotes the space of bounded operators from $\H$ to itself that are additionally self-adjoint. The spaces $\mathcal{B}_0(\H)_\R$, $\mathcal{B}_{00}(\H)_\R$, $L_1(\H)_\R$ and $L_2(\H)_\R$ and the set $\B_{00,r}(\H)_\R$ for $r\in\N$ are defined similarly.

For $T\in\B(\mathcal{H})$ we write $T\geq 0$ and $T>0$ if $T$ is nonnegative or positive respectively, i.e.\ if respectively $\langle Th,h\rangle\geq 0$ or $\langle Th,h\rangle >0$ for all $h\in\mathcal{H}\setminus\{0\}$.
If $T\in\B(\H)_\R$ is nonnegative, then $T^{1/2}$ will denote its nonnegative self-adjoint square root, i.e.\ $T^{1/2}\in \B(\mathcal{H})_\R$. Since $T^*T$ is self-adjoint and nonnegative for any $T\in\B(\H)$, we may define $\abs{T}\coloneqq (T^*T)^{1/2}$.

For $h\in\mathcal{H}$ and $k\in\mathcal{K}$, we interpret the tensor product $k\otimes h$ as a rank-1 operator in $\B(\mathcal{H},\mathcal{K})$, and this operator is $\tilde{h}\mapsto \langle h,\tilde{h}\rangle k$. For $T\in\B_0(\mathcal{H,\mathcal{K}})$, $T$ can be written in its `singular value decomposition' as a series of rank-1 operators $T=\sum_{i}^{}\sigma_i k_i\otimes h_i$ where $(\sigma_i)_i$ is nonincreasing and nonnegative and $(h_i)_i$ and $(k_i)_i$ are orthonormal sequences in $\mathcal{H}$ and $\mathcal{K}$ respectively, see also \Cref{lemma:operator_svd}.

A linear operator $T:\mathcal{H}\rightarrow\mathcal{K}$ which is not necessarily bounded is said to be unbounded. Furthermore $T$ is densely defined if its domain $\dom{T}$ is dense in $\mathcal{H}$. We also write $T:\dom{T}\subset\mathcal{H}\rightarrow\mathcal{K}$ to emphasise the domain of definition of $T$. Thus, an unbounded linear operator $T:\mathcal{H}\rightarrow\mathcal{K}$ generalises the notion of $T\in\B(\H,\mathcal{K})$ in two ways: $\dom{T}$ may be a proper subspace of $\mathcal{H}$ and $T$ need not be bounded on $\dom{T}$.
If $T:\mathcal{H}\rightarrow\mathcal{K}$, $S:\mathcal{H}\rightarrow\mathcal{K}$ and $U:\mathcal{K}\rightarrow\mathcal{Z}$ for some separable Hilbert space $\mathcal{Z}$, then $T+S:\mathcal{H}\rightarrow\mathcal{K}$ is defined on $\dom{T}\cap\dom{S}$ and $UT:\mathcal{H}\rightarrow\mathcal{Z}$ is defined on  the preimage or inverse image $T^{-1}(\dom{U})$.

Self-adjoint unbounded operators are recalled in \Cref{def:unbounded_ajoint} and \Cref{def:unbounded_self_adjoint}.

For a densely defined linear operator $S:\dom{S}\subset\mathcal{H}\rightarrow\mathcal{K}$ with domain $\dom{S}\subset\H$,
an extension $T$ of $S$ is an operator defined on $\dom{T}\subset\H$, such that $\dom{S}\subset\dom{T}$ and the restriction of $T$ to $\dom{S}$ agrees with $S$. 
We shall write $S\subset T$ to denote that $T$ is an extension of $S$. If $T$ is bounded, then $T$ is the unique extension of $S$ to all of $\mathcal{H}$, see \cite[Theorem I.7]{Reed1980}.

We let $\Lambda:L_2(\mathcal{H})_\R\rightarrow\ell^2(\R)$ be a function that sends a self-adjoint Hilbert--Schmidt operator to its square-summable eigenvalue sequence. One possible ordering labels the negative eigenvalues with the even integers and the positive eigenvalues with the odd integers, both ordered decreasingly in absolute value. A different choice is to order the eigenvalues in order of decreasing absolute value. Note that the choice of ordering of $\Lambda$ in two operators $T,S\in L_2(\mathcal{H})_\R$ is allowed to be different. The precise ordering of eigenvalues that $\Lambda$ assigns to an operator is not important, as we shall only consider compositions of $\Lambda$ with functions on $\ell^2(\R)$ that are permutation invariant.
In analogy to the eigenvalue map $\Lambda:L_2(\mathcal{H})_\R\rightarrow\ell^2(\R)$, we define $\Lambda^m:L_2(\mathcal{Z})_\R\rightarrow\R^m$ for any $m$-dimensional subspace $\mathcal{Z}\subset\mathcal{H}$ for $m\in\N$ to be the map that sends $X\in L_2(\mathcal{Z})_\R$ to its eigenvalue sequence, which is of length $m$, ordered in a nonincreasing way.

We denote equivalence of two measures $\mu$ and $\nu$ by $\mu\sim\nu$. That is, $\mu\sim\nu$ if $\mu$ and $\nu$ are absolutely continuous with respect to each other. The measure $\nu$ is absolutely continuous with respect to $\mu$ if $\mu(A)=0$ implies $\nu(A)=0$ for every measurable set $A$. We denote the support of a measure $\mu$ by $\supp{\mu}$.

We write $X\sim \mu$ to denote that the distribution of a random variable $X$ is $\mu$. If $X$ has a Gaussian distribution on $\mathcal{H}$, i.e.\ $\langle X,h\rangle$ is a one-dimensional Gaussian random variable for each $h\in\mathcal{H}$, then we write $X\sim \mathcal{N}(m,\mathcal{C})$, where $m=\mathbb{E}X$ is the mean of $X$ and $\langle \mathcal{C}h,k\rangle = \mathbb{E}\langle h,X-m\rangle\langle X-m,k\rangle$ defines the covariance $\mathcal{C}$ of $X$.
The `precision' of $\mathcal{N}(m,\mathcal{C})$ is $\mathcal{C}^{-1}$.

For $I$ a non-empty interval in $\R$, $\ell^2(I)$ denotes the space of square-summable sequences, i.e.\ $\ell^2(I)=\{(x_i)_{i\in\N}\subset I\ :\ \sum_{i\in\N} \abs{x_i}^2<\infty\}$. If $I\subset\R$ is open, then $C^1(I)$ denotes the set of continuously differentiable functions on $I$.

We write `$a\leftarrow b$' to denote the replacement of $a$ with $b$.

\section{Low-rank posterior covariance approximations}
\label{sec:formulation}

Let $\H$ be a separable Hilbert space over $\R$ of dimension $\dim{\H}\leq\infty$, which models the parameter space. Consider the observation model defined by a continuous linear forward model $G\in\mathcal{B}(\mathcal{H},\R^n)$ and additive Gaussian observation error
\begin{align}
	\label{eqn:observation_model}
	Y = Gx^\dagger + \zeta,\quad \zeta\sim\mathcal{N}(0,\mathcal{C}_{\obs}).
\end{align}
The covariance $\mathcal{C}_\obs\in \B(\R^n)_\R$ of the observation noise $\zeta$ is a positive operator and $x^\dagger\in\H$.
By the Gaussian assumption on the noise, it follows that for any fixed $x\in\mathcal{H}$, the likelihood of observing $y$ is proportional to $\exp(-\frac{1}{2}\norm{\mathcal{C}_\obs^{-1/2}(y-Gx)}^2 )$. The Hessian of the negative log-likelihood with respect to $x$ is
\begin{align}
	H = G^*\mathcal{C}_\obs^{-1}G\in\B_{00,n}(\mathcal{H})_\R.
	\label{eqn:hessian}
\end{align}
It follows from $H={G}^*\mathcal{C}_{\obs}^{-1/2}(G^*\mathcal{C}_{\obs}^{-1/2})^*$ that $H$ is self-adjoint and nonnegative.

We adopt the Bayesian perspective to the problem of inferring $x^\dagger$ given the observation $y$ of $Y$, by modeling the unknown $x^\dagger$ with an $\mathcal{H}$-valued random variable $X$. Its distribution, the prior distribution, is taken to be a Gaussian measure $\mu_{\pr}=\mathcal{N}(0,\mathcal{C}_{\pr})$ on $\mathcal{H}$ and we assume that $X$ and $\zeta$ are independent. As the covariance of a Gaussian measure on $\H$, $\mathcal{C}_\pr$ lies in $L_1(\mathcal{H})_\R$ and $\mathcal{C}_\pr\geq 0$, hence $\mathcal{C}_\pr$ has a unique nonnegative square root $\mathcal{C}_\pr^{1/2}\in L_2(\mathcal{H})_\R$. In this work, we make the following assumption.

\begin{assumption}
	We assume that the prior distribution $\mu_\pr=\mathcal{N}(0,\mathcal{C}_\pr)$ is nondegenerate on $\H$.
\end{assumption}

Nondegeneracy of $\mu_\pr$ implies that $\supp{(\mu_\pr)}=\mathcal{H}$, see e.g.\ \cite[Definition 3.6.2]{bogachev_gaussian_1998}, and that $\mathcal{C}_\pr>0$ and $\mathcal{C}_\pr^{1/2}>0$, see \Cref{lemma:covariance_properties}. In particular, $\mathcal{C}_\pr$ and $\mathcal{C}_\pr^{1/2}$ are injective by \Cref{lemma:positive_is_injective_nonnegative}. Hence the inverses $\mathcal{C}_\pr^{-1}$ and $\mathcal{C}_\pr^{-1/2}$ are well-defined bijections $\ran{\mathcal{C}}_\pr\rightarrow\mathcal{H}$ and $\ran{\mathcal{C}}_\pr^{1/2}\rightarrow\mathcal{H}$ respectively. They are unbounded, self-adjoint operators, c.f.\ \Cref{def:unbounded_ajoint} and \Cref{lemma:symmetric_operators}\ref{item:symmetric_operators_2}. The Cameron--Martin space of $\mu_{\pr}$ is the Hilbert space $(\ran{\mathcal{C}}_\pr^{1/2},\norm{\cdot}_{\mathcal{C}_\pr^{-1}})$, see e.g.\ \cite[p. 293]{bogachev_gaussian_1998}, where the Cameron--Martin norm of an element $h\in\ran{\mathcal{C}_\pr^{1/2}}$ is defined by $\norm{h}_{\mathcal{C}_\pr^{-1}}\coloneqq\norm{\mathcal{C}_\pr^{-1/2}h}$. As $\mathcal{C}_\pr$ is injective and compact, $\ran{\mathcal{C}}_\pr^{1/2}$ is dense in $\mathcal{H}$ and if $\dim{\H}=\infty$, then $\ran{\mathcal{C}}_\pr^{1/2}$ is a proper dense subspace of $\mathcal{H}$.

A common way to construct covariance operators on function spaces is to consider inverses of Laplacian-like operators, c.f.\ \cite{Stuart2010}. This approach is used in computation; see e.g.\ \cite{Bui-Thanh2013}.

Given a realisation $y$ of the random variable $Y$ defined by the observation model, the posterior distribution $\mu_\pos=\mu_\pos(y)$ of $X$ given $Y=y$ is the Gaussian measure $\mathcal{N}(m_\pos,\mathcal{C}_\pos)$, where
\begin{subequations}
	\label{eqn:update_equations}
	\begin{align}
		\label{eqn:pos_mean}
		m_\pos = m_{\pos}(y) &= \mathcal{C}_{\pos}G^*\mathcal{C}_{\obs}^{-1}y\in\ran{\mathcal{C}_\pos},\\
		\label{eqn:pos_covariance}
		\mathcal{C}_{\pos} &= \mathcal{C}_{\pr} - \mathcal{C}_{\pr} G^*(\mathcal{C}_{\obs}+G\mathcal{C}_{\pr} G^*)^{-1}G\mathcal{C}_{\pr} ,\\
		\label{eqn:pos_precision}
		\mathcal{C}_\pos^{-1} &= \mathcal{C}_\pr^{-1} + G^*\mathcal{C}_\obs^{-1}G = \mathcal{C}_\pr^{-1}+H,
	\end{align}
\end{subequations}
see e.g.\ \cite[Example 6.23]{Stuart2010}. Equation \eqref{eqn:pos_precision} should be understood to imply the following two facts: $\ran{\mathcal{C}}_\pos\coloneqq\dom{(\mathcal{C}_\pr^{-1}+H)}=\ran{\mathcal{C}}_\pr$, and $\mathcal{C}_\pr^{-1}+H:\ran{\mathcal{C}_\pr}\rightarrow\H$ is the inverse of the operator $\mathcal{C}_{\pos}$ given in \eqref{eqn:pos_covariance}.
While all nondegenerate Gaussians are equivalent in a finite-dimensional setting, this is no longer true in an infinite-dimensional setting, where in fact it holds that nondegenerate Gaussians that are not equivalent must be mutually singular; see the Feldman--Hajek theorem, recalled in \Cref{thm:feldman--hajek}. By \cite[Theorem 6.31]{Stuart2010}, $\mu_\pos$ and $\mu_\pr$ are in fact equivalent. In particular, $\mu_\pos$ is a nondegenerate measure and the above properties of $\mathcal{C}_\pr$ also hold for $\mathcal{C}_\pos$.
We shall construct approximations to $\mu_\pos$ that are equivalent to $\mu_\pos$. 

The equations in \eqref{eqn:update_equations} motivate certain Gaussian approximations of $\mu_\pos$ that, as we shall see, retain equivalence to $\mu_\pos$. By \eqref{eqn:pos_covariance}, $\mathcal{C}_\pos$ is an update of $\mathcal{C}_\pr$ by a nonpositive self-adjoint operator $ - \mathcal{C}_{\pr} G^*(\mathcal{C}_{\obs}+G\mathcal{C}_{\pr} G^*)^{-1}G\mathcal{C}_{\pr}$. The range of this update is contained in $\ran{\mathcal{C}_\pr}$ and the rank of this update is at most $n$ since $G\in\B(\mathcal{H},\R^n)$. For $r\in\N$, this motivates the rank-constrained approximation of $\mathcal{C}_\pos$ by updating $\mathcal{C}_\pr$ using nonpositive self-adjoint operators of the form $-KK^*$, for $K\in\B(\R^r,\mathcal{H})$ with $\ran{K}\subset\ran{\mathcal{C}_\pr}$ and $\mathcal{C}_\pr-KK^*>0$. That is, we consider
\begin{align}
	\label{eqn:class_approx_covariance}
	\mathscr{C}_r\coloneqq\left\{ \mathcal{C}_{\pr}-KK^*>0:\ K\in\B(\R^r,\H),\ \ran{K}\subset\ran{\mathcal{C}_\pr}\right\},\quad r\in\N.
\end{align}
Since $\mathcal{C}_\pr G^*(\mathcal{C}_\obs+G\mathcal{C}_\obs G^*)^{-1}G\mathcal{C}_\pr\in\B_{00,n}(\mathcal{H})_\R$, we have $\mathcal{C}_\pos\in\mathscr{C}_r$ for all $r\geq r_0$ by \eqref{eqn:pos_covariance}, where $r_0\coloneqq\rank{\mathcal{C}_\pr G^*(\mathcal{C}_\obs+G\mathcal{C}_\obs G^*)^{-1}G\mathcal{C}_\pr}\leq n$.

Alternatively, we can consider approximations of $\mu_\pos$ by constructing rank-constrained updates of the prior precision $\mathcal{C}_\pr^{-1}$. By \eqref{eqn:pos_precision}, $\mathcal{C}_\pos^{-1}$ is an update of $\mathcal{C}_\pr^{-1}$ by the Hessian $H$, which is self-adjoint, nonnegative, and has rank at most $n$. For $r\in\N$, we can therefore consider the class of approximations of $\mathcal{C}_\pos^{-1}$ of the form $\mathcal{C}_\pr^{-1}+UU^*$, for $U\in\B(\R^r,\mathcal{H})$. That is, we consider
\begin{align}
	\label{eqn:class_approx_precision}
	\mathscr{P}_r\coloneqq\left\{\mathcal{C}_\pr^{-1}+UU^*:\ U\in\B(\R^r,\H)\right\},\quad r\in\N.
\end{align}
Since $H\in\B_{00,n}(\mathcal{H})_\R$, $\mathcal{C}_\pos^{-1}\in\mathscr{P}_r$ for all $r\leq r_0$ with ${r_0}=\rank{H}$. The updates $\mathcal{C}_\pr^{-1}+UU^*$ in \eqref{eqn:class_approx_precision} are defined on $\ran{\mathcal{C}_\pr}$, by definition of the sum of unbounded operators, c.f.\ \Cref{sec:notation}.

We note that every operator $SS^*$ for $S\in\B(\R^r,\mathcal{H})$ is a nonnegative, self-adjoint operator with rank at most $r$, and that every nonnegative operator $T\in \B_{00,r}(\mathcal{H})_\R$ can be written in this way. Therefore, we could write the above approximations as $\mathcal{C}_\pr-T$ or $\mathcal{C}_\pr^{-1}+T$ for nonnegative $T\in\mathcal{B}_{00,r}(\mathcal{H})$, such that $\mathcal{C}_\pr-T$ is positive and maps into $\ran{\mathcal{C}_\pr}$. However, the set of nonnegative elements of $\mathcal{B}_{00,r}(\mathcal{H})$ is not convex, since rank is not preserved by convex combinations. By replacing $T$ by $SS^*$, we avoid formulating an optimisation problem over a nonconvex set. Indeed, $\mathcal{B}(\R^r,\mathcal{H})$ is not only convex but is also a Banach space.

The classes $\mathscr{C}_r$ and $\mathscr{P}_r$ are generalisations to a possibly infinite-dimensional setting of those considered in \cite{Spantini2015}.
We search for low-rank approximations of the objects in \eqref{eqn:pos_covariance} and \eqref{eqn:pos_precision}, where `low-rank' refers to the fact that we consider approximations in the classes $\mathscr{C}_r$ and $\mathscr{P}_r$, for $r<n$ respectively. \cite[Section 8]{PartII} contains two examples which can be analysed in the framework described in this section.

\section{Equivalence and Divergences between Gaussian measures}
\label{sec:equivalence_and_divergences_between_gaussian_measures}

Since our approximation problems are formulated in the context of statistical inverse problems, and since absolute continuity of the posterior with respect to the prior is important for statistical inference, we require our approximate posteriors to be equivalent to $\mu_\pos$. In \Cref{subsec:equivalence}, we recall the Feldman--Hajek theorem which gives necessary and sufficient conditions for Gaussian measures to be equivalent, and apply this theorem to the setting described in \Cref{sec:formulation}. Then, in \Cref{subsec:divergences}, we consider certain divergences between equivalent Gaussian measures, which we use to measure the approximation quality of low-rank posterior approximations.

Unless otherwise specified, the proofs of the results below are given in \Cref{subsec:proofs_for_formulation}.

\subsection{Equivalence between Gaussian measures}
\label{subsec:equivalence}

Given a fixed nondegenerate reference Gaussian measure, the set of equivalent Gaussian measures is described by the Feldman--Hajek theorem, see e.g.\ \cite[Corollary 6.4.11]{bogachev_gaussian_1998} or \cite[Theorem 2.25]{da_prato_stochastic_2014}.

\begin{theorem}[Feldman--Hajek]
	\label{thm:feldman--hajek}
	Let $\H$ be a Hilbert space and $\mu=\mathcal{N}(m_1,\mathcal{C}_1)$ and $\nu=\mathcal{N}(m_2,\mathcal{C}_2)$ be Gaussian measures on $\H$.
	Then $\mu$ and $\nu$ are either singular or equivalent, and $\mu$ and $\nu$ are equivalent if and only if the following conditions hold:
	\begin{enumerate}
		\item 
			\label{item:fh_ranges}
			$\ran{\mathcal{C}_1^{1/2}} = \ran{\mathcal{C}_2^{1/2}}$,
		\item 
			\label{item:fh_means}
			$m_1-m_2\in\ran{\mathcal{C}_1^{1/2}}$ and,
		\item 
			\label{item:fh_covariance}
			$(\mathcal{C}_1^{-1/2}\mathcal{C}_2^{1/2})(\mathcal{C}_1^{-1/2}\mathcal{C}_2^{1/2})^*-I\in L_2(\H)$.
	\end{enumerate}
\end{theorem}

The operator appearing in \Cref{thm:feldman--hajek}\ref{item:fh_covariance} quantifies the amount of similarity between Gaussian measures. If it does not have square-summable eigenvalues, then the Gaussian measures are mutually singular. In the other extreme, if the Gaussian measures are equal, then this operator is equal to 0 and the squared eigenvalues sum to 0.

\begin{remark}[Cameron--Martin norm equivalence]
	\Cref{thm:feldman--hajek} states that the Cameron--Martin spaces $\ran{\mathcal{C}_i^{1/2}}$, $i=1,2$, of the Gaussian measures $\mu$ and $\nu$ are equal as subspaces if $\mu$ and $\nu$ are equivalent, see also \cite[Proposition 2.7.3]{bogachev_gaussian_1998}. In fact, the two Cameron--Martin spaces $(\ran{\mathcal{C}_i^{1/2}},\norm{\cdot}_{\mathcal{C}_i^{-1}})$, $i=1,2$, must then have equivalent Cameron--Martin norms as well. This follows from \Cref{lemma:equivalent_norms} applied to the square root of the two covariances. This fact is mentioned without proof in \cite[Proposition B.2]{Pinski2015} and \cite[Proposition B.1]{Bolin2023}.
\end{remark}

Let us define
\begin{align}
	\label{eqn:equivalent_covariances}
	\mathcal{E}\coloneqq \{\mathcal{C}\in L_1(\mathcal{H})_\R:\ \mathcal{N}(m_\pos,\mathcal{C})\sim\mu_\pos\},
\end{align}
and more generally, for $m_1\in\H$ and $\mathcal{C}_1\in L_1(\H)_\R$ with $\mathcal{C}_1>0$,
\begin{align}
	\label{eqn:equivalent_covariances_general}
	\mathcal{E}(m_1,\mathcal{C}_1)\coloneqq \{\mathcal{C}\in L_1(\mathcal{H})_\R:\ \mathcal{N}(m_1,\mathcal{C})\sim\mathcal{N}(m_1,\mathcal{C}_1)\}.
\end{align}
That is, $\mathcal{E}$ contains those covariances $\mathcal{C}$ such that $\mathcal{N}(m_\pos,\mathcal{C})$ is equivalent to $\mu_\pos$ and $\mathcal{E}=\mathcal{E}(m_\pos,\mathcal{C}_\pos)$.
Since $\mu_\pos$ and $\mu_\pr$ are equivalent, we have $\mathcal{C}_\pr\in\mathcal{E}$.

In order to characterise the set $\mathcal{E}$ in \eqref{eqn:equivalent_covariances}, we introduce the following definition, which is closely related to \Cref{thm:feldman--hajek}\ref{item:fh_covariance}. This definition appears in \cite[Section 6.3]{bogachev_gaussian_1998}.

\begin{definition}
	\label{def:property_E}
	If $A\in\mathcal{B}(\H)$ is invertible and $AA^*-I\in L_2(\mathcal{H})$, then we say that $A$ satisfies `property E'.
\end{definition}

By \cite[Lemma 6.3.1(ii)]{bogachev_gaussian_1998} the set of operators that satisfy property E is closed under taking inverses, adjoints and compositions.
Furthermore, since $\mathcal{\mu}_\pos\sim\mu_\pr$, $\mathcal{C}_\pr^{-1/2}\mathcal{C}_\pos^{1/2}$ satisfies property E.
One can now use \Cref{thm:feldman--hajek} to describe the set $\mathcal{E}$ in \eqref{eqn:equivalent_covariances} explicitly, see \Cref{lemma:equivalent_cm_condition}:
\begin{align}
	\label{eqn:characterisation_of_equivalent_covariances}
	\begin{split}
		\mathcal{E} &= \left\{ \mathcal{C}\in L_1(\mathcal{H})_\R:\ \mathcal{C}>0,\ \mathcal{C}^{-1/2}\mathcal{C}_{\pos}^{1/2} \text{ satisfies property E} \right\} \\
		&= \left\{ \mathcal{C}\in L_1(\mathcal{H})_\R:\ \mathcal{C}>0,\ \mathcal{C}^{-1/2}\mathcal{C}_{\pr}^{1/2}\text{ satisfies property E} \right\} .
	\end{split}
\end{align}
For $\mathcal{C}_1,\mathcal{C}_2\in\mathcal{E}$, we now define
\begin{align}
	\label{eqn:feldman_hajek_operator}
	R(\mathcal{C}_2\Vert \mathcal{C}_1) \coloneqq \mathcal{C}_1^{-1/2}\mathcal{C}_2^{1/2}(\mathcal{C}_1^{-1/2}\mathcal{C}_2^{1/2})^* - I.
\end{align}
By \Cref{thm:feldman--hajek}\ref{item:fh_covariance}, $R(\mathcal{C}_2\Vert \mathcal{C}_1)\in L_2(\mathcal{H})$. Since $R(\mathcal{C}_2\Vert \mathcal{C}_1)$ is a self-adjoint compact operator, there exists an ONB of $\mathcal{H}$ that diagonalises $R(\mathcal{C}_2\Vert \mathcal{C}_1)$, see \Cref{lemma:operator_svd}. We note that $R(\cdot\Vert\cdot)$ is in general not symmetric in its arguments. The result below will be used frequently in our analysis of low-rank approximations of the posterior covariance operator.

\begin{restatable}{lemma}{expansionsFeldmanHajek}
	\label{lemma:expansions_feldman_hajek}
	Let $\mathcal{C}_1,\mathcal{C}_2$ be injective covariances of equivalent Gaussian measures.
	Then there exists a sequence $(\lambda_i)_i\in\ell^2( (-1,\infty))$ and ONBs $(w_i)_i$ and $(v_i)_i$ of $\H$ such that $v_i = \sqrt{1+\lambda_i}\mathcal{C}_2^{-1/2}\mathcal{C}_1^{1/2}w_i$ and the following statements hold:
	\begin{enumerate}
		\item
			\label{item:covariance_mix_1}
			$\mathcal{C}_1^{-1/2}\mathcal{C}_2\mathcal{C}_1^{-1/2} - I\subset (\mathcal{C}_1^{-1/2}\mathcal{C}_2^{1/2})(\mathcal{C}_1^{-1/2}\mathcal{C}_2^{1/2})^*-I= R(\mathcal{C}_2\Vert \mathcal{C}_1) = \sum_{i=1}^{} \lambda_i w_i\otimes w_i\in L_2(\H)$,
		\item
			\label{item:covariance_mix_2}
			$\mathcal{C}_2^{1/2}\mathcal{C}_1^{-1}\mathcal{C}_2^{1/2} - I \subset (\mathcal{C}_1^{-1/2}\mathcal{C}_2^{1/2})^*(\mathcal{C}_1^{-1/2}\mathcal{C}_2^{1/2})-I = \sum_{i}^{}\lambda_i v_i\otimes v_i\in L_2(\H)$,
		\item
			\label{item:covariance_mix_3}
			$\mathcal{C}_2^{-1/2}\mathcal{C}_1^{}\mathcal{C}_2^{-1/2} - I \subset (\mathcal{C}_2^{-1/2}\mathcal{C}_1^{1/2})(\mathcal{C}_2^{-1/2}\mathcal{C}_1^{1/2})^* - I = R(\mathcal{C}_1\Vert\mathcal{C}_2) =\sum_{i}^{}\frac{-\lambda_i}{1+\lambda_i}v_i \otimes v_i\in L_2(\H)$,
		\item
			\label{item:covariance_mix_4}
			$\mathcal{C}_1^{1/2}\mathcal{C}_2^{-1}\mathcal{C}_1^{1/2}-I \subset (\mathcal{C}_2^{-1/2}\mathcal{C}_1^{1/2})^*(\mathcal{C}_2^{-1/2}\mathcal{C}_1^{1/2})-I= \sum_{i}^{}\frac{-\lambda_i}{1+\lambda_i}w_i\otimes w_i\in L_2(\H)$,
	\end{enumerate}
	where the domains of the leftmost operators in each statement are dense and in \cref{item:covariance_mix_1,item:covariance_mix_3} contain $\ran{\mathcal{C}_1^{1/2}=\ran{\mathcal{C}_2^{1/2}}}$. 
\end{restatable}

If $\mathcal{C}_1$ and $\mathcal{C}_2$ are as given in \Cref{lemma:expansions_feldman_hajek}, then the operator $\mathcal{C}_2^{-1/2}\mathcal{C}_1^{1/2}$ is invertible, by \Cref{thm:feldman--hajek} and \Cref{lemma:equivalent_cm_condition}. Furthermore, the map $\lambda\mapsto\frac{-\lambda}{1+\lambda}$ is a bijection on $(-1,\infty)$. Thus, each of the pairs $(\lambda_i,w_i)$, $(\lambda_i,v_i)$, $(\frac{-\lambda_i}{1+\lambda_i},v_i)$ and $(\frac{-\lambda_i}{1+\lambda_i},w_i)$ determines the other three. Hence, \Cref{lemma:expansions_feldman_hajek} shows that the operator in \Cref{thm:feldman--hajek}\ref{item:fh_covariance} can equivalently be described by the operators in \cref{item:covariance_mix_2,item:covariance_mix_3,item:covariance_mix_4}, which thus all contain the same information.
As explained in \Cref{rmk:operator_pencils}, finding the eigenpairs of the operators in \Cref{lemma:expansions_feldman_hajek} can be seen as a generalisation of finding the generalised eigenpairs of certain covariance operator pencils, which we first define below.

\begin{definition}
	For possibly unbounded operators $T,S:\mathcal{H}\rightarrow\mathcal{H}$, the operator pencil $(T,S)$ is defined by the collection of operators $\{T-\lambda S,\ \lambda\in\R\}$. A `generalised eigenvalue' of $(T,S)$ is a value $\lambda\in\R$ for which $T-\lambda S$ is not injective. For such $\lambda$ there exists a nonzero $v\in\dom{T}\cap\dom{S}$ such that $Tv=\lambda Sv$, which is called a `generalised eigenvector', and we say that $(\lambda,v)$ is a `generalised eigenpair' of $(T,S)$.
	\label{def:operator_pencils}
\end{definition}

\begin{remark}[Generalised eigenpairs]
	If $w_i\in\dom{\mathcal{C}_2^{1/2}\mathcal{C}_1^{-1/2}}=\ran{\mathcal{C}_1^{1/2}}$ for some $i$, then the statement of \Cref{lemma:expansions_feldman_hajek}\ref{item:covariance_mix_1} implies $\mathcal{C}_2 \mathcal{C}_1^{-1/2}w_i = (1+\lambda_i)\mathcal{C}_1^{1/2}w_i$. In other words, $\mathcal{C}_2(\mathcal{C}_1^{-1/2}w_i)=(1+\lambda_i)\mathcal{C}_1(\mathcal{C}_1^{-1/2}w_i)$, showing that $(1+\lambda_i,\mathcal{C}_1^{-1/2}w_i)$ is a generalised eigenpair of the generalised operator pencil $(\mathcal{C}_2,\mathcal{C}_1)$. Furthermore, if $(v_i)_i$ lies in the dense subspace $\dom{\mathcal{C}_2^{1/2}\mathcal{C}_1^{-1}\mathcal{C}_2^{1/2}} = \dom{\mathcal{C}_1^{-1}\mathcal{C}_2^{1/2}}$, then for any $i$ we have $\mathcal{C}_2^{1/2}v_i\in\dom{\mathcal{C}_1^{-1}}$. The relation in \Cref{lemma:expansions_feldman_hajek}\ref{item:covariance_mix_2} shows that $\mathcal{C}_2^{1/2}\mathcal{C}_1^{-1}\mathcal{C}_2^{1/2}v_i=(1+\lambda_i)v_i$, so that $v_i\in\ran{\mathcal{C}_2^{1/2}}$. 
	Hence, $\mathcal{C}_2^{1/2}v_i\in\dom{\mathcal{C}_1^{-1}}\cap\dom{\mathcal{C}_2^{-1}}$. The previous relation implies $\mathcal{C}_1^{-1}\mathcal{C}_2^{1/2}v_i = (1+\lambda_i)\mathcal{C}_2^{-1}\mathcal{C}_2^{1/2}v_i$, showing that $( 1+\lambda_i,\mathcal{C}_2^{1/2}v_i)=(1+\lambda_i,\sqrt{1+\lambda_i}\mathcal{C}_1^{1/2}w_i)$ is a generalised eigenpair of $(\mathcal{C}_1^{-1},\mathcal{C}_2^{-1})$. Thus, in the case $(w_i)_i$ and $(v_i)_i$ lie in suitable dense subspaces of $\H$, \Cref{lemma:expansions_feldman_hajek}\ref{item:covariance_mix_1},\ref{item:covariance_mix_2},\ref{item:covariance_mix_3},\ref{item:covariance_mix_4} can be interpreted as statements on generalised eigenpairs of operator pencils, such as $(\mathcal{C}_2,\mathcal{C}_1)$ or $(\mathcal{C}_1^{-1},\mathcal{C}_2^{-1})$. On the contrary, the diagonalisation of the operators in \Cref{lemma:expansions_feldman_hajek} does not require such assumptions on the eigenvectors. Thus, the Feldman--Hajek operator $R(\mathcal{C}_2\Vert \mathcal{C}_1)$ in \Cref{lemma:expansions_feldman_hajek}\ref{item:covariance_mix_1} and its reformulations in \Cref{lemma:expansions_feldman_hajek}\ref{item:covariance_mix_2},\ref{item:covariance_mix_3},\ref{item:covariance_mix_4} are the canonical objects to diagonalise when $\dim{\H}=\infty$, rather than covariance operator pencils.
	\label{rmk:operator_pencils}
\end{remark}

\Cref{thm:feldman--hajek} and \Cref{lemma:expansions_feldman_hajek} hold for any equivalent Gaussian measures. In the specific case of the linear Bayesian inverse problem \eqref{eqn:observation_model}, in which case the posterior precision is a finite-rank update $H$ of the prior by \eqref{eqn:pos_precision}, more can be said about the eigenvectors and eigenvalues given by \Cref{lemma:expansions_feldman_hajek} of the operators $R(\mathcal{C}_\pr\Vert \mathcal{C}_\pos)$ and $R(\mathcal{C}_\pos\Vert\mathcal{C}_\pr)$. We remind the reader of the definition of the Hessian ${H}$ in \eqref{eqn:hessian}.

\begin{restatable}{proposition}{bayesianFeldmanHajek}
	\label{prop:bayesian_feldman_hajek}
	There exists a nondecreasing sequence $(\lambda_i)_i\in\ell^2( (-1,0] )$ consisting of exactly $\rank{H}$ nonzero elements and ONBs $(w_i)_i$ and $(v_i)_i$ of $\mathcal{H}$ such that $w_i,v_i\in\ran{\mathcal{C}_\pr^{1/2}}$ and $v_i = \sqrt{1+\lambda_i}\mathcal{C}_\pos^{-1/2}\mathcal{C}_\pr^{1/2}w_i$ for every $i\in \N$, and
	\begin{subequations}
		\begin{align}
			R(\mathcal{C}_\pos\Vert \mathcal{C}_\pr) &= \sum_{i}^{}\lambda_i w_i\otimes w_i, \nonumber \\
			\label{eqn:prior_preconditioned_Hessian}
			\mathcal{C}_{\pr}^{1/2}H\mathcal{C}_{\pr}^{1/2}
			&= (\mathcal{C}_\pos^{-1/2}\mathcal{C}_\pr^{1/2})^* (\mathcal{C}_\pos^{-1/2}\mathcal{C}_\pr^{1/2}) - I
			=\sum_{i}\frac{-\lambda_i}{1+\lambda_i}w_i\otimes w_i, \\
			\label{eqn:posterior_preconditioned_Hessian}
			\mathcal{C}_{\pos}^{1/2}H\mathcal{C}_{\pos}^{1/2}	
			&= I-(\mathcal{C}_\pr^{-1/2}\mathcal{C}_\pos^{1/2})^* (\mathcal{C}_\pr^{-1/2}\mathcal{C}_\pos^{1/2})
			=\sum_{i}(-\lambda_i) v_i\otimes v_i,\\
			\label{eqn:bayesian_cov_pencil}
			\mathcal{C}_\pos^{1/2}\mathcal{C}_\pr^{-1/2}w_i 
			&= (1+\lambda_i)\mathcal{C}_\pos^{-1/2}\mathcal{C}_\pr^{1/2}w_i,\quad \forall i\in\N.
		\end{align}
	\end{subequations}
\end{restatable}

In \Cref{prop:bayesian_feldman_hajek}, $w_i,v_i\in\ran{\mathcal{C}_\pr^{1/2}}$ for all $i$, so that $v_i\in\dom{\mathcal{C}_\pos^{1/2}\mathcal{C}_\pr^{-1}\mathcal{C}_\pos^{1/2}}$ and $w_i\in\dom{\mathcal{C}_\pr^{1/2}\mathcal{C}_\pos^{-1}\mathcal{C}_\pr^{1/2}}$,  because $\ran{\mathcal{C}_\pr}=\ran{\mathcal{C}_\pos}$. The equations \eqref{eqn:prior_preconditioned_Hessian} and \eqref{eqn:posterior_preconditioned_Hessian} can be interpreted as statements on operator pencils by \Cref{rmk:operator_pencils}.
More specifically, \eqref{eqn:prior_preconditioned_Hessian} states that $(\frac{-\lambda_i}{1+\lambda_i},\mathcal{C}_\pr^{1/2}w_i)$ is a generalised eigenpair of $(H,\mathcal{C}_\pr^{-1})$ and \eqref{eqn:posterior_preconditioned_Hessian} states that $(-\lambda_i,\mathcal{C}_\pos^{1/2}v_i)=(-\lambda_i,\sqrt{1+\lambda_i}\mathcal{C}_\pr^{-1/2}w_i)$ is a generalised eigenpair of $(H,\mathcal{C}_\pos^{-1})$, for any $i$. Furthermore, \eqref{eqn:bayesian_cov_pencil} can be interpreted as a statement on the operator pencils $(\mathcal{C}_\pos,\mathcal{C}_\pr)$ and $(\mathcal{C}_\pr^{-1},\mathcal{C}_\pos^{-1})$. The prior-preconditioned Hessian $\mathcal{C}_\pr^{1/2}H\mathcal{C}_\pr^{1/2}$ has been found to be the central object of study in the reduction of finite-dimensional linear Gaussian inverse problems, see \cite{Spantini2015,Cui2014}. We observe that this operator is directly related to $R(\mathcal{C}_\pos\Vert\mathcal{C}_\pr)$ via the equivalent characterisations given by \Cref{lemma:expansions_feldman_hajek}\ref{item:covariance_mix_1},\ref{item:covariance_mix_2}\ref{item:covariance_mix_3},\ref{item:covariance_mix_4} and hence to the function $R(\cdot\Vert\cdot)$ which quantifies the similarity of Gaussian measures by \Cref{thm:feldman--hajek}\ref{item:fh_covariance} and definition \eqref{eqn:feldman_hajek_operator}.

\subsection{Divergences between equivalent Gaussian measures}
\label{subsec:divergences}
To measure the quality of an approximation $\tilde{\mathcal{C}}\in\mathcal{E}$ of $\mathcal{C}_\pos$, we shall use the R\'enyi divergences of order $\rho\in(0,1)$ and the forward and reverse Kullback--Leibler (KL) divergences. The KL divergence from a measure $\nu_1$ to a measure $\nu_2$ equivalent to $\nu_1$ is defined as
\begin{align*}
	D_{\kl}(\nu_2\Vert\nu_1) = \int_\mathcal{H}\log{\frac{\d \nu_2}{\d \nu_1}}\d\nu_2.
\end{align*}
If $\nu_2$ is a given measure that needs to be approximated and $\nu_1$ is an approximation of $\nu_2$, then we refer to $D_{\kl}(\nu_2 \Vert \nu_1)$ and to $D_{\kl}(\nu_1 \Vert \nu_2)$ as the `forward' and `reverse' KL divergence of the approximation respectively. The R\'enyi divergence of order $\rho\in(0,1)$ is defined by
\begin{align*}
	D_{\ren,\rho}(\nu_2\Vert\nu_1) = -\frac{1}{\rho(1-\rho)}\log{\int_\mathcal{H}\left(\frac{\d\nu_2}{\d\nu_1}\right)^\rho\d\nu_1},
\end{align*}
c.f.\ \cite[eq.\ (130)]{minh_regularized_2021}.
It holds that $D_{\ren,\rho}(\nu_1\Vert \nu_2) = D_{\ren,1-\rho}(\nu_2\Vert \nu_1)$, because
\begin{align*}
	\int_{\mathcal{H}}^{}\left( \frac{\d\nu_1}{\d\nu_2}\right)^\rho\d\nu_2 = \int_{\mathcal{H}}^{}\left(\frac{\d\nu_2}{\d\nu_1}\right)^{1-\rho} \left( \frac{\d\nu_1}{\d\nu_2} \right)^{1-\rho}\left( \frac{\d\nu_1}{\d\nu_2} \right)^\rho\d\nu_2 = \int_{\mathcal{H}}^{}\left(\frac{\d\nu_2}{\d\nu_1}\right)^{1-\rho} \d\nu_1.
\end{align*}
This is known as the `skew symmetry' of the R\'enyi divergence, c.f.\ \cite[Proposition 2]{vanErven2014}. Consequently, there is no need to consider forward R\'enyi divergences $D_{\ren,\rho}(\nu_2\Vert\nu_1)$ and reverse R\'enyi divergences $D_{\ren,\rho}(\nu_1\Vert\nu_2)$ separately.

In the Gaussian case, an explicit representation of these divergences holds, as shown in \cite{minh_regularized_2021}. For this, we need a generalisation of the determinant to infinite-dimensional Hilbert spaces. Because in infinite dimensions the eigenvalues of a compact operator accumulate at 0, direct extension of the finite-dimensional definition of the determinant as the product of the eigenvalues to infinite dimensions will result in the determinant function being equal to the constant $0$. A generalisation of the concept of the determinant for trace-class and Hilbert--Schmidt operators is given by the Fredholm determinant and Hilbert--Carleman determinant respectively. These are defined on respectively trace-class and Hilbert--Schmidt perturbations of the identity, and are indicated by $\det(I+A)$, $A\in L_1(\mathcal{H})$, and respectively $\det_2{(I+A)}$, $A\in L_2(\mathcal{H})$. We refer to \cite[Theorem 3.2, Theorem 6.2]{simon_notes_1977} or \cite[Lemma 3.3, Theorem 9.2]{Simon2005}. For $A\in L_1(\mathcal{H})$, we have $\det_2{(I+A)}=\det(I+A)\exp(-\tr{A})$, and for $A\in L_2(\mathcal{H})$ the determinants are related via $\det_2( I+A ) = \det( I + (I+A)\exp(-A) ).$
By \cite[Theorem 4.2, Theorem 6.2]{simon_notes_1977} or \cite[Theorem 3.7, Theorem 9.2]{Simon2005} for each $\mu\in\R$, we have the expression
\begin{align}
	\label{eqn:determinant_for_L_1}
	\det( I+\mu A ) &= \prod_i(1+\mu \lambda_i),\quad A\in L_1(\mathcal{H})\\
	\label{eqn:determinant_for_L_2}
	\det\nolimits_{2}( I+\mu A ) &= \prod_{i}(1+\mu\lambda_i)\exp{(-\mu\lambda_i)}, \quad A\in L_2(\mathcal{H}),
\end{align}
where $(\lambda_i)_i$ denotes the eigenvalue sequence of $A$.
In the case that $\dim{\mathcal{H}}<\infty$ we note that $A-I\in L_1(\mathcal{H})$ and $\det(A)=\det(I+(A-I)) = \prod_i\lambda_i$ and thus $\det(\cdot)$ indeed extends the finite-dimensional definition of the determinant. We can now formulate the explicit expression of the KL and R\'enyi divergences for equivalent Gaussian measures.
The result below holds when $\H$ is a separable Hilbert space of finite or infinite dimension.

\begin{restatable}{theorem}{gaussianDivergences}
	\label{thm:gaussian_divergences}
	Let $m_1,m_2\in\mathcal{H}$ and $\mathcal{C}_1,\mathcal{C}_2\in L_2(\mathcal{H})_\R$ be positive. If $m_1-m_2\in\ran{\mathcal{C}_1^{1/2}}$ and if $\mathcal{C}_1^{-1/2}\mathcal{C}_2^{1/2}$ satisfies property E, then
	\begin{subequations}
		\label{eqn:divergences}
		\begin{align}
			\label{eqn:kullback_leibler_divergence}
			D_{\kl}(\mathcal{N}(m_2,\mathcal{C}_2)\Vert \mathcal{N}(m_1,\mathcal{C}_1)) &= \frac{1}{2}\Norm{\mathcal{C}_1^{-1/2}(m_2-m_1)}^2-\frac{1}{2}\log\det\nolimits_{2}(I+R(\mathcal{C}_2\Vert\mathcal{C}_1)),\\
			\label{eqn:renyi_divergence_order_rho}
			\begin{split}
				D_{\ren,\rho}(\mathcal{N}(m_2,\mathcal{C}_2)\Vert \mathcal{N}(m_1,\mathcal{C}_1)) &= \frac{1}{2}\Norm{\bigr(\rho I+(1-\rho)(I+R(\mathcal{C}_2\Vert\mathcal{C}_1))\bigr)^{-1/2}\mathcal{C}_1^{-1/2}(m_2-m_1)}^2\\
				&+\frac{\log\det\left[\bigl(I+R(\mathcal{C}_2\Vert\mathcal{C}_1)\bigr)^{\rho-1}\bigl(\rho I+(1-\rho)(I+R(\mathcal{C}_2\Vert\mathcal{C}_1))\bigr)\right]}{2\rho(1-\rho)}.
			\end{split}
		\end{align}
	\end{subequations}
	Furthermore, 
	\begin{align*}
		\lim_{\rho\rightarrow 1} D_{\ren,\rho}(\mathcal{N}(m_2,\mathcal{C}_2)\Vert \mathcal{N}(m_1,\mathcal{C}_1))= D_{\kl}(\mathcal{N}(m_2,\mathcal{C}_2)\Vert \mathcal{N}(m_1,\mathcal{C}_1)),\\
		\lim_{\rho\rightarrow 0} D_{\ren,\rho}(\mathcal{N}(m_2,\mathcal{C}_2)\Vert \mathcal{N}(m_1,\mathcal{C}_1))= D_{\kl}(\mathcal{N}(m_1,\mathcal{C}_1)\Vert \mathcal{N}(m_2,\mathcal{C}_2)).
	\end{align*}
\end{restatable}
The limits above show that for Gaussian measures, the R\'enyi divergence interpolates the forward KL, obtained in the limit $\rho\uparrow 1$, and the reverse KL, obtained in the limit $\rho\downarrow 0$.

\begin{remark}[Amari $\alpha$-divergences and R\'{e}nyi divergences]
\label{rmk:correspondence_Amari_divergence_Renyi_divergence}
For $\alpha\geq 0$, the Amari $\alpha$-divergences form another family of divergences which interpolates the forward KL at $\alpha=1$ and reverse KL at $\alpha=0$, c.f.\ \cite[eq.\ (7)]{Li2024a}. For $\alpha\in(0,1)$ and $\alpha>1$, the Amari $\alpha$-divergence $D_{\am,\alpha}(\nu_2\Vert \nu_1)$ for equivalent measures $\nu_1$ and $\nu_2$ on $\mathcal{H}$ is defined by
\begin{align*}
	D_{\am,\alpha}(\nu_2\Vert\nu_1)\coloneqq\frac{-1}{\alpha(1-\alpha)}\left( \int_{\mathcal{H}}^{}\left( \frac{\d\nu_2}{\d\nu_1} \right)^\alpha\d\nu_1-1 \right),
\end{align*}
and for $\alpha\in(0,1)$ it is related to the $\rho$-R\'{e}nyi divergence in \eqref{eqn:renyi_divergence_order_rho} with $\rho\leftarrow\alpha$ by
\begin{equation*}
	D_{\ren,\alpha}(\nu_2\Vert \nu_1)=\frac{-1}{\alpha(1-\alpha)}\log[1-\alpha(1-\alpha)D_{\am,\alpha}(\nu_2\Vert \nu_1)],
\end{equation*}
that is,
\begin{align}
	\label{eqn:amari_renyi_relation}
	D_{\am,\alpha}(\nu_2\Vert\nu_1) = \frac{-1}{\alpha(1-\alpha)}\left(\exp[-\alpha(1-\alpha)D_{\ren,\alpha}(\nu_2\Vert\nu_1)]-1\right).
\end{align}
Since $\nu_1$ and $\nu_2$ are equivalent and $\alpha\in(0,1)$, $(\frac{\d\nu_2}{\d\nu_1})^\alpha>0$ with $\nu_1$-measure 1 and hence $1-\alpha(1-\alpha)D_{\am,\alpha}(\nu_2\Vert \nu_1)=\int(\tfrac{\mathrm{d}\nu_2}{\mathrm{d}\nu_1})^\alpha\mathrm{d}\nu_1$ is strictly positive. It follows that $D_{\ren,\alpha}(\nu_2\Vert \nu_1)$ is a strictly increasing function of $D_{\am,\alpha}(\nu_2\Vert \nu_1)$. Thus, for every $0<\alpha<1$, minimising the $\alpha$-R\'{e}nyi divergence corresponds to minimising the Amari $\alpha$-divergence, and vice versa.
\end{remark}

\begin{remark}[Hellinger distance]
	\label{rmk:hellinger_divergence}
	Let us denote the Hellinger distance between equivalent measures $\nu_1$ and $\nu_2$ on $\mathcal{H}$ by $D_\hel(\nu_1,\nu_2)$, i.e.\ 
	\begin{align*}
		D_\hel(\nu_2,\nu_1)^2 \coloneqq \int_{\mathcal{H}}^{}\left( 1-\sqrt{\frac{\d\nu_2}{\d\nu_1}} \right)^2\d\nu_1 = 2-2\int_{\mathcal{H}}^{}\sqrt{\frac{\d\nu_2}{\d\nu_1}}\d\nu_1.
	\end{align*}
	It holds that 
	\begin{align}
		\label{eqn:hellinger_renyi_relation}
		D_\hel(\nu_2,\nu_1)^2=2\left(1-\exp\left(-\frac{1}{4}D_{\ren,1/2}(\nu_2\Vert\nu_1)\right)\right),
	\end{align}
	by e.g.\ \cite[eqs.\ (134)--(135)]{minh_regularized_2021}, and it follows that minimising the Hellinger distance $D_\hel(\nu_2,\nu_1)$ is equivalent to minimising the Bhattacharyya distance $\frac{1}{4}D_{\ren,1/2}(\nu_2\Vert\nu_1)$ and vice versa.
\end{remark}

\section{Optimal approximations of covariance operators}
\label{sec:optimal_approximation_covariance}

In this section, we formulate a minimisation problem that aims at finding low-rank approximations of $\mathcal{C}_\pos$ that are optimal simultaneously with respect to all members of a class of spectral loss functions. This class includes the R\'enyi divergences and forward and reverse KL divergences as special cases. The loss class and the low-rank covariance approximation problems are introduced in \Cref{subsec:spectral_loss_functions_and_problem_formulation}, the equivalence to the exact posterior of the approximations considered in \Cref{sec:formulation} is studied in \Cref{subsec:equivalence_to_the_exact_posterior}, the approximation problems are formulated as minimisation problems involving a differentiable function in \Cref{subsec:differentiability_and_minimisers_of_covariance_loss_functions}, and the approximation problems are solved in \Cref{subsec:optimal_low_rank_posterior_covariance_approximations}. A comparison with the finite-dimensional results of \cite{Spantini2015} is provided in \Cref{sec:comparison_with_finite_dimensional_setting}.
The proofs of all the results in this section are given in \Cref{subsec:proofs_for_optimal_approximation_covariance}.

\subsection{Spectral loss functions and problem formulation}
\label{subsec:spectral_loss_functions_and_problem_formulation}

To measure the quality of a given approximation of the exact posterior covariance $\mathcal{C}_\pos$, we define a class of loss functions on $\mathcal{E}\times\mathcal{E}$ in the following way. Recall the definition of the eigenvalue map $\Lambda$ defined on Hilbert--Schmidt operators, from \Cref{sec:notation}. Also recall the definition of the Hilbert--Schmidt operator-valued map $R(\cdot\Vert\cdot)$ from \eqref{eqn:feldman_hajek_operator}. Define
\begin{subequations}
	\label{eqn:definitions_losses}
	\begin{align}
		\label{eqn:definition_spectral_f}
		\mathscr{F} &\coloneqq  \bigg\{f\in C^1( (-1,\infty) ):\ f(0)=0,\ xf'(x)>0 \text{ for } x\not=0,  \\
		\nonumber
	&\qquad\lim_{x\rightarrow\infty}f(x)=\infty,\ f'\text{ Lipschitz at } 0 \bigg\},\\
		\label{eqn:definition_loss_class}
		\mathscr{L} &\coloneqq \left\{ \mathcal{E}\times \mathcal{E}\ni (\mathcal{C}_2,\mathcal{C}_1)\mapsto \mathcal{L}_f(\mathcal{C}_2\Vert \mathcal{C}_1)\coloneqq \sum_{i}^{}f(\Lambda_i(R(\mathcal{C}_2\Vert \mathcal{C}_1))) :\ f\in\mathscr{F}\right\}.
	\end{align}
\end{subequations}
As we show in \Cref{lemma:finite_loss} below, the conditions $f\in C^1( (-1,\infty))$ and $xf'(x)>0$ for $x\not=0$ ensure that $0$ is the unique minimiser of every $f\in\mathscr{F}$. The Lipschitz continuity of $f'$ at $0$ implies that $(f(x_i))_i$ is summable for $(x_i)_i\in\ell^2((-1,\infty))$, so that every $\mathcal{L}\in\mathscr{L}$ satisfies $\mathcal{L}(\cdot\Vert\cdot)<\infty$. Furthermore, this Lipschitz continuity implies that $\mathcal{L}(\mathcal{C}_\pos\Vert\cdot)$ is differentiable on a suitable subspace of $\mathcal{E}$, as will be shown later in \Cref{lemma:differentiability_of_decomposition}. The blowup at infinity condition is used to prove coercivity of $\mathcal{L}(\mathcal{C}_\pos\Vert\cdot)$ on suitable subspaces of $\mathcal{E}$, as we show in \Cref{lemma:coercivity}. 

\begin{restatable}{lemma}{finiteLoss}
	\label{lemma:finite_loss}
	Let $\mathscr{F}$ be given in \eqref{eqn:definition_spectral_f} and $f\in\mathscr{F}$. Then
	\begin{enumerate}
		\item 
			\label{item:properties_spectral_f}
			$f'(x)=0$ if and only if $x=0$, the image of $f$ lies in $[0,\infty)$ and for every $x\in\ell^2( (-1,\infty))$ it holds that $\sum_{i}^{}f(x_i)<\infty$. In particular, the image of every $\mathcal{L}_f\in\mathscr{L}$ lies in $[0,\infty)$.
		\item 
			\label{item:closed_under_transform}
			Let $\eta:(-1,\infty)\rightarrow (-1,\infty)$ be defined by $\eta(x)=\frac{-x}{1+x}$.	If $f\in\mathscr{F}$ and $\lim_{x\rightarrow -1} f(x)=\infty$, then $f\circ \eta\in\mathscr{F}$.
	\end{enumerate}
\end{restatable}

\begin{remark}[Choice of spectral loss functions]
	\label{rmk:difference_spantini_spectral_functions}
	The class of loss functions considered in the finite-dimensional setting of \cite[Definition 2.1]{Spantini2015} differs from the class \eqref{eqn:definitions_losses} in two aspects. First, for every function $f$ in the former, the domain is $(0,\infty)$ and $f$ need not have minimum equal to 0,
while for every function in the latter, the domain is $(-1,\infty)$ and we require $f(0)=0$. To motivate this difference with \cite[Definition 2.1]{Spantini2015}, we note that the fundamental object governing the losses is given by the operator $R(\mathcal{C}_2\Vert\mathcal{C}_1)$ defined in \eqref{eqn:feldman_hajek_operator}. This operator is compact and therefore has an eigenvalue sequence accumulating at 0. The horizontal shift of $-1$ and the vertical shift, which together ensure that $f$ has minimum $0$ at $x=0$, are therefore natural. Second, there is an additional Lipschitz condition in \eqref{eqn:definition_spectral_f}, which implies that $\mathcal{L}(\cdot\Vert\cdot)<\infty$ on $\mathcal{E}\times\mathcal{E}$ for every $\mathcal{L}\in\mathscr{L}$. 
\end{remark}

Note that the condition that $f'$ is Lipschitz continuous at $0$ is not implied by the other conditions in \eqref{eqn:definition_spectral_f}, as the function $f$ with $f'(x)=\operatorname{sgn}(x)\abs{x}^\alpha$, $\alpha\in(0,1)$, and $f(0)=0$ shows. Here, $\operatorname{sgn}(x)$ denotes the function assigning $1$ to $x\geq 0$ and $-1$ otherwise. This function satisfies all conditions of \eqref{eqn:definition_spectral_f} except the Lipschitz condition of $f'$ at 0.

While restricted compared to the class in \cite[Definition 2.1]{Spantini2015}, the class \eqref{eqn:definitions_losses} is still rich enough to include the forward and reverse KL divergences and R\'enyi divergences between equivalent Gaussian measures with the same mean, as shown in the following result. This result partially extends \cite[Lemma 2.2]{Spantini2015}, in which the analogous statement is shown for the forward KL divergence in the finite-dimensional setting. 

\begin{restatable}{lemma}{divergencesInLossClass}
	Let $m\in\mathcal{H}$. Let $\mu_i=\mathcal{N}(m,\mathcal{C}_i)$ be nondegenerate and $\mathcal{C}_i\in\mathcal{E}$ for $i=1,2$.
	\begin{enumerate}
	\label{lemma:divergences_in_loss_class}
		\item 
			\label{item:f_for_kl_divergence}
			Let $f_{\kl}(x)\coloneqq \frac{1}{2}(x-\log(1+x))$. Then $f_{\kl}\in\mathscr{F}$ and 
			\begin{align*}
				D_{\kl}(\mu_2\Vert\mu_1) = -\frac{1}{2}\log\det\nolimits_{2}(I+R(\mathcal{C}_2\Vert \mathcal{C}_1)) = \mathcal{L}_{f_\kl}(\mathcal{C}_2\Vert \mathcal{C}_1).
			\end{align*}
		\item
			\label{item:f_for_renyi_divergence}
			Let $\rho\in(0,1)$ and $f_{\ren,\rho}(x)\coloneqq \frac{\rho-1}{2\rho(1-\rho)}\log(1+x)+\frac{1}{2\rho(1-\rho)}\log{(\rho+(1-\rho)(1+x))}.$ Then $f_{\ren,\rho}\in\mathscr{F}$ and
			\begin{align*}
				D_{\ren,\rho}(\mu_2\Vert\mu_1) &=  \frac{\log\det\left[ \bigl(I+R(\mathcal{C}_2\Vert \mathcal{C}_1)\bigr)^{\rho-1}\bigl(\rho I + (1-\rho)(I+R(\mathcal{C}_2\Vert \mathcal{C}_1))\bigr)\right]}{2\rho(1-\rho)} \\
				&=  \mathcal{L}_{f_{\ren,\rho}}(\mathcal{C}_2\Vert \mathcal{C}_1).
			\end{align*}
		\item
			\label{item:f_for_reverse_divergence}
			For the reverse divergences, we have $f_{\kl}\circ \eta,f_{\ren,\rho}\circ \eta\in\mathscr{F}$ with $\eta(x)\coloneqq\frac{-x}{1+x}$ on $(-1,\infty)$, and
			\begin{align*}
				D_{\kl}(\mu_1\Vert\mu_2)=\mathcal{L}_{f_\kl\circ \eta}(\mathcal{C}_2\Vert\mathcal{C}_1),\quad D_{\ren,\rho}(\mu_1\Vert\mu_2)=\mathcal{L}_{f_{\ren,\rho}\circ \eta}(\mathcal{C}_2\Vert \mathcal{C}_1).
			\end{align*}
	\end{enumerate}
\end{restatable}

Given the approximation classes \eqref{eqn:class_approx_covariance} and \eqref{eqn:class_approx_precision} and given the covariance loss functions in \eqref{eqn:definition_loss_class}, we can define the following low-rank approximation problems, for every $r\leq n$. We do not consider the case $r>n$, because in this case the problems have the trivial solutions $\mathcal{C}_\pos$ and $\mathcal{C}_\pos^{-1}$ respectively.

\begin{problem}[Rank-$r$ nonpositive covariance updates] 
	\label{prob:optimal_covariance}
	Find $\mathcal{C}^\opt_r\in\mathscr{C}_r$ such that for every $\mathcal{L}\in\mathscr{L}$, $\mathcal{L}(\mathcal{C}_\pos\Vert\mathcal{C}^\opt_r)=\min\{\mathcal{L}(\mathcal{C}_\pos\Vert\mathcal{C}):\ \mathcal{C}\in\mathscr{C}_r\}$.
\end{problem}

\begin{problem}[Rank-$r$ nonnegative precision updates] 
	\label{prob:optimal_precision}
	Find $\mathcal{P}^\opt_r\in\mathscr{P}_r$ such that for every $\mathcal{L}\in\mathscr{L}$, $\mathcal{L}(\mathcal{C}_\pos\Vert (\mathcal{P}^\opt_r)^{-1})=\min\{\mathcal{L}(\mathcal{C}_\pos\Vert\mathcal{P}^{-1} ):\ \mathcal{P}\in\mathscr{P}_r\}$.
\end{problem}

We note that even if an optimal covariance and precision can be found for some given $\mathcal{L}\in\mathscr{L}$, it is not a priori clear that they are in fact independent of $\mathcal{L}$. 

We also emphasise the following. Since the inverse of a self-adjoint positive matrix is again self-adjoint and positive, inverses of covariance operators are covariance operators in the case $\dim\mathcal{H}<\infty$, but not if $\dim\mathcal{H}=\infty$. This is because the trace-class property is not preserved under inversion in infinite dimensions. In fact, if $\dim{\H}=\infty$ and $T$ is trace class, then $T^{-1}$ is an unbounded operator and its eigenvalue sequence is not summable since this sequence is not bounded. In order to define a loss on precisions, as is done in the finite-dimensional case of \cite[Corollary 3.1]{Spantini2015}, one approach is to extend the domain of $\mathcal{L}\in\mathscr{L}$ from $(\mathcal{E}\times\mathcal{E})$ to $(\mathcal{E}\times\mathcal{E})\cup(\mathcal{E}^{-1}\times\mathcal{E}^{-1})$ where $\mathcal{E}^{-1}\coloneqq\{\mathcal{C}^{-1}:\mathcal{C}\in\mathcal{E}\}$. This can be done by setting $\mathcal{L}_{ext}(\mathcal{C}_1^{-1}\Vert \mathcal{C}_2^{-1}) \coloneqq \mathcal{L}(\mathcal{C}_1\Vert\mathcal{C}_2)$ and $\mathcal{L}_{ext}(\mathcal{C}_1\Vert\mathcal{C}_2)\coloneqq\mathcal{L}(\mathcal{C}_1\Vert\mathcal{C}_2)$ for $\mathcal{C}_1,\mathcal{C}_2\in \mathcal{E}$. 
If $(\mathcal{C}_1,\mathcal{C}_2)\in (\mathcal{E}\times\mathcal{E})\cap(\mathcal{E}^{-1}\times\mathcal{E}^{-1})$, then $\mathcal{L}_{ext}(\mathcal{C}_1^{-1}\Vert\mathcal{C}_2^{-1}) = \mathcal{L}(\mathcal{C}_1\Vert \mathcal{C}_2)=\mathcal{L}_{ext}(\mathcal{C}_1\Vert\mathcal{C}_2)$, showing that $\mathcal{L}_{ext}$ is well-defined.
By the definitions \eqref{eqn:definitions_losses}, \eqref{eqn:feldman_hajek_operator} and \Cref{lemma:expansions_feldman_hajek}, $\mathcal{L}$ depends on $\mathcal{C}_1$ and $\mathcal{C}_2$ only via the set of eigenvalues of the bounded extensions of the densely defined operator $\mathcal{C}_1^{-1/2}\mathcal{C}_2\mathcal{C}_1^{-1/2}-I$. \Cref{lemma:expansions_feldman_hajek}\ref{item:covariance_mix_1}-\ref{item:covariance_mix_2} show that eigenvalues of the latter operator remain unchanged when replacing $\mathcal{C}_1$ by $\mathcal{C}_2^{-1}$ and $\mathcal{C}_2$ by $\mathcal{C}_1^{-1}$. Thus, $\mathcal{L}_{ext}(\mathcal{C}_2^{-1}\Vert\mathcal{C}_1^{-1})=\mathcal{L}_{ext}(\mathcal{C}_1\Vert\mathcal{C}_2)=\mathcal{L}(\mathcal{C}_1\Vert\mathcal{C}_2)$. This equation firstly implies that $\mathcal{L}(\mathcal{C}_\pos\Vert\mathcal{P}^{-1})=\mathcal{L}_{ext}(\mathcal{P}\Vert\mathcal{C}_\pos^{-1})$ for $\mathcal{P}\in\mathscr{P}_r$, and we can reformulate \Cref{prob:optimal_precision} accordingly in terms of a loss $\mathcal{L}_{ext}$ on precisions. Secondly, it shows that there is no need to explicitly define a loss on precisions, as we can just use $\mathcal{L}$ on the corresponding covariances in reverse order instead.

\subsection{Equivalence to target measures of low-rank Gaussian approximations}
\label{subsec:equivalence_to_the_exact_posterior}

As discussed in \Cref{sec:equivalence_and_divergences_between_gaussian_measures}, not all approximations $\mathcal{N}(m_\pos,\mathcal{C}_\pr-KK^*)$ are probability measures equivalent to $\mu_\pos$. This equivalence holds only if $\mathcal{C}_\pr-KK^*\in\mathcal{E}$, with $\mathcal{E}$ defined in \eqref{eqn:equivalent_covariances}. The first aim of this section is to characterise the sets $\mathscr{C}_{r,\mathcal{E}}\coloneqq\{\mathcal{C}_\pr-KK^*\in\mathcal{E}:\ K\in\B(\R^r,\H)\}$ and $\mathscr{P}_{r,\mathcal{E}}\coloneqq\{(\mathcal{C}_\pr-KK^*)^{-1}:\ \mathcal{C}_\pr-KK^*\in\mathcal{E},\ K\in\B(\R^r,\H)\}$, which is done in \Cref{lemma:properties_equivalence}\ref{item:properties_equivalence_3} and \Cref{prop:range_of_K}\ref{item:range_of_K_1} respectively. 
We also show that $\mathscr{C}_r\subset\mathscr{C}_{r,\mathcal{E}}$ and $\mathscr{P}_r\subset\mathscr{P}_{r,\mathcal{E}}$, with $\mathscr{C}_r$ and $\mathscr{P}_r$ from \eqref{eqn:class_approx_covariance} and \eqref{eqn:class_approx_precision} respectively, and that these inclusions are strict. In this section, we also determine the relationship between $\mathscr{C}_r^{-1}\coloneqq\{\mathcal{C}^{-1}:\ \mathcal{C}\in\mathscr{C}_r\}$ and $\mathscr{P}_r$, and between \Cref{prob:optimal_covariance} and \Cref{prob:optimal_precision}.

We shall characterise the elements in $\mathcal{E}$ of the form $\mathcal{C}_\pr-KK^*$ with $K\in\B(\R^r,\H)$ for some $r\in\N$ using \Cref{lemma:properties_equivalence}. Because this result is not intrinsic to the Bayesian interpretation, we formulate it for the more general set $\mathcal{E}(m_1,\mathcal{C}_1)$ defined in \eqref{eqn:equivalent_covariances_general}, which contains all covariances $\mathcal{C}$ such that $\mathcal{N}(m_1,\mathcal{C})\sim\mathcal{N}(m_1,\mathcal{C}_1)$, for arbitrary Gaussian target measures $\mathcal{N}(m_1,\mathcal{C}_1)$ with $m_1\in\H$ and injective $\mathcal{C}_1\in L_1(\H)_\R$. \Cref{lemma:properties_equivalence} shows that the operator $I-(\mathcal{C}_1^{-1/2}K)(\mathcal{C}_1^{-1/2}K)^*$ is important for determining whether $\mathcal{C}=\mathcal{C}_1-KK^*$ belongs to $\mathcal{E}(m_1,\mathcal{C}_1)$. \Cref{item:properties_equivalence_2} shows that the assumption that $I-(\mathcal{C}_1^{-1/2}K)(\mathcal{C}_1^{-1/2}K)^*$ is well-defined and nonnegative, is equivalent to $\mathcal{C}\geq 0$, which is necessary for $\mathcal{C}\in\mathcal{E}(m_1,\mathcal{C}_1)$. \Cref{item:properties_equivalence_3} shows that this assumption with the additional assumption of invertibility of $I-(\mathcal{C}_1^{-1/2}K)(\mathcal{C}_1^{-1/2}K)^*$ is both necessary and sufficient for $\mathcal{C}\in\mathcal{E}(m_1,\mathcal{C}_1)$. By \cref{item:properties_equivalence_1}, $I-(\mathcal{C}_1^{-1/2}K)(\mathcal{C}_1^{-1/2}K)^*$ is well-defined under the range condition $\ran{K}\subset\ran{\mathcal{C}_1^{1/2}}$. \Cref{item:properties_equivalence_4} relates the properties of the eigenvalues of $I-(\mathcal{C}_1^{-1/2}K)(\mathcal{C}_1^{-1/2}K)^*$ to the properties $\mathcal{C}\geq 0$ or $\mathcal{C}\in\mathcal{E}(m_1,\mathcal{C}_1)$ of $\mathcal{C}$. If $\mathcal{C}>0$, then \cref{item:properties_equivalence_4} together with \Cref{lemma:extension_of_finite_basis_in_dense_subspace} also shows that the range condition $\ran{K}\subset\ran{\mathcal{C}_1}$ implies the diagonalisability of $I-(\mathcal{C}_1^{-1/2}K)(\mathcal{C}_1^{-1/2}K)^*$ in the Cameron--Martin space of $\mathcal{N}(m_1,\mathcal{C}_1)$.

\begin{restatable}{lemma}{propertiesEquivalence}
	\label{lemma:properties_equivalence}
	Let $\mathcal{C}_1\in L_1(\H)_\R$ be injective and $m_1\in\H$. Let $\mathcal{C}\coloneqq\mathcal{C}_1-KK^*$ for some $K\in\B(\R^r,\H)$ and $r\in\N$. The following holds:
	\begin{enumerate}
		\item
			\label{item:properties_equivalence_1}
			If $\ran{K}\subset\ran{\mathcal{C}_1^{1/2}}$, then $\X\coloneqq I-(\mathcal{C}_1^{-1/2}K)(\mathcal{C}_1^{-1/2}K)^*$ is well-defined and $\mathcal{C}=\mathcal{C}_1^{1/2}\X\mathcal{C}_1^{1/2}$,
		\item 
			\label{item:properties_equivalence_2}
			$\mathcal{C}\geq 0$ if and only if $\ran{K}\subset\ran{\mathcal{C}_1^{1/2}}$ and $\X\geq 0$,
		\item
			\label{item:properties_equivalence_3}
			The following are equivalent:
			\begin{enumerate}
				\item 
					\label{item:properties_equivalence_3_1}
					$\mathcal{C}\in\mathcal{E}(m_1,\mathcal{C}_1)$, with $\mathcal{E}(m_1,\mathcal{C}_1)$ defined in \eqref{eqn:equivalent_covariances_general},
				\item 
					\label{item:properties_equivalence_3_2}
					$\mathcal{C}>0$ and $\ran{\mathcal{C}^{1/2}}=\ran{\mathcal{C}}_1^{1/2}$,
				\item 
					\label{item:properties_equivalence_3_3}
					$\mathcal{C}\geq 0$ and $\ran{\mathcal{C}^{1/2}}=\ran{\mathcal{C}}_1^{1/2}$,
				\item 
					\label{item:properties_equivalence_3_4}
					$\ran{K}\subset\ran{\mathcal{C}_1^{1/2}}$, $\X\geq 0$ and $\ran{\mathcal{C}^{1/2}}=\ran{\mathcal{C}}_1^{1/2}$,
				\item
					\label{item:properties_equivalence_3_5}
					$\ran{K}\subset\ran{\mathcal{C}_1^{1/2}}$, $\X\geq 0$ and $\X$ is invertible.
			\end{enumerate}
			\item
				\label{item:properties_equivalence_4}
				Let $\mathcal{C}\geq 0$. Then $\X=I-\sum_{i=1}^{\rank{K}}d_i^2 e_i\otimes e_i$ with $(d_i^2)_{i=1}^{\rank{K}}\subset(0,1]$ nonincreasing and $(e_i)_{i=1}^{\rank{K}}$ orthonormal. The equivalent statements of \cref{item:properties_equivalence_3} hold if and only if $(d_i^2)_i\subset(0,1)$. If additionally to $\mathcal{C}\geq 0$ also $\mathcal{C}>0$ and $\ran{K}\subset\ran{\mathcal{C}_1}$ holds, then $(d_i^2)_{i=1}^{\rank{K}}\subset(0,1)$ and $(e_i)_{i=1}^{\rank{K}}\subset\ran{\mathcal{C}}_1^{1/2}$.
	\end{enumerate}
\end{restatable}

With \Cref{lemma:properties_equivalence} describing those elements in $\mathcal{E}$ of the form $\mathcal{C_\pr}-KK^*$ with $K\in\B(\R^r,\H)$, we can now characterise the inverses of these elements, and also the inverses of the elements in $\mathscr{C}_r$ and $\mathscr{P}_r$. As we did for \Cref{lemma:properties_equivalence}, we state the result for low-rank approximations of injective covariances of arbitrary Gaussian measures, rather than only for the prior.

\begin{restatable}{proposition}{propRangeOfK}
	\label{prop:range_of_K}
	Let $\mathcal{C},\mathcal{C}_1\in L_1(\H)_\R$, $m_1\in\H$ and $r\in\N$. Suppose $\mathcal{C}_1$ is injective. The following hold:
	\begin{enumerate}
		\item
			\label{item:range_of_K_1}
			$\mathcal{C}=\mathcal{C}_1-KK^*$ for $K\in\B(\R^r,\H)$ and $\mathcal{C}\in\mathcal{E}(m_1,\mathcal{C}_1)$ if and only if $\mathcal{C}$ is injective and $\mathcal{C}^{-1} = \mathcal{C}_1^{-1/2}(I+ZZ^*)\mathcal{C}_1^{-1/2}$ on $\ran{\mathcal{C}}$ for some $Z\in\B(\R^r,\mathcal{H})$. In this case, $\rank{Z}=\rank{K}$.
		\item
			\label{item:range_of_K_2}
			$\mathcal{C}=\mathcal{C}_1-KK^*$ for $K\in\B(\R^r,\H)$, $\mathcal{C}\in\mathcal{E}(m_1,\mathcal{C}_1)$ and $\ran{K}\subset\ran{\mathcal{C}_1}$ if and only if $\mathcal{C}$ is injective, $\ran{\mathcal{C}}=\ran{\mathcal{C}_1}$ and $\mathcal{C}^{-1}=\mathcal{C}_1^{-1}+UU^*$ on $\ran{\mathcal{C}_1}$ for some $U\in\B(\R^r,\H)$. In this case, $\rank{U}=\rank{K}$.
		\item
			\label{item:range_of_K_3}
			$\mathcal{C}=\mathcal{C}_1-KK^*$ for $K\in\B(\R^r,\H)$, $\mathcal{C}>0$ and $\ran{K}\subset\ran{\mathcal{C}_1}$ if and only if $\mathcal{C}$ is injective, $\ran{\mathcal{C}}=\ran{\mathcal{C}_1}$ and $\mathcal{C}^{-1}=\mathcal{C}_1^{-1}+UU^*$ on $\ran{\mathcal{C}_1}$ for some $U\in\B(\R^r,\H)$. In this case, $\rank{U}=\rank{K}$.
	\end{enumerate}
\end{restatable}

\Cref{item:range_of_K_3} is a slight reformulation of \cref{item:range_of_K_2}, and is useful in view of the definition \eqref{eqn:class_approx_covariance} and \eqref{eqn:class_approx_precision}, because with $(m_1,\mathcal{C}_1)\leftarrow(0,\mathcal{C}_\pr)$ it shows that $\mathscr{C}_r^{-1}=\mathscr{P}_r$. This fact is summarised in \Cref{cor:correspondence}\ref{item:equivalent_approximation_classes}. Furthermore, a comparison of \cref{item:range_of_K_2} and \cref{item:range_of_K_3} shows that for $\mathcal{C}_\pr-KK^*$ with $K\in\B(\R^r,\H)$, we have the equivalent statements 
\begin{enumerate}
		\item 
			$\mathcal{C}_\pr-KK^*>0$ and $\ran{K}\subset\ran{\mathcal{C}_\pr}$, and 
		\item 
			$\mathcal{C}_\pr-KK^*\in\mathcal{E}$ and $\ran{K}\subset\ran{\mathcal{C}_\pr}$.
\end{enumerate}
Hence, $\mathscr{C}_r\subset\mathcal{E}$. This fact is reiterated in \Cref{cor:problems_well_defined}.

We comment on the difference between the statements in \cref{item:range_of_K_1,item:range_of_K_2}. If the equivalent conditions of \cref{item:range_of_K_1} hold, then \cref{item:range_of_K_1} implies $(I+ZZ^*)\mathcal{C}_1^{-1/2}h\in\ran{\mathcal{C}_1^{1/2}}$ for any $h\in\ran{\mathcal{C}}$. However, this does not imply that $k_1\coloneqq \mathcal{C}_1^{-1/2}h$ and $k_2\coloneqq ZZ^*\mathcal{C}_1^{-1/2}h$ lie in $\ran{\mathcal{C}_1^{1/2}}$, only that their sum $k_1+k_2$ does. Under the additional condition that $k_1,k_2\in\ran{\mathcal{C}_1^{1/2}}$, we may write $\mathcal{C}^{-1}h = \mathcal{C}_1^{-1/2}k_1+\mathcal{C}_1^{-1/2}k_2= \mathcal{C}_1^{-1}h + \mathcal{C}_1^{-1/2}ZZ^*\mathcal{C}_1^{-1/2}h$. We can formulate the latter as $\mathcal{C}_1^{-1}h+UU^*h$ for a suitable $U$, as shown in the proof. 
Thus, to be able to write $\mathcal{C}_1^{-1/2}(I+ZZ^*)\mathcal{C}_1^{-1/2}=\mathcal{C}_1^{-1}+\mathcal{C}_1^{-1/2}ZZ^*\mathcal{C}_1^{-1/2}$ on all of $\ran{\mathcal{C}}$, as in the formulation of \cref{item:range_of_K_2}, one needs to impose restrictions on $\ran{\mathcal{C}}$, and hence on $\ran{K}$, in \cref{item:range_of_K_1}. As \cref{item:range_of_K_2} shows, the required condition is precisely $\ran{K}\subset\ran{\mathcal{C}}_1$. 

We give an example of $\mathcal{C}=\mathcal{C}_1-KK^*$ with $K\in\B(\R^r,\H)$ for which $\mathcal{C}\in\mathcal{E}(m_1,\mathcal{C}_1)$ but not $\ran{K}\subset\ran{\mathcal{C}_1}$, which shows that the additional condition $\ran{K}\subset\ran{\mathcal{C}_1}$ in \cref{item:range_of_K_2} compared to \cref{item:range_of_K_1} is not vacuous. 
Let $\H$ be infinite-dimensional, so that $\ran{\mathcal{C}_1}$ is a proper subspace of $\ran{\mathcal{C}_1^{1/2}}$. Let $h\in\ran{\mathcal{C}_1^{1/2}}\backslash\ran{\mathcal{C}_1}$ and define $k\coloneqq \norm{\mathcal{C}_1^{-1/2}h}^{-1}h$ so that $z\coloneqq\mathcal{C}_1^{-1/2}k$ has unit norm.  With $\varphi$ any unit vector in $\R^r$, we define the rank-1 operator $K\coloneqq \frac{1}{2}k\otimes \varphi\in\B(\R^r,\H)$. Hence $\mathcal{C}\coloneqq \mathcal{C}_1-KK^*=\mathcal{C}_1-\frac{1}{4}k\otimes k$ satisfies $\ran{K}=\Span{k}\subset\ran{\mathcal{C}_1^{1/2}}$ and $\ran{K}\not\subset\ran{\mathcal{C}_1}$. Furthermore, $I-(\mathcal{C}_1^{-1/2}K)(\mathcal{C}_1^{-1/2}K)^* = I-\frac{1}{4}z\otimes z$ is nonnegative and invertible by \Cref{lemma:inverse_of_self_adj_hilbert_schmidt_perturbation} applied with $e_1\leftarrow z$, $\delta_1\leftarrow -\frac{1}{4}$ and $\delta_i\leftarrow 0$ for $i>1$. By the implication \ref{item:properties_equivalence_3_5}$\Rightarrow$\ref{item:properties_equivalence_3_1} in \Cref{lemma:properties_equivalence}\ref{item:properties_equivalence_3}, $\mathcal{C}\in\mathcal{E}(m_1,\mathcal{C}_1)$, which furnishes the desired example.

In our Bayesian context, i.e.\ setting $(m_1,\mathcal{C}_1)\leftarrow(0,\mathcal{C}_\pr)$, \Cref{lemma:properties_equivalence}\ref{item:properties_equivalence_3} shows that for all $\mathcal{C}\in L_1(\H)_\R$ which satisfy $\mathcal{C}\in\mathcal{E}$ and $\mathcal{C}=\mathcal{C}_\pr-KK^*$ for some $K\in\B(\R^r,\H)$, it holds that $\ran{K}\subset\ran{\mathcal{C}_\pr^{1/2}}$. However, as \Cref{prop:range_of_K}\ref{item:range_of_K_1}-\ref{item:range_of_K_2} shows, not all $\mathcal{C}\in L_1(\H)_\R$ which satisfy $\mathcal{C}\in\mathcal{E}$ and $\mathcal{C}=\mathcal{C}_\pr-KK^*$ for some $K\in\B(\R^r,\H)$ have associated precision of the form \eqref{eqn:class_approx_precision}, only those for which not only $\ran{K}\subset\ran{\mathcal{C}_\pr^{1/2}}$ but also $\ran{K}\subset\ran{\mathcal{C}_\pr}$ holds. In general, the precisions of covariance operators $\mathcal{C}\in\mathcal{E}$ of the form $\mathcal{C}_\pr-KK^*$ are of the form $\mathcal{C}_\pr^{-1/2}(I+ZZ^*)\mathcal{C}_\pr^{-1/2}$. Of course, if $\dim{\H}<\infty$, then $\ran{\mathcal{C}_\pr}=\H$ and both forms always agree, so that the difference between \Cref{prop:range_of_K}\ref{item:range_of_K_1}-\ref{item:range_of_K_2} disappears.

Since $\ran{\mathcal{C}_\pr G^*(\mathcal{C}_\obs+G\mathcal{C}_\pr G^*)^{-1/2}}\subset\ran{\mathcal{C}_\pr}$, the update \eqref{eqn:pos_covariance} of $\mathcal{C}_\pr$ and \Cref{prop:range_of_K}\ref{item:range_of_K_2} or \Cref{prop:range_of_K}\ref{item:range_of_K_3} with $r\leftarrow n$ and $(m_1,\mathcal{C}_1)\leftarrow(0,\mathcal{C}_\pr)$ show that $\ran{\mathcal{C}_\pr}=\ran{\mathcal{C}_\pos}$. This provides another argument showing that $\ran{\mathcal{C}_\pr}=\ran{\mathcal{C}_\pos}$, besides the explicit computation of \cite[Example 6.23]{Stuart2010}.

\begin{remark}[Choice of approximation class]
	\label{rmk:difference_spantini_updates}
	A natural generalisation to infinite dimensions of the low-rank approximation classes for covariance and precision considered in the finite-dimensional setting of \cite[eqs.\ (2.4) and (4.1)]{Spantini2015}, is to take $\widetilde{\mathscr{C}}_r\coloneqq\{\mathcal{C}_\pr-KK^*>0:\ K\in\B(\R^r,\H)\}$ and $\widetilde{\mathscr{P}}_r=\{\mathcal{C}_\pr^{-1}+UU^*:\ U\in\B(\R^r,\H)\}$. Let $\mathscr{C}_{r,\mathcal{E}}\coloneqq\{\mathcal{C}_\pr-KK^*\in\mathcal{E}:\ K\in\B(\R^r,\H)\}$ as in the start of \Cref{subsec:equivalence_to_the_exact_posterior}. By the preceding discussion, we have the proper inclusion $\widetilde{\mathscr{P}}_r^{-1}\subset\mathscr{C}_{r,\mathcal{E}}$, and by the equivalence \ref{item:properties_equivalence_3_2}$\Leftrightarrow$\ref{item:properties_equivalence_3_1} of \Cref{lemma:properties_equivalence}\ref{item:properties_equivalence_3} also the proper inclusion $\mathscr{C}_{r,\mathcal{E}}\subset\widetilde{\mathscr{C}}_r$ holds. That is, in general $\widetilde{\mathscr{P}}_r^{-1}$ contains strictly fewer covariances than those that maintain equivalence between the resulting approximate and exact posterior, while $\widetilde{\mathscr{C}}_r$ contains strictly more. The loss classes considered in this work require densities to exist, thus the set $\widetilde{\mathscr{C}}_r$ is not suitable. Note that by the definition of $\mathscr{P}_r$ in \eqref{eqn:class_approx_precision}, $\widetilde{\mathscr{P}}_r=\mathscr{P}_r$. The fact that $\widetilde{\mathscr{P}}_r^{-1}\subset\mathscr{C}_{r,\mathcal{E}}$ motivates the use of approximation class $\widetilde{\mathscr{P}}_r$ in this work. Because the form of the precision updates in $\widetilde{\mathscr{P}}_r$ parallels the form in \cite[eq.\ (4.1)]{Spantini2015} that lies at the core of the approximation procedure of \cite{Spantini2015}, our work can be considered to naturally generalise \cite{Spantini2015}.
\end{remark}

\begin{example}
	\label{ex:laplacian}
	To illustrate the conditions of \Cref{lemma:properties_equivalence} and \Cref{prop:range_of_K} in an example, we consider the setting of \cite[Example 6.16]{Stuart2010}.
	Consider the infinite-dimensional separable Hilbert space $\H\coloneqq L^2((0,1))$ and let $H^1_0((0,1))$ denote the Sobolev space of functions in $\H$ which vanish on the boundary of $(0,1)$ and have a square-integrable weak derivative. For $m\in\N$, we also denote by $H^m((0,1))$ the Sobolev space of functions in $\H$ with $m$ square-integrable weak derivatives. We can then define the one-dimensional Laplacian ${\d{}^2}/{\d{\mathrm{x}}^2}:H^1_0((0,1))\cap H^2((0,1))\subset\H\rightarrow \H$, which is an unbounded operator with nonpositive spectrum. Its eigenpairs are given by $(a_i,h_i)\coloneqq(i^2\pi^2,\sqrt{2}\sin(i\pi x))_i$ and $(h_i)_i$ forms an ONB of $\H$.
	Consider the Gaussian measure $\mathcal{N}(0,\mathcal{C}_1)$ on $\H$ with $\mathcal{C}_1\coloneqq(-{\d{}^2}/{\d{\mathrm{x}}^2})^{-m}$ for some $m\in\N$. The Cameron--Martin space $\ran{\mathcal{C}_1^{1/2 }}$ is given by $H^1_0((0,1))\cap H^m((0,1))$. By \Cref{lemma:properties_equivalence}\ref{item:properties_equivalence_3}\ref{item:properties_equivalence_3_1}-\ref{item:properties_equivalence_3_5}, it follows that a finite-rank update $\mathcal{C}\coloneqq \mathcal{C}_1-KK^*$ of $\mathcal{C}_1$ with $K\in\B(\R^r,\H)$ yields a Gaussian measure $\mathcal{N}(0,\mathcal{C})$ that is equivalent to $\mathcal{N}(0,\mathcal{C}_1)$ if and only if $K$ maps into $H^1_0((0,1))\cap H^m((0,1))$ and 
	$\X \coloneqq I-([-{\d{}^2}/{\d{\mathrm{x}}^2}]^{m/2}K)([-{\d{}^2}/{\d{\mathrm{x}}^2}]^{m/2}K)^*$ is nonnegative and invertible. For example, if $m=1$, $r=1$, and $Kt\coloneqq cth$, $t\in\R$, for some $h\in\H$ and scaling parameter $c>0$, then we require that $h\in H^1_0( (0,1))$ and that $\X=I-c^2([{\d{}^{}}/{\d{{\mathrm{x}}^{}}}]^{1/2}h)\otimes([{\d{}^{2}}/{\d{{\mathrm{x}}^{2}}}]^{1/2}h)$ is nonnegative and invertible. The choice $h\coloneqq h_1\in H^1_0( (0,1))$ implies $[{\d{}^{2}}/{\d{{\mathrm{x}}^{2}}}]^{1/2}h=\pi h$, so that $\X=I-c^2\pi^2P$ with $P\coloneqq h_1\otimes h_1$ the orthogonal projector in $\H$ onto $\Span{h_1}$. Thus, this choice yields an equivalent measure for any $c<1/\pi$. Since $h_1\in H^1_0( (0,1))\cap H^2( (0,1))=\ran{\mathcal{C}_1}$ as well, \Cref{prop:range_of_K}\ref{item:range_of_K_2} implies that $\mathcal{C}^{-1}= \d{}^2/{\d{\mathrm{x}^2}} + k\otimes k$ on $H^1_0( (0,1))\cap H^2( (0,1))$ for some $k\in\H$. 
The condition $h\in H^1_0( (0,1))$ is violated for functions $h$ that do not vanish on the boundary of $(0,1)$, or functions $h$ that lack sufficient smoothness, such as indicator functions. 
\end{example}

\Cref{cor:correspondence}\ref{item:equivalent_approximation_classes} below shows that $\mathscr{C}_r$ and $\mathscr{P}_r$ are in one-to-one correspondence by the operation of taking inverses, and can be seen as a generalisation of the finite-dimensional result \cite[Lemma A.2]{Spantini2015} to infinite-dimensional Hilbert spaces. \Cref{cor:correspondence}\ref{item:problem_equivalence} shows that one may solve \Cref{prob:optimal_covariance} by solving \Cref{prob:optimal_precision} and vice versa.

\begin{restatable}{corollary}{correspondence}
	\label{cor:correspondence}
	Let $r\in\N$ and let $\mathscr{C}_r$ and $\mathscr{P}_r$ be as in \eqref{eqn:class_approx_covariance} and \eqref{eqn:class_approx_precision} respectively.
	\begin{enumerate}
		\item 	
			\label{item:equivalent_approximation_classes}
			For every $K\in\B(\R^r,\H)$ such that $\mathcal{C}_\pr-KK^*\in\mathscr{C}_r$, there exists $U\in\B(\R^r,\H)$ of the same rank as $K$, such that $(\mathcal{C}_\pr-KK^*)^{-1}=\mathcal{C}_\pr^{-1}+UU^*\in\mathscr{P}_r$. The reverse correspondence also holds: for every $U\in\B(\R^r,\H)$ such that $\mathcal{C}_\pr^{-1}+UU^*\in\mathscr{P}_r$, there exists $K\in\B(\R^r,\H)$ of the same rank as $U$, such that $(\mathcal{C}_\pr^{-1}+UU^*)^{-1}=\mathcal{C}_\pr-KK^*\in\mathscr{C}_r$. In particular, $\mathscr{C}_r^{-1} \coloneqq\{\mathcal{C}^{-1}:\ \mathcal{C}\in\mathscr{C}_r\}= \mathscr{P}_r$ and $\mathscr{P}_r^{-1}\coloneqq \{\mathcal{P}^{-1}:\ \mathcal{P}\in\mathscr{P}_r\}=\mathscr{C}_r$.
		\item
			\label{item:problem_equivalence}
			An approximation $\mathcal{C}^\opt_r\in\mathscr{C}_r$ solves \Cref{prob:optimal_covariance} if and only if $(\mathcal{C}^\opt_r)^{-1}\in\mathscr{P}_r$ solves \Cref{prob:optimal_precision}. Furthermore, $\mathcal{L}(\mathcal{C}_\pos\Vert\mathcal{C}_r^{\opt}) = \mathcal{L}( \mathcal{C}_\pos\Vert(\mathcal{P}_r^\opt)^{-1})$.
	\end{enumerate}
\end{restatable}

As discussed after \Cref{prop:range_of_K} it holds that $\mathscr{C}_r\subset\mathcal{E}$. Hence \Cref{prob:optimal_covariance} is well-defined in the sense that $\mathcal{L}(\mathcal{C}_\pos\Vert\cdot)$ is finite on $\mathscr{C}_r$ for any $\mathcal{L}\in\mathscr{L}$. By \Cref{cor:correspondence}\ref{item:problem_equivalence}, it follows that \Cref{prob:optimal_precision} is analogously well-defined. These facts are emphasised below.

\begin{restatable}{corollary}{problemsWellDefined}
	\label{cor:problems_well_defined}
	It holds that $\mathscr{C}_r\subset\mathcal{E}$. Thus, for any $\mathcal{L}\in\mathscr{L}$, the map $\mathcal{C}\mapsto\mathcal{L}(\mathcal{C}_\pos\Vert\mathcal{C})$ is finite on $\mathscr{C}_r$ and the map $\mathcal{P}\mapsto\mathcal{L}(\mathcal{C}_\pos\Vert\mathcal{P}^{-1})$ is finite on $\mathscr{P}_r$.
\end{restatable}

\subsection{Differentiability and minimisers of covariance loss function}
\label{subsec:differentiability_and_minimisers_of_covariance_loss_functions}

In order to solve \Cref{prob:optimal_precision}, we formulate it as a minimisation problem over the set of $U\in\B(\R^r,\H)$. By \Cref{cor:correspondence}\ref{item:problem_equivalence}, solving \Cref{prob:optimal_covariance} is equivalent to solving \Cref{prob:optimal_precision}.
We want to find the minimiser of the function 
\begin{equation}
	\label{eqn:definition_Jf}
	J_f:\B(\R^r,\H)\rightarrow\R,\quad U\mapsto \mathcal{L}_f(\mathcal{C}_\pos\Vert(\mathcal{C}_{\pr}^{-1}+UU^*)^{-1}),
\end{equation}
for any $f\in\mathscr{F}$ and $\mathcal{L}_f\in\mathscr{L}$ defined in \eqref{eqn:definitions_losses}, which we shall express as a composition of functions. This composition will facilitate the analysis of its differentiability and thereby the identification of its stationary points.

As described in \Cref{sec:notation}, we denote by $\Lambda$ an eigenvalue map defined on $L_2(\H)_\R$. Fix an arbitrary $f\in\mathscr{F}$. The restriction of $\Lambda$ to the self-adjoint Hilbert--Schmidt operators with eigenvalues sequence in $(-1,\infty)$ shall be postcomposed with the functions
\begin{align}
	\label{eqn:definition_F}
	F_f:\ell^2( (-1,\infty))\rightarrow [0,\infty), \quad F_f((x_i)_i) = \sum_{i}^{}f(x_i),
\end{align}
which are well-defined by \Cref{lemma:finite_loss}\ref{item:properties_spectral_f}. If $P\in\B(\ell_2( (-1,\infty)) )$ is a permutation, so that $((Px)_i)_i=(x_{\pi(i)})_i$ for some bijection $\pi$ of $\N$, then $F_f(Px)=F_f(x)$ for any $x\in\ell^2((-1,\infty))$.

We finally define the function $g:\B(\R^r,\H)\rightarrow L_2(\H)_\R$ by 
\begin{align}
	\label{eqn:definition_g}
	g(U) = \mathcal{C}_\pos^{1/2}UU^*\mathcal{C}_\pos^{1/2} - \mathcal{C}_\pos^{1/2}H\mathcal{C}_\pos^{1/2},
\end{align}
where $H$ is the Hessian given in \eqref{eqn:hessian}.
That is, $g(U)$ is a nonnegative, self-adjoint, rank-$r$ update of the negative of the posterior-preconditioned Hessian.
The image of $g$ in fact consists of Hilbert--Schmidt operators which can be diagonalised in the Cameron--Martin space, as is shown next. This result also motivates the definition of $g$, as it shows that $g(U)$ has the same eigenvalues as $R(\mathcal{C}_{\pos}\Vert(\mathcal{C}_{\pr}^{-1}+UU^*)^{-1})$. 

\begin{restatable}{lemma}{propertiesG}
	\label{lemma:properties_g}
	Let $r\in\N$, $U\in\mathcal{B}(\R^r,\mathcal{H})$ and $g$ be as in \eqref{eqn:definition_g}. Then $\rank{g(U)}\leq r+\rank{H}$ and there exists a sequence $(e_i)_i\subset\ran{\mathcal{C}_\pr^{1/2}}$ which forms an ONB of $\mathcal{H}$ and a sequence $(\gamma_i)_i\in\ell^2( (-1,\infty))$ satisfying
	$g(U) = \sum_{i}^{}\gamma_i e_i\otimes e_i$.
	Finally, the eigenvalues of $g(U)$ and $R(\mathcal{C}_{\pos}\Vert(\mathcal{C}_{\pr}^{-1}+UU^*)^{-1})$ agree, counting multiplicities.
\end{restatable}

As a consequence of \Cref{lemma:properties_g}, we can write $J_f$ as
\begin{align}
	\label{eqn:J_expression}
	J_f(U) = F_f(\Lambda(R(\mathcal{C}_\pos\Vert (\mathcal{C}_\pr^{-1}+UU^*)^{-1} ))) = F_f\circ \Lambda \circ g\ (U),
\end{align}
which yields the desired reformulation of the loss as a composition of functions. We use \cite[Theorem 12.4.5 (i)]{bogachevRealFunctionalAnalysis2020} to search for a solution of \Cref{prob:optimal_precision} in the set of stationary points of $J_f$. For this, we need to show that $J_f$ is Gateaux differentiable. To do so, we use the following result, which states that $g$ and $F_f$ are Fr\'echet differentiable and gives an explicit form of the derivatives. Gateaux- and Fr\'echet derivatives are infinite-dimensional analogues of directional and total derivatives, see for example \cite[Section 3.6]{Hsing2015}, \cite[Section 1.4]{hinze_optimization_2009} or \cite[Section 12.1]{bogachevRealFunctionalAnalysis2020} for the definition of Gateaux and Fr\'echet differentiability.

\begin{restatable}{lemma}{differentiabilityOfDecomposition}
	\label{lemma:differentiability_of_decomposition}
	The functions $g$ and $F_f$ defined in \eqref{eqn:definition_g} and \eqref{eqn:definition_F} respectively are Fr\'echet differentiable, with derivatives
	\begin{align*}
		g'(U)(V) &= \mathcal{C}_\pos^{1/2}(UV^*+VU^*)\mathcal{C}_\pos^{1/2}, & U,V\in\B(\R^r,\H), \\
		F_f'(x)(y) &= \sum_{i}^{} f'(x_i)y_i, &x\in\ell^2( (-1,\infty)),\ y\in\ell^2(\R).
	\end{align*}
\end{restatable}

\begin{remark}[Necessity of assumptions on $\mathscr{F}$]
	For a finite set of indices $i\in\{1,\ldots,l\}$, $l\in\N$, the convergence $(f(x_i+y_i)-f(x_i)-f'(x_i)y_i)/y_i\rightarrow 0$ is uniform in $i$. This implies that $\frac{1}{\norm{y}}(\sum_{i}^{l}f(x_i+y_i)-f(x_i) - f'(x_i)y_i)\rightarrow 0$, which implies differentiability of $F_f$ for finite-dimensional $\mathcal{H}$. In infinite dimensions, the convergence of each term is no longer uniform in $i$, as now $i\in\N$, and the previous sum need not converge to 0. Compared to a finite-dimensional setting, in the infinite-dimensional setting we therefore need more assumptions on $f$ to obtain the desired convergence. Hence we restrict the function $f$ to the class of spectral functions $\mathscr{F}$ from \eqref{eqn:definition_spectral_f}. In particular, we require additionally that $f$ has minimum 0 and a derivative which is Lipschitz at 0. 
\end{remark}

Let $U,V\in\B(\R^r,\mathcal{H})$. 
If $\mathcal{W}\subset L_2(\mathcal{H})_\R$ is a subspace of finite dimension that contains $g(U+tV)$ for all $t\in\R$, then the restriction $\restr{(F_f\circ\Lambda)}{\mathcal{W}}:\mathcal{W}\rightarrow\R$ of $F_f\circ \Lambda$ to $\mathcal{W}$ satisfies $\restr{(F_f\circ\Lambda)}{\mathcal{W}}\circ g(U+tV) = F_f\circ\Lambda\circ g(U+tV)$ for all $t\in\R$. Thus, $F_f\circ\Lambda\circ g$ is Gateaux differentiable at $U$ in the direction $V$ if and only if $\restr{(F_f\circ\Lambda)}{\mathcal{W}}\circ g$ is. Hence, by the chain rule, e.g.\ \cite[Section 1.4.1]{hinze_optimization_2009} or \cite[Theorem 12.2.2]{bogachevRealFunctionalAnalysis2020}, it suffices to show that $\restr{(F_f\circ\Lambda)}{\mathcal{W}}$ is Fr\'echet differentiable on all of $(\mathcal{W},\norm{\cdot}_{L_2(\mathcal{H})})$ and that $g$ is Gateaux differentiable at $U$ in the direction $V$ in order to show the Gateaux differentiability of $J_f = F_f\circ\Lambda\circ g$ at $U$ in the direction $V$. This observation is useful, since such a finite-dimensional subspace $\mathcal{W}$ exists, e.g.\ $\mathcal{W}\coloneqq\{X\in L_2(\H):\ \ran{X}\subset\ran{\mathcal{C}_\pos^{1/2}UU^*\mathcal{C}_\pos^{1/2}}+\ran{\mathcal{C}_\pos^{1/2}(UV^*+VU^*)\mathcal{C}_\pos^{1/2}}+\ran{\mathcal{C}_\pos^{1/2}VV^*\mathcal{C}_\pos^{1/2}}+\ran{\mathcal{C}_\pos^{1/2}H\mathcal{C}_\pos^{1/2}}$\}. This subspace is finite-dimensional because $U$, $V$, and $H$ are finite-rank.

We now use the finite-dimensional result of \cite[Theorem 1.1]{Lewis1996} on differentiability of permutation invariant functions of spectra of symmetric matrices to deduce Fr\'echet differentiability of $\restr{(F\circ\Lambda)}{\mathcal{W}}$ for certain $\mathcal{W}\subset L_2(\mathcal{H})_\R$ and $F:\R^{\dim{\mathcal{H}}}\rightarrow\R$. This is done in \Cref{prop:differentiability_F_after_Lambda}, using \Cref{prop:spectral_derivative}. For this purpose, we introduce the following definition.

\begin{definition}
	\label{def:symmetry}
	Let $m\in\N\cup\{\infty\}$. A set $\Omega\subset \R^m$ is symmetric if $Px\in\Omega$ for every $x\in\Omega$ and every permutation $P:\R^m\rightarrow\R^m$.
	If $\Omega\subset\R^m$ is symmetric, then a function $\G:\Omega\rightarrow\R$ is symmetric if $\G(Px)=\G(x)$ for every $x\in\Omega$ and every permutation $P:\R^m\rightarrow\R^m$. 
\end{definition}

As an example of a symmetric set and a symmetric function, consider the set $\ell^2( (-1,\infty))$ and the function $F_f$ from \eqref{eqn:definition_F} for any $f\in\mathscr{F}$. 

Recall from \Cref{sec:notation} the definition of the eigenvalue map $\Lambda^m:L_2(\mathcal{Z})_\R\rightarrow\R^m$ for an $m$-dimensional subspace $\mathcal{Z}\subset\mathcal{H}$ and $m\in\N$. The ordering of eigenvalues given by $\Lambda^m$ is nonincreasing.
The following result relates the Fr\'echet differentiability of $\G\circ\Lambda^m$ and $\G$ for symmetric functions $\G$.

\begin{restatable}{proposition}{spectralDerivative}
	\label{prop:spectral_derivative}
	Let $m\in\N$ and let the set $\Omega\subset\R^m$ be open and symmetric, and suppose that $\G:\Omega \rightarrow\R$ is symmetric. Let $\mathcal{Z}\subset\mathcal{H}$ be $m$-dimensional and let $X\in L_2(\mathcal{Z})_\R$ be such that $\Lambda^m(X)\in\Omega$.
	Then the function $\G\circ\Lambda^m:L_2(\mathcal{Z})_\R\rightarrow\R$ is Fr\'echet differentiable at $X$ if and only if $\G$ is Fr\'echet differentiable at $\Lambda^m(X)\in\R^m$. In this case the Fr\'echet derivative of $\G\circ\Lambda^m$ at $X$ is
	\begin{align*}
		(\G\circ\Lambda^m)'(X) = \sum_{i}^{}\G'(\Lambda^m(X))_i e_i\otimes e_i\in L_2(\mathcal{Z}),
	\end{align*}
		where $(e_i)_i$ is an orthonormal sequence in $\mathcal{Z}$ satisfying $X=\sum_{i}^{}(\Lambda^m(X))_ie_i\otimes e_i$.
\end{restatable}

\begin{remark}
	By definition of the Fr\'echet derivative, $(\G\circ \Lambda^m)'(X)\in L_2(\mathcal{Z})_\R^*$. By the Riesz representation theorem, $L_2(\mathcal{Z})_\R^*\simeq L_2(\mathcal{Z})_\R$, and we consider $(\G\circ \Lambda^m)'(X)$ as an element of $L_2(\mathcal{Z})_\R$. 
\end{remark}

As a consequence of \Cref{prop:spectral_derivative}, which is a result on Fr\'echet differentiability for symmetric functions of spectra of operators in $L_2(\mathcal{Z})$ for $\dim{\mathcal{Z}}<\infty$, we can now deduce the Fr\'echet differentiability of $\restr{(F\circ\Lambda)}{\mathcal{W}}$ for symmetric functions $F$ and suitable finite-dimensional subspaces $\mathcal{W}$ of $L_2(\mathcal{H})_\R$.

\begin{restatable}{proposition}{differentiabilityFAfterLambda}
	Let $\mathcal{Z}\subset\mathcal{H}$ be a finite-dimensional subspace. Let $\mathcal{W}\coloneqq\{X\in L_2(\mathcal{H})_\R:\ \ran{X}\subset \mathcal{Z}\}\subset L_2(\mathcal{H})_\R$.
	Let $F:\ell^2(\R)\rightarrow\R$ be a symmetric function and let $X\in\mathcal{W}$.
Then $\ker{X}^\perp=\ran{X}$, and if $F$ is Fr\'echet differentiable at $\Lambda(X)$, then $\restr{(F\circ \Lambda)}{\mathcal{W}}:\mathcal{W}\rightarrow\R$ is Fr\'echet differentiable at $X\in\mathcal{W}$. In this case, the Fr\'echet derivative is given by
	\begin{align*}
		\restr{(F\circ\Lambda)}{\mathcal{W}}'(X) = \sum_{i}^{}F'(\Lambda(X))_ie_i\otimes e_i\in L_2(\mathcal{H})_\R,
	\end{align*}
	where $(e_i)_i$ is an orthonormal sequence in $\mathcal{Z}$ satisfying $X=\sum_{i}^{}\Lambda(X)_i e_i\otimes e_i.$
	\label{prop:differentiability_F_after_Lambda}
\end{restatable}

\begin{remark}
	The stated differentiability of $\restr{(F\circ \Lambda)}{\mathcal{W}}:\mathcal{W}\rightarrow\R$ holds with respect to the subspace topology on $\mathcal{W}$ inherited from $L_2(\mathcal{H})$. That is, we consider $\mathcal{W}$ as a Hilbert space with its $\norm{\cdot}_{L_2(\mathcal{H})}$ norm.
\end{remark}

Recall that $J_f=F_f\circ\Lambda\circ g$ by \eqref{eqn:J_expression}. We can now prove Gateaux differentiability of $J_f$.

\begin{restatable}{proposition}{differentiabilityJ}
	\label{prop:differentiability_J}
	Let $F_f$, $g$, and $J_f$ be as defined in \eqref{eqn:definition_F}, \eqref{eqn:definition_g}, and \eqref{eqn:J_expression} respectively. Then $J_f$ is Gateaux differentiable on $\B(\R^r,\mathcal{H})$, and for any $U,V\in\mathcal{B}(\R^r,\mathcal{H})$, the Gateaux derivative at $U$ in the direction $V$ is given by
	\begin{align*}
		J_f'(U)(V) = 2\sum_{i}^{}f'(\Lambda_i(g(U)))\langle \mathcal{C}_\pos^{1/2}e_i,VU^*\mathcal{C}_\pos^{1/2}e_i\rangle,
	\end{align*}
	where $(e_i)_i$ is an ONB of $\mathcal{H}$ satisfying $g(U)=\sum_{i}^{}\Lambda_i(g(U))e_i\otimes e_i$.
\end{restatable}

A coercive and Gateaux differentiable function has a global minimum, that can be found among its stationary points. The following lemma establishes the coercivity of $J_f$ over every finite-dimensional subspace of $\B(\R^r,\mathcal{H})$, that is, it establishes the coercivity of $J_f$ over $\mathcal{B}(\R^r,\mathcal{V})$ for every finite-dimensional subspace $\mathcal{V}$ of $\H$.

\begin{restatable}{lemma}{coercivity}
	\label{lemma:coercivity}
	Let $f\in\mathscr{F}$ and $\mathcal{V}\subset \H$ be finite-dimensional. Then $J_f$ is coercive over $\B(\R^r,\mathcal{V})$, i.e.\ $J_f(U_n)\rightarrow\infty$ whenever $\norm{U_n}\rightarrow\infty$. In particular, $J_f$ has a global minimum on $\B(\R^r,\mathcal{V})$, which can be found among the stationary points of the restriction of $J_f$ to $\B(\R^r,\mathcal{V}$).
\end{restatable}

Unless $\H$ is finite-dimensional, the function $J_f$ is not coercive on all of $\B(\R^r,\H)$. To show this, we exploit the property that the finite-dimensional ranges of a sequence $U_n\in \B(\R^r,\H)$ do not need to lie in the same finite-dimensional subspace of $\H$. 

\begin{example}
	Let $\H=\ell^2(\R)$ and $r=1$. Furthermore, let $(e_k)_k$ be the standard basis of $\H$ and put $\mathcal{C}_\pos e_k = k^{-\alpha}e_k$ for $\alpha>1$. Let $U_mt = m^{\beta}te_m$ for $t\in\R$, with $\beta>0$. Then $U_mU_m^*e_k = \delta_{m,k}m^{2\beta}e_m$. It follows that $\mathcal{C}_\pos^{1/2}U_mU_m^*\mathcal{C}_\pos^{1/2}e_k= \delta_{m,k}m^{-\alpha+2\beta}e_m$. Hence $\norm{\mathcal{C}_\pos^{1/2}U_mU_m^*\mathcal{C}_\pos^{1/2}}^2 = \langle \mathcal{C}_\pos^{1/2}U_mU_m^*\mathcal{C}_\pos^{1/2}e_m,e_m\rangle= m^{-\alpha+2\beta}$ which is bounded from above for $\alpha>2\beta$. Therefore, for $\alpha>2\beta$, $\norm{g(U_m)}$ is bounded in $m$ by the triangle inequality, while $\norm{U_m} = m^\beta\rightarrow\infty$. We now argue that $J_f(U_m)$ is bounded in $m$.
	Let $\gamma_m$ be the eigenvalue of largest magnitude among the eigenvalues of $g(U_m)$. \Cref{lemma:norm_of_compact_self_adjoint} implies $\gamma_m=\norm{g(U_m)}$ is bounded in $m$ for $\alpha>2\beta$.
	Now, $f(\gamma)\leq f(\gamma_m)+f(-\gamma_m)$ for every eigenvalue $\gamma$ of $g(U_m)$, because $xf'(x)<0$ for $x\not=0$ implies that $f$ increases as $\abs{x}$ increases. By \Cref{lemma:properties_g}, at most $n+r$ eigenvalues of $g(U_m)$ are nonzero. Because $f(0)=0$ for $f\in\mathscr{F}$ in \eqref{eqn:definition_spectral_f}, we conclude from \eqref{eqn:J_expression}, \eqref{eqn:definition_F} and continuity of $f$ that $J_f(U_m)\leq (n+r)(f(\gamma_m)+f(-\gamma_m))$ is bounded in $m$.
\end{example}

In the proof of Theorem \ref{thm:opt_covariance_and_precision}, coercivity on finite-dimensional subspaces of $\B(\R^r,\H)$ of the form $\B(\R^r,\mathcal{V})$ for finite-dimensional $\mathcal{V}\subset\H$ is enough to show the existence of a global minimiser of $J_f$, because all the stationary points lie in one such finite-dimensional subspace.

\subsection{Optimal low-rank posterior covariance approximations}
\label{subsec:optimal_low_rank_posterior_covariance_approximations}

We can now state the solutions to \Cref{prob:optimal_covariance} and \Cref{prob:optimal_precision}. In these solutions, the eigenpairs $(\lambda_i,w_i)_i$ of the operator $R(\mathcal{C}_\pos\Vert\mathcal{C}_\pr)$ defined in \eqref{eqn:feldman_hajek_operator}, occurring in the Feldman--Hajek theorem relating the prior and the posterior distributions, take a central role.

\begin{restatable}{theorem}{optCovarianceAndPrecision}
	Let $r\leq n$ and let $(\lambda_i)_i\in\ell^2((-1,0])$ and $(w_i)_i\subset \ran{\mathcal{C}_\pr^{1/2}}$ be as given in \Cref{prop:bayesian_feldman_hajek}. Define
	\begin{align}
		\label{eqn:optimal_precision}
		\mathcal{P}^\opt_r &\coloneqq \mathcal{C}_{\pr}^{-1}+ \sum_{i=1}^{r}\frac{-\lambda_i}{1+\lambda_i}(\mathcal{C}_{\pr}^{-1/2}w_i)\otimes(\mathcal{C}_{\pr}^{-1/2}w_i),\\
		\label{eqn:optimal_covariance}
		\mathcal{C}^\opt_r &\coloneqq \mathcal{C}_{\pr} - \sum_{i=1}^{r}-\lambda_i(\mathcal{C}_{\pr}^{1/2}w_i)\otimes(\mathcal{C}_{\pr}^{1/2}w_i).
	\end{align}
	Then $\mathcal{P}^\opt_r$ and $\mathcal{C}^\opt_r$ are solutions to \Cref{prob:optimal_precision} and \Cref{prob:optimal_covariance} respectively and $\mathcal{P}^\opt_r$ and $\mathcal{C}^\opt_r$ are inverses of each other. For every $f\in\mathscr{F}$, the associated minimal loss is $\mathcal{L}_f(\mathcal{C}_\pos\Vert \mathcal{C}^\opt_r)=\sum_{i>r}^{}f(\lambda_i)$. The solutions $\mathcal{P}^\opt_r$ and $\mathcal{C}^\opt_r$ are unique if and only if the following holds: $\lambda_{r+1}=0$ or $\lambda_r<\lambda_{r+1}$.
	\label{thm:opt_covariance_and_precision}
\end{restatable}

\begin{remark}[Uniqueness condition]
	\label{rmk:difference_spantini_uniquenss}
	Two remarks are in order when comparing the uniqueness characterisations of \Cref{thm:opt_covariance_and_precision} and of its finite-dimensional analogue in \cite[Theorem 2.3 and Corollary 3.1]{Spantini2015}. Firstly, the condition in \Cref{thm:opt_covariance_and_precision} is not only sufficient but also necessary. On the contrary, \cite{Spantini2015} only states a sufficient condition. Secondly, the sufficient condition of \cite[Theorem 2.3 and Corollary 3.1]{Spantini2015} is that $(\frac{-\lambda_1}{1+\lambda_1},\ldots,\frac{-\lambda_r}{1+\lambda_r})$ are different, i.e.\ that $(\lambda_1,\ldots,\lambda_r)$ are different. From \Cref{thm:opt_covariance_and_precision}, we see that this condition should be interpreted as the condition that  $(\lambda_1,\ldots,\lambda_r)$ are different from $\lambda_{r+1}$ or $\lambda_{r+1}=0$. Indeed, if $(\lambda_1,\ldots,\lambda_r)$ are different among each other but $\lambda_r=\lambda_{r+1}\not=0$, then replacing $(\lambda_r,w_r)$ by $(\lambda_{r+1},w_{r+1})$ in \eqref{eqn:optimal_covariance} and \eqref{eqn:optimal_precision} gives a different solution to \Cref{prob:optimal_covariance} and \Cref{prob:optimal_precision} respectively.
\end{remark}

\begin{example}
	We consider simultaneously diagonalisable prior and posterior covariance operators in the setting $\H=L^2( (0,1))$ of \Cref{ex:laplacian}, based on \cite[Example 6.17]{Stuart2010}. Suppose that the prior and the posterior covariance are given by $\mathcal{C}_\pr=(-{\d{}^2}/{\d{\mathrm{x}}^2})^{-1}$, $\mathcal{C}_\pos=(-{\d{}^2}/{\d{\mathrm{x}}^2}+I)^{-1}$ respectively. 
	The prior and the posterior constructed in this way indeed form equivalent measures, see \cite[Example 6.17]{Stuart2010}.
	The eigenpairs of $\mathcal{C}_\pr$ are given by $(a_i^{-1},w_i)_i$ and of $\mathcal{C}_\pos$ by $((1+a_i)^{-1},w_i)_i$, with 
	\begin{align*}
		a_i\coloneqq i^2\pi^2,\quad \text{and}\quad w_i\coloneqq\sqrt{2}\sin(i\pi x). 
	\end{align*}
	Then, $R(\mathcal{C}_\pos\Vert\mathcal{C}_\pr)$ is diagonalisable in the same ONB $(w_i)_i$, with eigenvalues $\lambda_i = a_i/(1+a_i)-1=-1/(1+a_i)$, $i\in\N$. 
	Since $\mathcal{C}_\pr^{-1/2}w_i=\sqrt{a_i}w_i$, $\mathcal{C}_\pr^{1/2}w_i = \sqrt{a_i}^{-1}w_i$, and $-\lambda_i/(1+\lambda_i)=1/a_i$ the optimal rank-$r$ posterior precision and covariance approximations are now given by \Cref{thm:opt_covariance_and_precision} to be 
	\begin{align*}
		\mathcal{P}^\opt_r &=  \left(I-\frac{\d{}^2}{\d{\mathrm{x}^2}}\right) + \sum_{i=1}^{r}w_i\otimes w_i,\quad
		\mathcal{C}^\opt_r =  \left(I-\frac{\d{}^2}{\d{\mathrm{x}^2}}\right)^{-1} - \frac{1}{a_i(1+a_i)}\sum_{i=1}^{r}w_i\otimes w_i.
	\end{align*}
\end{example}

\Cref{thm:opt_covariance_and_precision} shows that $\mathcal{C}^\opt_r$ and $\mathcal{P}^\opt_r$ are the optimal rank-$r$ updates of $\mathcal{C}_\pr$ and $\mathcal{C}_\pr^{-1}$ respectively, for all $\mathcal{L}\in\mathscr{L}$ simultaneously. By \Cref{lemma:divergences_in_loss_class}, this optimality includes optimality with respect to the forward and reverse KL divergences and the R\'enyi divergences when keeping the mean fixed. 
This also holds for the Amari $\alpha$-divergences and the Hellinger distance, by \Cref{rmk:correspondence_Amari_divergence_Renyi_divergence,rmk:hellinger_divergence}.
The associated losses can be directly calculated using \Cref{thm:opt_covariance_and_precision}. For the Amari $\alpha$-divergences $D_{\am,\alpha}(\cdot\Vert\cdot)$, this follows by \eqref{eqn:amari_renyi_relation}, \Cref{lemma:divergences_in_loss_class}\ref{item:f_for_reverse_divergence} and the skew symmetry of the R\'enyi divergences. For the Hellinger distance $D_\hel(\cdot,\cdot)$, this follows from \eqref{eqn:hellinger_renyi_relation}. We summarise these facts in the following corollary. 

\begin{corollary}
	\label{cor:optimal_covariance_for_amari_and_hellinger}
	Let $r\leq n$, let $\mathcal{C}^\opt_r$ be given by \eqref{eqn:optimal_covariance} and $(\lambda_i)_i$ as in \Cref{prop:bayesian_feldman_hajek}. For $\alpha\in(0,1)$ and $m\in\mathcal{H}$ arbitrary, we have
	\begin{align*}
		D_{\am,\alpha}(\mu_\pos\Vert\mathcal{N}(m,\mathcal{C}^\opt_r)) 
		&= \min\{D_{\am,\alpha}(\mu_\pos\Vert\mathcal{N}(m,\mathcal{C})):\ \mathcal{C}\in \mathscr{C}_r\} \\
		&= \frac{-1}{\alpha(1-\alpha)}\left(\exp\left(-\alpha(1-\alpha)\sum_{i>r}^{}f_{\ren,\alpha}\left(\lambda_i\right)\right)-1\right),\\
		D_{\am,\alpha}(\mathcal{N}(m,\mathcal{C}^\opt_r)\Vert\mu_\pos) 
		&= \min\{D_{\am,\alpha}(\mathcal{N}(m,\mathcal{C})\Vert\mu_\pos)):\ \mathcal{C}\in \mathscr{C}_r\} \\
		&= \frac{-1}{\alpha(1-\alpha)}\left(\exp\left(-\alpha(1-\alpha)\sum_{i>r}^{}f_{\ren,\alpha}\left(\frac{-\lambda_i}{1+\lambda_i}\right)\right)-1\right),\\
		&= \frac{-1}{\alpha(1-\alpha)}\left(\exp\left(-\alpha(1-\alpha)\sum_{i>r}^{}f_{\ren,1-\alpha}\left(\lambda_i\right)\right)-1\right),
	\end{align*}
	where $f_{\ren,\alpha}$ is given in \Cref{lemma:divergences_in_loss_class}\ref{item:f_for_renyi_divergence}. Furthermore, for arbitrary $m\in\mathcal{H}$,
	\begin{align*}
		D_{\hel}(\mu_\pos,\mathcal{N}(m,\mathcal{C}^\opt_r)) 
		&= \min\{D_{\hel}(\mu_\pos,\mathcal{N}(m,\mathcal{C})):\ \mathcal{C}\in \mathscr{C}_r\}\\
		&= \sqrt{2\left(1-\exp\left(-\frac{1}{4}\sum_{i>r}^{}f_{\ren,1/2}\left(\lambda_i\right)\right)\right)}.
	\end{align*}
	The minimiser $\mathcal{C}^\opt_r$ is unique if and only if the following holds: $\lambda_{r+1}=0$ or $\lambda_r<\lambda_{r+1}$.
\end{corollary}

Together, \Cref{thm:opt_covariance_and_precision} and \Cref{cor:optimal_covariance_for_amari_and_hellinger} describe those approximate covariances which retain the most posterior covariance information with respect to several divergences simultaneously. After discretising, this allows one to significantly reduce computational costs, c.f.\ \cite[Table 1]{Flath2011}. Computation of the leading $r$ discretised versions of the eigenpairs $(\lambda_i,w_i)$ in \Cref{thm:opt_covariance_and_precision} and \Cref{cor:optimal_covariance_for_amari_and_hellinger} requires diagonalising a discretised version of one of the operators in \Cref{prop:bayesian_feldman_hajek}, for example $R(\mathcal{C}_\pos\Vert\mathcal{C}_\pr)$ or $\mathcal{C}_\pr^{1/2}H\mathcal{C}_\pr^{1/2}$. This can be done efficiently using matrix-free truncated eigensolvers, such as Krylov methods or randomised singular value decomposition, see \cite{Bui-Thanh2012b,Flath2011,Saad2003}. Furthermore, given the optimal approximation on function space, one can study the consistency of the discretised approximation with this infinite-dimensional limit. The optimal approximation given in \Cref{thm:opt_covariance_and_precision}, which reduces to that of \cite{Spantini2015} in a finite-dimensional setting, is then dimension robust when such consistency is proven. The above results thereby enable both tractable and scalable uncertainty quantification for linear Gaussian inverse problems.

\subsection{Comparison with finite-dimensional setting}
\label{sec:comparison_with_finite_dimensional_setting}

The optimal low-rank posterior covariance approximation procedure was analysed in \cite{Spantini2015} in the finite-dimensional setting. The main differences between the results in \cite{Spantini2015} and the infinite-dimensional setting of this work are as follows.

\begin{enumerate}
	\item The choice of spectral loss functions, see \Cref{rmk:difference_spantini_spectral_functions}.
	\item The choice of approximation classes for the posterior covariance, see \Cref{rmk:difference_spantini_updates}.
	\item The given uniqueness condition, see \Cref{rmk:difference_spantini_uniquenss}.
\end{enumerate}

We emphasise that \Cref{thm:opt_covariance_and_precision} generalises the results of \cite{Spantini2015} to the infinite-dimensional setting. In the infinite-dimensional setting, the intricate connection between optimal posterior covariance approximation and the theory of equivalent Gaussian measures becomes apparent. The analysis provided in this work connects posterior covariance approximation with equivalent Gaussian measure theory, and is inherently discretisation-independent. For consistent discretisation schemes, the optimal approximation of \Cref{thm:opt_covariance_and_precision} then is dimension robust, and may be obtained by applying \Cref{thm:opt_covariance_and_precision} in the discretised setting, i.e.\ by applying \cite[Theorem 2.3]{Spantini2015} to the discretised operators occurring in the discretised inverse problem.

\section{Conclusion}
\label{sec:conclusion}
Linear Gaussian inverse problems on possibly infinite-dimensional Hilbert spaces are an important kind of nonparametric inverse problem. For example, they can be used to approximate nonlinear nonparametric problems using the Laplace approximation. They often serve as the native infinite-dimensional formulation of linear inverse problems before the parameter space $\mathcal{H}$ is discretised and they are in this sense `discretisation independent'. 

Optimal low-rank approximation of the posterior covariance for a class of losses that includes the KL divergence and the Hellinger metric, and optimal low-rank approximation of the posterior mean for the Bayes risk were studied in \cite{Spantini2015}. The analysis showed that certain matrix pencils, namely the ones defined by the Hessian and prior covariance and the prior and posterior covariance, form the central objects of study. So far, these results applied to finite-dimensional parameter spaces only. 

In this work we have formulated the low-rank posterior covariance approximation problem on possibly infinite-dimensional separable Hilbert spaces. We solved this problem and derived the optimal low-rank approximations to the posterior covariance in \Cref{thm:opt_covariance_and_precision}. Equivalent conditions for its uniqueness are also given.
This builds upon the finite-dimensional conclusions of \cite[Section 2 and 3]{Spantini2015} for posterior covariance approximation. 
The resulting posterior approximation, obtained by replacing the covariance with the optimal low-rank approximation and by keeping the mean fixed, is equivalent to the exact posterior distribution, and we have shown exactly which low-rank updates of the prior covariance and precision satisfy this equivalence property in \Cref{lemma:properties_equivalence} and \Cref{prop:range_of_K}.
Furthermore, the posterior covariance approximations are optimal for a class of losses which includes the forward and reverse KL divergences, the Hellinger metric, the Amari $\alpha$-divergences for $\alpha\in(0,1)$ and the R\'enyi divergences. 
Finally, we have shown in \Cref{prop:bayesian_feldman_hajek} that the operator pencils which proved central in the finite-dimensional analysis, are equivalent representations of the Hilbert--Schmidt operator appearing in the Feldman--Hajek theorem which quantifies similarity of Gaussian measures. For linear Gaussian inverse problems, it is therefore this operator that is central to the approximation of the posterior covariance as a low-rank update of the prior covariance. This observation is consistent with the fact that the Hilbert--Schmidt operator in the Feldman--Hajek theorem quantifies the similarity of the Gaussian prior and exact posterior.

The low-rank approximations constructed in this work provide a basis for showing the consistency of optimal low-rank covariance approximations in discretised versions of linear inverse problems. Furthermore, these approximations may be useful for the development of computationally efficient approximations of certain linear Gaussian problems. Finally, they could be used for the approximation of nonlinear nonparametric inverse problems.

\section{Acknowledgements}
The research of the authors has been partially funded by the Deutsche Forschungsgemeinschaft (DFG) Project-ID 318763901 -- SFB1294. The authors thank 
	Youssef Marzouk (Massachusetts Institute of Technology) and Bernhard Stankewitz (University of Potsdam) for helpful discussions, and Thomas Mach (University of Potsdam) for suggestions about the manuscript. The authors thank the referees for their helpful feedback.

\appendix

\section{Auxiliary results}
\label{sec:theoretical_facts}

In this section we collect some auxiliary results on Hilbert spaces and bounded operators, unbounded operators and Gaussian measures.

\subsection{Hilbert spaces and bounded operators}

\begin{lemma}
	\label{lemma:extension_of_finite_basis_in_dense_subspace}
	Let $\mathcal{H}$ be a separable Hilbert space and $\mathcal{D}\subset\mathcal{H}$ be a dense subspace and $(e_i)_{i=1}^m$ be an orthonormal sequence in $\mathcal{D}$ for $m\in\N$. Then there exists a countable sequence $(d_i)_i\subset \mathcal{D}$ such that $(d_i)_i$ is an ONB of $\mathcal{H}$ and $d_i=e_i$ for $i\leq m$.
\end{lemma}

\begin{proof}
	The proof is a slight modification of the argument of \cite[Lemma A.2]{Eldredge2016}. By separability of $\mathcal{H}$ there exists a countable and dense sequence $(h_i)_i$ of $\mathcal{H}$. By density of $\mathcal{D}$ we can construct a countable sequence $(d_i')_i\subset\mathcal{D}$ that is dense in $\mathcal{H}$ by taking an element of $\mathcal{D}$ from the ball $B(h_i,1/j)$, for all $i$ and $j\in\N$. Now, we apply Gram--Schmidt to the countable sequence $(e_1,\ldots,e_m,d_1',d_2',\ldots)\subset\mathcal{D}$ to obtain a countable orthonormal sequence $(d_i)_i\subset\mathcal{D}$. Since $(e_i)_{i=1}^m$ is already orthonormal, $d_i=e_i$ for $i\leq m$. Furthermore, $d_i'\in\Span{d_j,j\leq m+i}$. It follows that $(d_i')_i\subset \Span{(d_i)_i}$ and since $(d_i')_i$ is dense, so is $\Span{(d_i)_i}$.
\end{proof}

\begin{lemma}[{\cite[Proposition II.2.7]{conway_course_2007}}]
	\label{lemma:cstar_property}
	Let $\mathcal{H}$ be a Hilbert space. If $A\in\B(\mathcal{H})$, then $\norm{A}=\norm{A^*}=\norm{AA^*}^{1/2}$.
\end{lemma}

\begin{lemma}[{\cite[Theorem 4.2.6]{Hsing2015}}]
	\label{lemma:norm_of_compact_self_adjoint}
	Let $\mathcal{H}$ be a Hilbert space and $A\in\mathcal{B}(\mathcal{H})$ be compact and self-adjoint.  Then $\norm{A}=\max\{\abs{\lambda}:\ \lambda\text{ is an eigenvalue of }A\}$.
\end{lemma}

\begin{lemma}
	\label{lemma:positive_is_injective_nonnegative}
	Let $\H$ be a Hilbert space and $A\in\B(\H)$. Then $A>0$ if and only if $A\geq 0$ and $A$ is injective.
\end{lemma}

\begin{proof}
	Assume $A$ is positive. If $h\in\ker{A}$, then $\langle Ah,h\rangle=0$, so $h=0$.
	Now assume $A$ is nonnegative and injective. If $\langle Ah,h\rangle = \norm{A^{1/2}h}^2=0$ for $h\not=0$, then $h\in\ker{A}^{1/2}\subset\ker{A}$, so $h=0$.
\end{proof}

\begin{lemma}[{\cite[Theorem 4.3.1]{Hsing2015}}]
	\label{lemma:operator_svd}
	Let $\H,\mathcal{K}$ be Hilbert spaces, and $A\in\B(\H,\mathcal{K})$ be compact. Then $A$ is diagonalisable, that is, there exists an ONB $(h_i)_i$ of $\H$ and an orthonormal sequence $(k_i)_i$ of $\mathcal{K}$ and a nonnegative and nonincreasing sequence $(\sigma_i)_i$ such that $A=\sum_{i}^{}\sigma_i k_i\otimes h_i$.
\end{lemma}

\begin{lemma}[{\cite[Proposition VI.1.8]{conway_course_2007}}]
	\label{lemma:kernel_range}
	Let $\mathcal{H}$, $\mathcal{K}$ be Hilbert spaces and $A\in\mathcal{B}(\mathcal{H},\mathcal{K})$. Then $\ker{A}=\ran{A^*}^\perp$ and $\ker{A}^\perp=\overline{\ran{A^*}}$.
\end{lemma}

\begin{lemma}
	\label{lemma:kernel_of_square}
	Let $\mathcal{H}$ and $\mathcal{K}$ be Hilbert spaces and $A\in\mathcal{B}(\mathcal{H},\mathcal{K})$. Then $\ker{AA^*}=\ker{A^*}$.
\end{lemma}

\begin{proof}
	The inclusion $\ker{A^*}\subset\ker{AA^*}$ is immediate. If $AA^*k=0$ for $k\in\mathcal{K}$, then $\norm{A^*k}^2=\langle AA^*k,k\rangle=0$. Hence $A^*k=0$, showing the reverse inclusion holds.
\end{proof}

\begin{lemma}
	\label{lemma:range_of_square_of_finite_rank}
	Let $\mathcal{H}$, $\mathcal{K}$ be Hilbert spaces and $A\in\mathcal{B}_{00}(\mathcal{H},\mathcal{K})$. Then $\ran{AA^*}=\ran{A}$.
\end{lemma}

\begin{proof}
	Since $\ran{A^*}$ is closed, we have by \Cref{lemma:kernel_range} that $\ran{AA^*}=\ran{AP_{\ran{A^*}}}=\ran{AP_{\overline{\ran{A^*}}}}=\ran{AP_{\ker{A}^\perp}} = \ran{A}$, where $P_V$ denotes the projection onto a closed subspace $V$ of $\mathcal{H}$.
\end{proof}

\begin{lemma}
	\label{lemma:inverse_of_self_adj_hilbert_schmidt_perturbation}
	Let $\H$ be a Hilbert space, $(e_i)_i$ an orthonormal sequence, $(\delta_i)_i\in\ell^2(\R)$ and $T\coloneqq I+\sum_{i}^{}\delta_ie_i\otimes e_i$. The following holds. 
	\begin{enumerate}
		\item
			\label{item:inverse_of_self_adj_hilbert_schmidt_perturbation_1}
			$T$ is invertible in $\B(\H)$ if and only if $\delta_i\not=-1$ for all $i$.
		\item
			\label{item:inverse_of_self_adj_hilbert_schmidt_perturbation_2}
			$T\geq 0$ if and only if $\delta_i\geq -1$ for all $i$.
		\item
			\label{item:inverse_of_self_adj_hilbert_schmidt_perturbation_3}
			$T>0$ if and only if $\delta_i>-1$ for all $i$.
	\end{enumerate}
	In cases \ref{item:inverse_of_self_adj_hilbert_schmidt_perturbation_1} and \ref{item:inverse_of_self_adj_hilbert_schmidt_perturbation_3} above, the inverse of $T$ is $I-\sum_{i}^{}\frac{\delta_i}{1+\delta_i}e_i\otimes e_i$. 
\end{lemma}

\begin{proof}
	Suppose that $T$ is invertible. Then $(1+\delta_i) e_i = Te_i\not=0$ for all $i$, hence $\delta_i\not=-1$ for all $i$. Conversely, suppose that $\delta_i\not=-1$ for all $i$ and let $k\in\mathcal{H}$. Since $(\delta_i)_i\in\ell^2(\R)$, $\abs{(1+\delta_i)^{-1}}\leq 2$ for all $i$ large enough. Because $(\langle k,e_i\rangle)_i\in\ell^2(\R)$, this implies that $\alpha\in\ell^2(\R)$ where $\alpha_i\coloneqq (1+\delta_i)^{-1}\langle k,e_i\rangle$ for all $i$. Hence $h\coloneqq \sum_{i}^{}\alpha_ie_i\in\H$ and $Th = \sum_{i}^{}(1+\delta_i)\langle h,e_i\rangle e_i = \sum_{i}^{}\langle k,e_i\rangle e_i =k$. This shows that $T$ is surjective. Since $T=T^*$, $\ker{T}=\ran{T}^\perp=\{0\}$ by \Cref{lemma:kernel_range}, showing that $T$ is injective, which proves \ref{item:inverse_of_self_adj_hilbert_schmidt_perturbation_1}.
	If $T\geq 0$, then $1+\delta_i=\langle T e_i,e_i\rangle \geq 0$, i.e.\ $\delta_i\geq-1$, for all $i$. Conversely, if $\delta_i\geq -1$ for all $i$, then $\langle Th,h\rangle = \sum_{i}^{}(1+\delta_i)\langle h,e_i\rangle^2\geq 0$. This proves \ref{item:inverse_of_self_adj_hilbert_schmidt_perturbation_2}, and replacing ``$>$'' by ``$\geq$'', also \ref{item:inverse_of_self_adj_hilbert_schmidt_perturbation_3}.

	To compute the inverse of $T$, note that $\frac{\delta_i}{1+\delta_i}\leq 2 \delta_i$ for all $i$ large enough, by the hypothesis that $(\delta_i)_i\in\ell^2(\R)$. Thus, $\frac{\delta_i}{1+\delta_i}\rightarrow 0$ and $\sum_{i}^{}\frac{\delta_i}{1+\delta_i}e_i\otimes e_i$ is well-defined in $\B(\H)$. For $h\in\mathcal{H}$, we have by direct computation,
	\begin{align*}
		\left( I+\sum_{i}^{}\delta_ie_i\otimes e_i \right)\left( I-\sum_{i}^{}\frac{\delta_i}{1+\delta_i}e_i\otimes e_i \right)h &= 
		\left( I+\sum_{i}^{}\delta_ie_i\otimes e_i \right) \sum_{i}^{}\left(1-\frac{\delta_i}{1+\delta_i}\right)\langle h,e_i\rangle e_i \\
		&= \sum_{i}^{}\left(1-\frac{\delta_i}{1+\delta_i}+\delta_i-\frac{\delta_i^2}{1+\delta_i}\right)\langle h,e_i\rangle e_i\\
		&= \sum_{i}^{}\langle h,e_i\rangle e_i=h.
	\end{align*}
	Similarly,
	\begin{align*}
		\left( I-\sum_{i}\frac{\delta_i}{1+\delta_i}e_i\otimes e_i \right) \left( I+\sum_{i}\delta_ie_i\otimes e_i \right)h =&\left( I-\sum_{i}\frac{\delta_i}{1+\delta_i}e_i\otimes e_i \right)\sum_{i}\left(1+\delta_i\right)\langle h,e_{i}\rangle e_{i} 
		\\
		=&\sum_{i} \left(1+\delta_i-\frac{\delta_i}{1+\delta_i}(1+\delta_i)\right)\langle h,e_i\rangle e_i.
	\end{align*}
\end{proof}

\begin{lemma}
	\label{lemma:range_inclusions}
	Let $\mathcal{H},\mathcal{K}$ be Hilbert spaces.
	Suppose $A_1,A_2\in\B(\mathcal{H},\mathcal{K})$. Then the following are equivalent:
	\begin{enumerate}
		\item
			$\ran{A_1}\subset\ran{A_2}$,
		\item
			There exists $C>0$ such that $\{A_1h:\ \norm{h}\leq 1\}\subset \{A_2h:\ \norm{h}\leq C\}$,
		\item
			There exists $C>0$ such that $\norm{A_1^*k}\leq C\norm{A_2^*k}$ for all $k\in\mathcal{K}$.
	\end{enumerate}
\end{lemma}

\begin{proof}
	See \cite[Proposition B.1(i)]{da_prato_stochastic_2014} and its proof.
\end{proof}

\begin{definition}[{\cite[Definition VIII.3.10]{conway_course_2007}}]
	Let $\mathcal{H}$ be a Hilbert space. We say that $W\in\B(\H)$ is a `partial isometry' if $W$ is an isometry on $\ker{W}^\perp$. We call $\ker{W}^\perp$ the `initial space' of $W$ and $\ran{W}$ the `final space' of $W$.
\end{definition}

Recall from \Cref{sec:notation} that $\abs{A}\coloneqq (A^*A)^{1/2}$ for $A\in\B(\H)$. For a proof of the following, see e.g.\ \cite[VIII.3.11]{conway_course_2007}.

\begin{lemma}[Polar decomposition]
	Let $\mathcal{H}$ be a Hilbert space and $A\in\B(\H)$. There exists a partial isometry $W\in\B(\H)$ with initial space $\ker{A}^\perp$ and final space $\ker{A^*}^\perp$ such that $A=W\abs{A}$.
	\label{lemma:polar_decomposition}
\end{lemma}

\begin{lemma}
	\label{lemma:nonsymmetric_root_relation}
	Let $\mathcal{H}$ be a Hilbert space and $A,B\in\B(\mathcal{H})$ be injective with $\ran{AA^*}$ dense. If $AA^*=BB^*$, then there exists a Hilbert space isomorphism $Q\in\B(\H)$ such that $B=AQ$.
\end{lemma}

\begin{proof}
	We first note that $\ran{A}$ is dense, since $\ran{AA^*} = A(\ran{A^*})\subset A(\overline{\ran{A^*}})=A(\ker{A}^\perp)=A(\H)=\ran{A}$ by \Cref{lemma:kernel_range} and $\ker{A}=\{0\}$. Since $AA^*=BB^*$, also $\ran{BB^*}$ and $\ran{B}$ are dense.
	Now, by the polar decomposition applied to $A^*$ and $B^*$, c.f.\ \Cref{lemma:polar_decomposition}, there exist $W_1,W_2\in\B(\H)$ such that $A^*=W_1\abs{A^*}$, $B^*=W_2\abs{B^*}$. Here, $W_1$ is an isometry on $\ker{A^*}^\perp$ with $\ran{W}_1 = \ker{A}^\perp$, and similarly, $W_2$ is an isometry on $\ker{B^*}^\perp$ with $\ran{W_2}=\ran{B}^\perp$. Since $\ker{A}=\{0\}$ by assumption, it follows that $W_1$ is surjective. Since $\ran{A}^\perp =\ker{A^*}=\{0\}$ by assumption and \Cref{lemma:kernel_range}, it follows that $W_1$ is an isometry on all of $\H$. Hence $W_1$ is a surjective isometry on $\mathcal{H}$, that is a Hilbert space isomorphism. Similarly, $W_2$ is a Hilbert space isomorphism, and therefore so is $W_2W_1^{-1}$. Now, $AA^* = BB^*$ implies $\abs{A^*}=\abs{B^*}$. Thus, $B^* = W_2 \abs{B^*}=W_2\abs{A^*}=W_2W_1^{-1} W_1\abs{A^*}=W_2W_1^{-1} A^*$. We conclude that $B=AQ$, where $Q\coloneqq (W_1W_2^{-1})^*\in\B(\H)$ is a Hilbert space isomorphism.
\end{proof}

\subsection{Unbounded operators}

For $A\in\B(\H)$, we denote by $A^\dagger$ the Moore--Penrose inverse of $A$, also known as the generalised inverse and pseudo-inverse of $A$, c.f.\ \cite[Definition 2.2]{engl_regularization_1996}, \cite[Section B.2]{da_prato_stochastic_2014} or \cite[Definition 3.5.7]{Hsing2015}. It holds that $A^\dagger$ is bounded if and only if $\ran{A}$ is closed, c.f.\ \cite[Proposition 2.4]{engl_regularization_1996}. If $A$ is injective, then $A^\dagger=A^{-1}$ on $\ran{A}$.

\begin{lemma}
	\label{lemma:equivalent_norms}
	Let $\mathcal{H}$ and $\mathcal{K}$ be Hilbert spaces. Suppose $A_1,A_2\in\B(\mathcal{H},\mathcal{K})$. If $\ran{A_1}\subset\ran{A_2}$, then there exists $C>0$ such that $\norm{A_2^{\dagger}k}\leq C\norm{A_1^{\dagger}k}$ for all $k\in\ran{A_1}$.
\end{lemma}

The proof is a modification of the arguments in the proof of \cite[Proposition B.1(ii)]{da_prato_stochastic_2014}.

\begin{proof}
	Let us first assume that $A_2$ is injective, so that $A_2^\dagger=A_2^{-1}$ on $\ran{A_2}$.
	We must show that there exists $C>0$ such that $\norm{A_2^{-1}k}\leq C\norm{A_1^\dagger k}$ for all $k\in\ran{A_1}$. We shall obtain a contradiction by supposing no $C>0$ exists such that $\norm{A_2^{-1}k}\leq C\norm{A_1^\dagger k}$ for all $k\in\ran{A_1}$. That is, we suppose that for each $m\in\N$ there exists $k^m\in\ran{A_1}$ such that $\norm{A_2^{-1}k^m}>m\norm{A_1^{\dagger}k^m}$. Since $k^m\in \ran{A_1}$ and $\ran{A_1}\subset \ran{A_2}$, there exist $\tilde{h}_1^m,\tilde{h}_2^m\in\mathcal{H}$ such that $\tilde{h}_1^m=A_1^{\dagger}k^m$ and $\tilde{h}_2^m=A_2^{-1}k^m$. Thus, $A_1\tilde{h}^m_1=A_2\tilde{h}^m_2=k^m$. Define $h_i^m\coloneqq \tilde{h}_i^m/\norm{\tilde{h}_1^m}$, $i=1,2$. Then $\norm{h_1^m}=1$ for all $m$ and $\norm{h_2^m}\rightarrow\infty$ as $m\rightarrow\infty$. On the one hand, for every $C>0$ there exists $M\in\N$ such that $A_2h_2^m\not\in\{A_2h:\ \norm{h}\leq C\}$ for all $m>M$, by injectivity of $A_2$. On the other hand, $A_1h^m_1 = k^m/\norm{\tilde{h}_1^m}=A_2h^m_2$, hence $A_2h^m_2\in\{A_1h:\ \norm{h}\leq 1\}$ for all $m$. By \Cref{lemma:range_inclusions}, $A_2h^m_2\in\{A_2h:\ \norm{h}\leq C\}$ for all $m$ and for some $m$-independent constant $C>0$, which is a contradiction.

	Now let $A_2\in\B(\mathcal{H},\mathcal{K})$ be arbitrary. The subspace $\ker{A}_2^\perp\subset\mathcal{H}$ is closed and therefore a Hilbert space with respect to its subspace topology. Let us denote the restriction of $A_2$ to $\ker{A_2}^\perp$ by $\tilde{A}_2\in\B(\ker{A_2}^\perp,\mathcal{K})$. Then $\tilde{A}_2$ is injective and satisfies $\ran{\tilde{A}_2}=\ran{A_2}$. By construction of the Moore--Penrose inverse, $A_2^\dagger k = \tilde{A}_2^{-1}k\in\mathcal{H}$ for $k\in\ran{\tilde{A}}_2=\ran{A_2}$. By the hypothesis $\ran{A}_1\subset\ran{A}_2$, we have $A_2^\dagger k = \tilde{A}_2^{-1}k\in\mathcal{H}$ for $k\in\ran{A_1}$. From the previous part of the proof we can then conclude the existence of $C>0$ such that $\norm{A_2^\dagger k}=\norm{\tilde{A}_2^{-1}k} \leq C\norm{A_1^\dagger k}$ for all $k\in\ran{A_1}$.
\end{proof}

\begin{definition}[{\cite[Definition X.1.3]{conway_course_2007}}]
	\label{def:closed_operator}
	Let $\H$ be a Hilbert space. A linear operator $A:\dom{A}\subset\H\rightarrow \H$ is said to be closed if its graph $\{(h,Ah):\ h\in\dom{A}\}$ is closed in $\H\oplus \H$.
\end{definition}

\begin{lemma}
	\label{lemma:closed_operators}
	Let $\H$ be a Hilbert space, $A:\dom{A}\subset\H\rightarrow \H$ be closed and $B\in\B(\H)$. Then,
	\begin{enumerate}
		\item
			\label{item:closed_composition}
			$AB$ is closed, 
		\item
			\label{item:closed_sum}
			$A+B$ is closed, 
		\item
			\label{item:closed_inverse}
			if $A$ is also injective, then $A^{-1}:\ran{A}\subset\mathcal{H}\rightarrow\dom{A}\subset\mathcal{H}$ is closed.
	\end{enumerate}
\end{lemma}

\begin{proof}
	If $(h_n,ABh_n)\rightarrow (h,k)\in \H\oplus \H$, then $(Bh_n,ABh_n)\rightarrow (Bh,k)$ by continuity of $B$. Since $A$ is closed, $Bh\in\dom{A}$, that is, $h\in\dom{AB}$, and $k=ABh$. This shows \cref{item:closed_composition}.
	Next, if $( h_n,Ah_n+Bh_n)\rightarrow (h,k)\in\mathcal{H}\oplus\mathcal{H}$, then $Bh_n\rightarrow z$ for some $z\in\mathcal{H}$ by continuity of $B$, and $(h_n,Ah_n)\rightarrow (h,k-z)$. Since $A$ is closed, $h\in\dom{A}=\dom{(A+B)}$ and $Ah=k-z=k-Bh$. This shows \cref{item:closed_sum}.
	Finally, if $A$ is also injective, then we have $\{(h,Ah):\ h\in\dom{A}\}=\{(A^{-1}k,k):\ k\in\ran{A}\}$, and this set is closed if and only if the set $\{(k,A^{-1}k):\ k\in\ran{A}\}=\{(k,A^{-1}k):\ k\in\dom{A^{-1}}\}$ is closed. This shows \cref{item:closed_inverse}.
\end{proof}

\begin{lemma}[Closed graph theorem]
	\label{lemma:closed_graph_theorem}
	Let $\mathcal{H}$ be a Banach space. If $A:\dom{A}\subset\mathcal{H}\rightarrow\mathcal{H}$ satisfies $\dom{A}=\mathcal{H}$, then $A$ is continuous if and only if $A$ is closed.
\end{lemma}

\begin{proof}
	It follows by definition of continuity that $A$ is closed for $A\in\mathcal{B}(\mathcal{H})$. For the converse, see \cite[Theorem III.12.6]{conway_course_2007}.
\end{proof}

\begin{definition}[{\cite[Definition X.1.5]{conway_course_2007}}]
	\label{def:unbounded_ajoint}
	Let $\H,\mathcal{K}$ be separable Hilbert spaces and $A:\H\rightarrow \mathcal{K}$ be a densely defined linear operator on $\H$. Then we define 
	\begin{align*}
		\dom{A^*}\coloneqq\{k\in \mathcal{K}:\ h\mapsto \langle Ah,k\rangle \text{ is a bounded linear functional on }\dom{A}\}.
	\end{align*}
	As $\dom{A}\subset \H$ is dense, if $k\in \mathcal{K}$, there exists by the Riesz representation theorem some $f\in \H$ such that $\langle Ah,k\rangle = \langle h,f\rangle $ for all $h\in \H$. We define $A^*:\dom{A^*}\rightarrow \H$ by setting $A^*k=f$.
\end{definition}

\begin{lemma}
	\label{lemma:adjoint_of_densely_defined_operators}
	Let $\H$ be a separable Hilbert space. If $A,B$, and $AB:\H\rightarrow \H$ are densely defined, then
	\begin{enumerate}
		\item 
			\label{item:adjoint_of_densely_defined_operators_1}
			$(AB)^*\supset B^*A^*$,
		\item 
			\label{item:adjoint_of_densely_defined_operators_2}
			If $B^*A^*$ is bounded, then $(AB)^*=B^*A^*$.
	\end{enumerate}
\end{lemma}

\begin{proof}
	Let $k\in\dom{B^*A^*}$ and $h\in\dom{AB}$. Since $k\in\dom{A^*}$ and $Bh\in\dom{A}$, $\langle ABh,k\rangle=\langle Bh,A^*k\rangle.$ Since $A^*k\in\dom{B^*}$ and $h\in\dom{B}$, $\langle Bh,A^*k\rangle=\langle h,B^*A^*k\rangle.$ Thus $\langle ABh,k\rangle=\langle h,B^*A^*k\rangle$. Hence $h\mapsto \langle ABh,k\rangle$ is bounded and $(AB)^*k = B^*A^*k$, proving part \ref{item:adjoint_of_densely_defined_operators_1}. If $B^*A^*\in\B(\H)$, then $\dom{(AB)^*}\subset \H= \dom{B^*A^*}$, showing part \ref{item:adjoint_of_densely_defined_operators_2}.
\end{proof}

\begin{definition}[{\cite[Definitions X.2.1 and X.2.3]{conway_course_2007}}]
	\label{def:unbounded_self_adjoint}
	Let $\H$ be a separable Hilbert space. A densely defined operator $A:\H\rightarrow \H$ is said to be symmetric if $\langle Ah,k\rangle = \langle h,Ak\rangle $ for all $h,k\in\dom{A}$.
	If $A=A^*$, then $A$ is said to be self-adjoint.
\end{definition}

\begin{remark}
	Note that $A=A^*$ if and only if $A$ is symmetric and additionally $\dom{A}=\dom{A^*}$ holds.
\end{remark}

\begin{lemma}[{\cite[Proposition X.2.4]{conway_course_2007}}]
	\label{lemma:symmetric_operators}
	Let $\H$ be a separable Hilbert space and $A$ be a symmetric operator on $\H$. 
	\begin{enumerate}
		\item 
			\label{item:symmetric_operators_1}
			If $\ran{A}$ is dense, then $A$ is injective.
		\item 
			\label{item:symmetric_operators_2}
			If $A=A^*$ and $A$ is injective, then $\ran{A}$ is dense and $A^{-1}$ is well-defined on $\ran{A}$ and self-adjoint.
		\item 
			\label{item:symmetric_operators_3}
			If $\dom{A}=\H$, then $A=A^*$ and $A\in \B(\H)$.
		\item 
			\label{item:symmetric_operators_4}
			If $\ran{A}=\H$, then $A=A^*$ and $A^{-1}\in \B(\H)$.
	\end{enumerate}
\end{lemma}

\begin{lemma}
	\label{lemma:covariance_properties}
	Let $\mathcal{H}$ be a separable Hilbert space and $\mathcal{C}_1, \mathcal{C}_2\in L_1(\mathcal{H})_\R$ be nonnegative.
	If $\ran{\mathcal{C}_1^{1/2}}\subset\H$ densely, then the following hold.
	\begin{enumerate}
		\item $\mathcal{C}_1>0$ and $\mathcal{C}_1^{1/2}>0$.
			\label{item:covariance_properties_1}
		\item $\mathcal{C}_1^{-1/2}:\ran{\mathcal{C}_1^{1/2}}\rightarrow \H$ and $\mathcal{C}_1^{-1}:\ran{\mathcal{C}_1}\rightarrow \H$ are bijective and self-adjoint operators that are unbounded if $\dim{\H}$ is unbounded.
			\label{item:covariance_properties_2}
	\end{enumerate}
\end{lemma}

\begin{proof}
	By \Cref{lemma:symmetric_operators}\ref{item:symmetric_operators_1}, $\mathcal{C}_{1}^{1/2}$ and hence $\mathcal{C}_1$ are injective, so \ref{item:covariance_properties_1} holds. Furthermore, by \Cref{lemma:symmetric_operators}\ref{item:symmetric_operators_2}, $\mathcal{C}_{1}^{-1/2}$ and $\mathcal{C}_1^{-1}$ are bijective and self-adjoint. The inverse of a compact operator on its range in an infinite-dimensional space is unbounded, hence \ref{item:covariance_properties_2} holds.
\end{proof}

Condition \ref{item:fh_ranges} of the Feldman--Hajek theorem, \Cref{thm:feldman--hajek}, can be stated equivalently as follows.

\begin{lemma}
	\label{lemma:equivalent_cm_condition}
	Let $\H$ be a Hilbert space and $\mathcal{C}_1,\mathcal{C}_2\in\B(\H)$ injective. Then $\ran{\mathcal{C}_1^{1/2}}=\ran{\mathcal{C}_2^{1/2}}$ if and only if $\mathcal{C}_2^{-1/2}\mathcal{C}_1^{1/2}$ is a well-defined invertible operator in $\B(\H)$.
\end{lemma}
\begin{proof}
	Suppose that $\ran{\mathcal{C}_1^{1/2}} = \ran{\mathcal{C}_2^{1/2}}$. Then $\mathcal{C}_1^{-1/2}\mathcal{C}_2^{1/2}$ is well-defined and bijective. By \Cref{lemma:closed_operators}\ref{item:closed_inverse}, $\mathcal{C}_1^{-1/2}$ closed, being the inverse of a bounded, hence closed, operator. By \Cref{lemma:closed_operators}\ref{item:closed_composition}, $\mathcal{C}_1^{-1/2}\mathcal{C}_2^{1/2}$ is closed, and by \Cref{lemma:closed_graph_theorem}, it is bounded.
	Conversely, if $\mathcal{C}_1^{-1/2}\mathcal{C}_2^{1/2}\in\B(\H)$ is invertible, then,	
	\begin{align*}
		\ran{\mathcal{C}_1^{1/2}} &= \{\mathcal{C}_1^{1/2}h:h\in\H\}=\{\mathcal{C}_1^{1/2}\mathcal{C}_1^{-1/2}\mathcal{C}_2^{1/2}h:h\in\H\}
		=\{\mathcal{C}_2^{1/2}h:h\in\H\}=\ran{\mathcal{C}_2^{1/2}}.
	\end{align*}
\end{proof}

\section{Proofs of results}
\label{sec:proofs_of_results}

\subsection{Proofs of Section \ref{sec:equivalence_and_divergences_between_gaussian_measures}}
\label{subsec:proofs_for_formulation}

\expansionsFeldmanHajek*
\begin{proof}[Proof of \Cref{lemma:expansions_feldman_hajek}]
	By the Feldman--Hajek theorem, \Cref{thm:feldman--hajek}, $\ran{\mathcal{C}_1^{1/2}}=\ran{\mathcal{C}_2^{1/2}}$. Thus, by \Cref{lemma:equivalent_cm_condition}, $A\coloneqq\mathcal{C}_1^{-1/2}\mathcal{C}_2^{1/2}$ is a well-defined bounded and invertible operator, and by \Cref{thm:feldman--hajek}\ref{item:fh_covariance}, $AA^*-I$ is Hilbert--Schmidt. That is, there exists a sequence $(\lambda_i)_i\subset\ell^2(\R)$ and ONB $(w_i)_i$ of $\H$ such that,
	\begin{align*}
		AA^*-I = \sum_{i}^{}\lambda_i w_i\otimes w_i,
	\end{align*}
	i.e.,
	\begin{align}
		AA^*w_i = (1+\lambda_i)w_i.
		\label{eqn:first_pencil}
	\end{align}
	As $A$ is invertible in $\B(\H)$, so are $A^*$ and $AA^*$. Furthermore, $AA^*\geq 0$, hence $AA^*>0$ by \Cref{lemma:positive_is_injective_nonnegative}, which shows that $\lambda_i>-1$ for all $i$, which proves \cref{item:covariance_mix_1} holds. 

	By applying $A^{-1}$, $A^{-1}(A^{-1})^{*}A^{-1}$ and $(A^{-1})^{*}A^{-1}$ to \eqref{eqn:first_pencil} and rearranging, we obtain respectively,
	\begin{align}
		A^*AA^{-1}w_i &= (1+\lambda_i)A^{-1}w_i\label{eqn:second_pencil},\\
		A^{-1}(A^{-1})^{*}A^{-1}w_i &= \frac{1}{1+\lambda_i}A^{-1}w_i\label{eqn:third_pencil},\\
		(A^{-1})^{*}A^{-1}w_i &= \frac{1}{1+\lambda_i}w_i\label{eqn:fourth_pencil}.
	\end{align}
	By \eqref{eqn:fourth_pencil}, $v_i\coloneqq(1+\lambda_i)^{1/2}A^{-1}w_i$ satisfies,
	\begin{align*}
		\langle v_i,v_j\rangle = (1+\lambda_i)^{1/2}(1+\lambda_j)^{1/2}\langle (A^{-1})^*A^{-1}w_i,w_j\rangle = \delta_{ij},
	\end{align*}
	and, for all $h\in\H$,
	\begin{align*}
		\sum_{i}^{}\langle A^{-1}h,v_i\rangle v_i &= \sum_{i}^{}(1+\lambda_i)\langle h,(A^{-1})^*A^{-1}w_i\rangle A^{-1}w_i \\
		&= A^{-1}\sum_{i}^{}\langle h, w_i\rangle w_i\\
		&= A^{-1}h,
	\end{align*}
	where we used that $A^{-1}$ is continuous and $(w_i)_i$ is an ONB. Hence, $(v_i)_i$ is an ONB. Now, \eqref{eqn:second_pencil}, \eqref{eqn:third_pencil} and \eqref{eqn:fourth_pencil} become
	\begin{align*}
		(A^*A-I)v_i &= \lambda_iv_i,\\
		(A^{-1}(A^{-1})^{*}-I)v_i &= \frac{-\lambda_i}{1+\lambda_i}v_i,\\
		((A^{-1})^{*}A^{-1}-I)w_i &= \frac{-\lambda_i}{1+\lambda_i}w_i.
	\end{align*}
	Notice that $\frac{-\lambda_i}{1+\lambda_i}\in\ell^2( (-1,\infty))$, since $1+\lambda_i\rightarrow 1$ and $(\lambda_i)_i\in\ell^2( (-1,\infty))$. This proves \cref{item:covariance_mix_2,item:covariance_mix_3,item:covariance_mix_4}.

	Finally, we prove the statements about the domains of the leftmost operators in \cref{item:covariance_mix_1,item:covariance_mix_2,item:covariance_mix_3,item:covariance_mix_4}.
	By \Cref{lemma:covariance_properties}\ref{item:covariance_properties_2}, $\mathcal{C}_1^{-1/2}$ is self-adjoint. By \Cref{lemma:adjoint_of_densely_defined_operators}\ref{item:adjoint_of_densely_defined_operators_1}, $A^* \supset \mathcal{C}_2^{1/2}\mathcal{C}_1^{-1/2}$ and the latter operator is defined on $\dom{\mathcal{C}_1^{-1/2}}=\ran{\mathcal{C}_1^{1/2}}$ by the definition of composition of linear operators, c.f.\ \Cref{sec:notation}. This shows that the leftmost operator in \cref{item:covariance_mix_1}, and by symmetry also in \cref{item:covariance_mix_3}, is defined on the dense subspace $\ran{\mathcal{C}_1^{1/2}}=\ran{\mathcal{C}_2^{1/2}}$. Since $A$ is boundedly invertible, $\overline{A^{-1}(\mathcal{D})} = A^{-1}(\overline{\mathcal{D}})=\mathcal{H}$ for any dense set $\mathcal{D}\subset\mathcal{H}$. This shows that $A^{-1}(\dom{\mathcal{C}_2^{1/2}\mathcal{C}_1^{-1/2}})$ is dense in $\mathcal{H}$. Since $\dom{\mathcal{C}_2^{1/2}\mathcal{C}_1^{-1/2}\mathcal{C}_1^{-1/2}\mathcal{C}_2^{1/2}}=A^{-1}(\dom{\mathcal{C}_2^{1/2}\mathcal{C}_1^{-1/2}})$, this proves that the leftmost operator in \cref{item:covariance_mix_2}, and by symmetry also in \cref{item:covariance_mix_4}, is densely defined.
\end{proof}

\bayesianFeldmanHajek*
\begin{proof}[Proof of \Cref{prop:bayesian_feldman_hajek}]
	By \Cref{lemma:expansions_feldman_hajek} with $\mathcal{C}_1 \leftarrow \mathcal{C}_\pr$ and $\mathcal{C}_2 \leftarrow \mathcal{C}_\pos$, there exists an eigenvalue sequence $(\lambda_i)_{i\in\N}\subset\ell^2((-1,\infty))$ and ONBs $(w_i)_{i\in\N}$ and $(v_i)_i$ of $\H$ such that $v_i=\sqrt{1+\lambda_i}\mathcal{C}_\pos^{-1/2}\mathcal{C}_\pr^{1/2}w_i$ and \Cref{lemma:expansions_feldman_hajek}\ref{item:covariance_mix_1},\ref{item:covariance_mix_2},\ref{item:covariance_mix_3},\ref{item:covariance_mix_4} hold. In particular, by \Cref{lemma:expansions_feldman_hajek}\ref{item:covariance_mix_1} and the definition of $R(\cdot\Vert\cdot)$ in \eqref{eqn:feldman_hajek_operator}, $R(\mathcal{C}_\pos\Vert\mathcal{C}_\pr)=\sum_{i}^{}\lambda_i w_i\otimes w_i$.
	By \eqref{eqn:pos_precision}, it holds on $\ran{\mathcal{C}_\pr^{1/2}}$,
	\begin{align*}
		\mathcal{C}_{\pr}^{1/2}\mathcal{C}_{\pos}^{-1}\mathcal{C}_{\pr}^{1/2}-I = \mathcal{C}_{\pr}^{1/2}(\mathcal{C}_{\pr}^{-1}+H)\mathcal{C}_{\pr}^{1/2} - I =\mathcal{C}_{\pr}^{1/2}\mathcal{C}_{\pr}^{-1}\mathcal{C}_{\pr}^{1/2} +\mathcal{C}_{\pr}^{1/2}H\mathcal{C}_{\pr}^{1/2}-I.
	\end{align*}
	Now $\mathcal{C}_{\pr}^{1/2}H\mathcal{C}_{\pr}^{1/2}-I\in\mathcal{B}(\mathcal{H})$, hence it is defined on all of $\H$. The operator $\mathcal{C}_{\pr}^{1/2}\mathcal{C}_{\pr}^{-1}\mathcal{C}_{\pr}^{1/2}$ is extended by the identity operator $I$. Thus,
	$\mathcal{C}_{\pr}^{1/2}\mathcal{C}_{\pos}^{-1}\mathcal{C}_{\pr}^{1/2}-I\subset \mathcal{C}_{\pr}^{1/2}H\mathcal{C}_{\pr}^{1/2}.$
	By the uniqueness of extensions of continuous functions on the dense set $\ran{\mathcal{C}_\pr^{1/2}}$, this implies together with \Cref{lemma:expansions_feldman_hajek}\ref{item:covariance_mix_4} that \eqref{eqn:prior_preconditioned_Hessian} holds.
	The proof of \eqref{eqn:posterior_preconditioned_Hessian} is similar: by \eqref{eqn:pos_precision}, it holds on $\ran{\mathcal{C}_\pos^{1/2}}=\ran{\mathcal{C}_{\pr}^{1/2}}$,
	\begin{align*}
		\mathcal{C}_\pos^{1/2}\mathcal{C}_\pr^{-1}\mathcal{C}_\pos^{1/2}-I=\mathcal{C}_\pos^{1/2}(\mathcal{C}_\pr^{-1}-\mathcal{C}_\pos^{-1})\mathcal{C}_\pos^{1/2} = \mathcal{C}_\pos^{1/2}(-H)\mathcal{C}_\pos^{1/2},
	\end{align*}
	and combining this with \Cref{lemma:expansions_feldman_hajek}\ref{item:covariance_mix_2}, \eqref{eqn:posterior_preconditioned_Hessian} follows by uniqueness of the extension. 	

	We now prove the stated properties of the eigenvalues and eigenvectors.
	Recall that by \eqref{eqn:hessian}, $H\in\B_{00,n}(\H)$ is self-adjoint and non-negative. Hence the operator $\mathcal{C}_{\pr}^{1/2}H\mathcal{C}_{\pr}^{1/2}=(\mathcal{C}_{\pr}^{1/2}H^{1/2})(\mathcal{C}_{\pr}^{1/2}H^{1/2})^*$ is also self-adjoint and non-negative, which implies that $(\tfrac{-\lambda_i}{1+\lambda_i})_{i\in\N}\subset \ell^2((-1,0])$, and thus that $(\lambda_i)_{i\in\N}\in\ell^2((-1,0])$. We thus may order $(\lambda_i)_i$ in a nondecreasing manner. Since $\mathcal{C}_\pr$ is injective on $\H$, it follows by applying \Cref{lemma:range_of_square_of_finite_rank} twice with $A\leftarrow \mathcal{C}_\pr^{1/2}H^{1/2}$ and $A\leftarrow H^{1/2}$ that 
	\begin{align*}
	\rank{\mathcal{C}_{\pr}^{1/2}H\mathcal{C}_{\pr}^{1/2}}&= \rank{\mathcal{C}_\pr^{1/2}H^{1/2}(\mathcal{C}_\pr^{1/2}H^{1/2})^*}=\rank{\mathcal{C}_\pr^{1/2}H^{1/2}}\\&= \rank{H^{1/2}}=\rank{H}.
	\end{align*}
	Therefore, $(\tfrac{-\lambda_i}{1+\lambda_i})_{i\in\N}$ contains exactly $\rank{H}\leq n$ many nonzero entries.
	It follows directly from \eqref{eqn:prior_preconditioned_Hessian}, \eqref{eqn:posterior_preconditioned_Hessian} and the fact that $\lambda_i\not=0$ for $i\leq\rank{H}$, that $w_i\in\ran{\mathcal{C}_\pr^{1/2}H\mathcal{C}_{\pr}^{1/2}}\subset \ran{\mathcal{C}_\pr^{1/2}}$ and $v_i\in\ran{\mathcal{C}_\pos^{1/2}H\mathcal{C}_\pos^{1/2}}\subset\ran{\mathcal{C}_\pos^{1/2}}=\ran{\mathcal{C}_\pr^{1/2}}$ for $i\leq \rank{H}$. By \Cref{lemma:extension_of_finite_basis_in_dense_subspace}, we can extend $(w_i)_{i=1}^{\rank{H}}$ to an ONB $(w_i')_i$ of $\mathcal{H}$ with $(w_i')_i\subset\ran{\mathcal{C}_\pr^{1/2}}$ and $w_i'=w_i$ for $i\leq \rank{H}$. We now replace $w_i$ by $w_i'$ and $v_i$ by $\mathcal{C}_\pos^{-1/2}\mathcal{C}_\pr^{1/2}w_i'$ for $i>\rank{H}$. After this replacement, the equations \eqref{eqn:prior_preconditioned_Hessian} and \eqref{eqn:posterior_preconditioned_Hessian} and $v_i=\sqrt{1+\lambda_i}\mathcal{C}_\pos^{-1/2}\mathcal{C}_\pr^{1/2}w_i$ for all $i$ remain valid, and we now have $w_i,v_i\in\ran{\mathcal{C}_\pr^{1/2}}$ for all $i$.

	By \Cref{lemma:expansions_feldman_hajek}\ref{item:covariance_mix_1} and the fact that $(w_i)_i$ lies in the Cameron--Martin space, it follows that
	\begin{align*}
		\mathcal{C}_{\pr}^{-1/2}\mathcal{C}_{\pos}\mathcal{C}_{\pr}^{-1/2}w_i=(1+\lambda_i)w_i,\quad i\in\N.
	\end{align*}
	Applying $\mathcal{C}_{\pos}^{-1/2}\mathcal{C}_{\pr}^{1/2}$ to both sides of the equation yields \eqref{eqn:bayesian_cov_pencil}.
\end{proof}

\gaussianDivergences*
\begin{proof}[Proof of \Cref{thm:gaussian_divergences}]
	We use the expressions for the KL and R\'enyi divergence of \cite[Theorem 14, Theorem 15]{minh_regularized_2021}. While they are stated for infinite-dimensional Hilbert spaces only, it is noted in \cite{Minh2022} that these expressions also hold for finite-dimensional Hilbert spaces; see the remarks after \cite[Theorem 3]{Minh2022}. 
	By \Cref{lemma:adjoint_of_densely_defined_operators}\ref{item:adjoint_of_densely_defined_operators_1}, $(\mathcal{C}_1^{-1/2}\mathcal{C}_2^{1/2})^*=\mathcal{C}_2^{1/2}\mathcal{C}_1^{-1/2}$ on $\ran{\mathcal{C}_1^{1/2}}$.
	The statements in the theorem now follow immediately from the expressions in \cite[Theorem 14, Theorem 15]{minh_regularized_2021}, because for $S\coloneqq -R(\mathcal{C}_2\Vert \mathcal{C}_1)\in L_2(\mathcal{H})_\R$, where $R(\cdot\Vert\cdot)$ is defined in \eqref{eqn:feldman_hajek_operator}, we have
	\begin{align*}
		\mathcal{C}_1^{1/2}(I-S)\mathcal{C}_1^{1/2} = \mathcal{C}_1^{1/2}(\mathcal{C}_1^{-1/2}\mathcal{C}_2^{1/2})(\mathcal{C}_1^{-1/2}\mathcal{C}_2^{1/2})^*\mathcal{C}_1^{1/2} = \mathcal{C}_1^{1/2}\mathcal{C}_1^{-1/2}\mathcal{C}_2\mathcal{C}_1^{-1/2}\mathcal{C}_1^{1/2} = \mathcal{C}_2,
	\end{align*}
	and $I-(1-\rho) S = \rho I + (1-\rho)(I+R(\mathcal{C}_2\Vert \mathcal{C}_1))$ for $0\leq \rho\leq 1$.
\end{proof}

\subsection{Proofs of Section \ref{sec:optimal_approximation_covariance}}
\label{subsec:proofs_for_optimal_approximation_covariance}

\finiteLoss*
\begin{proof}[Proof of \Cref{lemma:finite_loss}]
	Given that $xf'(x)>0$ for $x\neq 0$, it follows that $f'(x)<0$ for $x<0$ and $f'(x)>0$ for $x>0$. This implies, by continuity of $f$, that $f'(0)=0$. Hence $f$ has a global minimum only at $x=0$ and $f\geq 0$. Thus, also $\mathcal{L}_f\geq 0$.
	By the Lipschitz assumption on $f'$ at 0, there exists $\varepsilon\in(0,1)$ and $M_0>0$ such that $f'(x)=f'(x)-f'(0)\leq M_0\abs{x}$ for $\abs{x}\leq \varepsilon$. For $\abs{y}\leq \varepsilon$,
	\begin{align*}
		f(y) = f(y)-f(0) = \int_{0}^{y}f'(x)\d x\leq \int_{0}^{y}M_0\abs{x}\d x = \frac{1}{2}M_0y^2.
	\end{align*}
	Let $(x_i)_i\in\ell^2( (-1,\infty))$. For $N$ large enough, its tail $(x_i)_{i>N}$ lies in $(-\varepsilon,\varepsilon)$, so that the inequality above implies $\sum_{i>N}f(x_i)\leq \tfrac{1}{2}M\sum_{i>N}x_i^2\leq \tfrac{1}{2}M\norm{x}_{\ell^2}^2$. For $\mathcal{C}_1$, $\mathcal{C}_2\in\mathcal{E}$ we have $R(\mathcal{C}_2\Vert \mathcal{C}_1)\in L_2(\mathcal{H})$ and its eigenvalue sequence is square-summable. Hence $\mathcal{L}_f<\infty$ by the definition of $\mathcal{L}_f$ in \eqref{eqn:definition_spectral_f}. This proves \cref{item:properties_spectral_f}. For \cref{item:closed_under_transform}, we note $f(\eta(0))=f(0)=0$. Furthermore, we compute $\eta'(x)=-(1+x)^{-2}$ and, by the fact that $f\in\mathscr{F}$, $x(f\circ \eta)'(x)=\frac{1}{1+x}\frac{-x}{1+x}f'(\frac{-x}{1+x})>0$ for $x\not=0$. By the assumption on $f$, $\lim_{x\rightarrow\infty}f(\eta(x))=\lim_{x\rightarrow-1}f(x)=\infty$. Finally, $\eta$ is smooth, so $\eta$ and $\eta'$ are Lipschitz at 0. Therefore, $f'\circ \eta$ is Lipschitz at 0 as the composition of Lipschitz functions at 0, and $(f\circ \eta)'=(f'\circ \eta)\eta'$ is Lipschitz at 0 as the product of two Lipschitz functions at 0.
\end{proof}

\divergencesInLossClass*
\begin{proof}[Proof of \Cref{lemma:divergences_in_loss_class}]
	Notice that $f_\kl$, $f_{\ren,\rho}\in\mathcal{C}^\infty(\R)$, which implies $f_{\kl}'$ and $f_{\ren,\rho}'$ are locally Lipschitz on $(-1,\infty)$. We compute $xf_\kl'(x)=\frac{x^2}{2(1+x)}>0$ for $x\not=0$ and
	\begin{align*}
		f'_{\ren,\rho}(x)=&\frac{1}{2\rho(1-\rho)}\left[\frac{\rho-1}{1+x}+\frac{1-\rho}{\rho+(1-\rho)(1+x)}\right]=\frac{x}{2(1+x)[\rho+(1-\rho)(1+x)]}.
	\end{align*}
	Hence $xf'_{\ren,\rho}(x)>0$ for $x\not=0$.
	Furthermore, $f_\kl(0)=0=f_{\ren,\rho}(0)$ and $\lim_{x\rightarrow\infty}f_\kl(x)=\infty=\lim_{x\rightarrow\infty}f_{\ren,\rho}(x)$, so $f_\kl,f_{\ren,\rho}\in\mathscr{F}$. 

	The first equations in \cref{item:f_for_kl_divergence,item:f_for_renyi_divergence} follow from \Cref{thm:gaussian_divergences}.
	With $(\lambda_i)_{i\in\N}$ the eigenvalues of $R(\mathcal{C}_2\Vert\mathcal{C}_1)\in L_2(\H)$, it holds by \eqref{eqn:determinant_for_L_2} that
	\begin{align*}
		\det\nolimits_{2}(I+R(\mathcal{C}_2\Vert\mathcal{C}_1))= \prod_{i\in\N} (1+\lambda_i)\exp( -\lambda_i )=\prod_{i\in\N}\exp(-2 f_{\kl}(\lambda_i))=\exp\biggr(-2\sum_{i\in\N} f_{\kl}(\lambda_i)\biggr),
	\end{align*}
	which proves that $D_{\kl}(\nu\vert\vert\mu) = \sum_{i\in\N} f_\kl(\lambda_i)=\mathcal{L}_{f_{\kl}}(\mathcal{C}_2\Vert\mathcal{C}_1)$. Hence \cref{item:f_for_kl_divergence} holds.

	By the spectral mapping theorem---see e.g.\ \cite[Theorem VII.1(e)]{Reed1980} for a version that does not assume that $\mathcal{H}$ is defined over the complex field $\C$---the eigenvalues of $I+R(\mathcal{C}_2\Vert\mathcal{C}_1)$ are $(1+\lambda_i)_i$, and the eigenvalues of $A_\rho\coloneqq\bigl(I+R(\mathcal{C}_2\Vert \mathcal{C}_1)\bigr)^{\rho-1}\bigl(\rho I + (1-\rho)(I+R(\mathcal{C}_2\Vert \mathcal{C}_1))\bigr)$ are $(\gamma_i)_i$, with $\gamma_i\coloneqq (1+\lambda_i)^{\rho-1}(\rho+(1-\rho)(1+\lambda_i))$, $i\in\N$. The eigenvalues of $A_\rho-I$ are then $(\gamma_i-1)_i$ and by \eqref{eqn:determinant_for_L_1},
	\begin{align*}
		\det(A_\rho) &=  \det(I+(A_\rho-I)) = \prod_i(1+(\gamma_i-1)) \\
		&= \exp\left(\sum_i\log(\gamma_i)\right)= \exp\left( 2\rho(1-\rho) \sum_{i}^{}f_{\ren,\rho}(\lambda_i)\right).
	\end{align*}
	This shows \cref{item:f_for_renyi_divergence} holds.
	Since $\lim_{x\rightarrow -1}f_\kl(x)=\infty=\lim_{x\rightarrow -1}f_{\ren,\rho}(x)$, \cref{item:f_for_reverse_divergence} now follows directly from \cref{item:f_for_kl_divergence,item:f_for_renyi_divergence} and \Cref{lemma:finite_loss}\ref{item:closed_under_transform}.
\end{proof}

\propertiesEquivalence*
\begin{proof}[Proof of \Cref{lemma:properties_equivalence}]
	\ref{item:properties_equivalence_1}\quad If $\ran{K}\subset\ran{\mathcal{C}_1^{1/2}}=\dom{\mathcal{C}_1^{-1/2}}$, then $\mathcal{C}_1^{-1/2}K$ is well-defined in $\B(\H)$ and thus so is $\X$. By \Cref{lemma:adjoint_of_densely_defined_operators}\ref{item:adjoint_of_densely_defined_operators_1}, $(\mathcal{C}_1^{-1/2}K)^* = K^*\mathcal{C}_1^{-1/2}$ on $\ran{\mathcal{C}_1^{1/2}}$, whence $(\mathcal{C}_1^{-1/2}K)^*\mathcal{C}_1^{1/2}=K^*$ and $\mathcal{C}_1^{1/2}\X\mathcal{C}_1^{1/2}=\mathcal{C}_1-\mathcal{C}_1^{1/2}\mathcal{C}_1^{-1/2}K(\mathcal{C}_1^{-1/2}K)^*\mathcal{C}_1^{1/2}=\mathcal{C}$.

	\ref{item:properties_equivalence_2}\quad If $\mathcal{C}=\mathcal{C}_1-KK^*\geq 0$, then $\langle KK^*h,h\rangle \leq \langle \mathcal{C}_1 h,h\rangle $ for all $h\in\H$. Hence $\norm{K^*h}\leq \norm{\mathcal{C}_1^{1/2}h}$ for all $h\in\H$. By \Cref{lemma:range_inclusions}, $\ran{K}\subset\ran{\mathcal{C}_1^{1/2}}$. By \cref{item:properties_equivalence_1}, $\X$ is well-defined in $\B(\H)$ and $\mathcal{C}=\mathcal{C}_1^{1/2}\X\mathcal{C}_1^{1/2}$. Furthermore, $\langle \X\mathcal{C}_1^{1/2}h,\mathcal{C}_1^{1/2}h\rangle= \langle \mathcal{C}h,h\rangle\geq 0$ for all $h\in\H$. Since $\ran{\mathcal{C}_1^{1/2}}\subset\H$ densely, it follows that $\langle \X h,h\rangle\geq 0$ for all $h\in\H$. Conversely, if $\ran{K}\subset\ran{\mathcal{C}_1^{1/2}}$ and $\X\geq 0$, then using \cref{item:properties_equivalence_1} we find $\langle \mathcal{C}h,h\rangle=\langle \X\mathcal{C}_1^{1/2} h,\mathcal{C}_1^{1/2}h\rangle \geq 0$ for $h\in\H$.

	\ref{item:properties_equivalence_3}\quad The implication \ref{item:properties_equivalence_3_1}$\Rightarrow$\ref{item:properties_equivalence_3_3} holds by definition of $\mathcal{E}(m_1,\mathcal{C}_1)$ in \eqref{eqn:equivalent_covariances_general} and \Cref{thm:feldman--hajek}\ref{item:fh_ranges}. 

	Now, \ref{item:properties_equivalence_3_3} implies \ref{item:properties_equivalence_3_2}. Indeed, $\ran{\mathcal{C}^{1/2}}^\perp=\ker{\mathcal{C}^{1/2}}$ and $\ran{\mathcal{C}_1^{1/2}}^\perp=\ker{\mathcal{C}_1^{1/2}}=\{0\}$ by \Cref{lemma:kernel_range} and injectivity of $\mathcal{C}_1^{1/2}$. Furthermore, $\ker{\mathcal{C}^{1/2}}=\ker{\mathcal{C}}$ by \Cref{lemma:kernel_of_square}. Thus, if \ref{item:properties_equivalence_3_3} holds, then $\mathcal{C}$ is nonnegative and injective, hence positive, by \Cref{lemma:positive_is_injective_nonnegative}. 

	Next, we show that \ref{item:properties_equivalence_3_2}$\Rightarrow$\ref{item:properties_equivalence_3_1}. By \eqref{eqn:equivalent_covariances_general} and \Cref{thm:feldman--hajek}, we only need to show that $\mathcal{C}$ is trace-class and that $\mathcal{C}_1^{-1/2}\mathcal{C}^{1/2}(\mathcal{C}_1^{-1/2}\mathcal{C}^{1/2})^*-I$ is Hilbert--Schmidt. Since $\mathcal{C}_1\in L_1(\H)_\R$ and $KK^*\in\B_{00,r}(\H)\subset L_1(\H)_\R$, also $\mathcal{C}\in L_1(\H)_\R$. 
	By \Cref{lemma:adjoint_of_densely_defined_operators}\ref{item:adjoint_of_densely_defined_operators_2} $(\mathcal{C}_1^{-1/2}\mathcal{C}^{1/2})^*=\mathcal{C}^{1/2}\mathcal{C}_1^{-1/2}$ on $\ran{\mathcal{C}_1^{1/2}}$.
	Therefore, 
	\begin{align*}
		(\mathcal{C}_1^{-1/2}\mathcal{C}^{1/2})(\mathcal{C}_1^{-1/2}\mathcal{C}^{1/2})^* - I = \mathcal{C}_1^{-1/2}\mathcal{C}\mathcal{C}_1^{-1/2} - I = \X-I = -(\mathcal{C}_1^{-1/2}K)(\mathcal{C}_1^{-1/2}K)^*.
	\end{align*}
	The outermost operators $(\mathcal{C}_1^{-1/2}\mathcal{C}^{1/2})(\mathcal{C}_1^{-1/2}\mathcal{C}^{1/2})^* - I$ and $(\mathcal{C}_1^{-1/2}K)(\mathcal{C}_1^{-1/2}K)^*$ are bounded and defined on all of $\mathcal{H}$. Since $\ran{\mathcal{C}_1^{1/2}}\subset\mathcal{H}$ densely, it follows that $(\mathcal{C}_1^{-1/2}\mathcal{C}^{1/2})(\mathcal{C}_1^{-1/2}\mathcal{C}^{1/2})^* - I = -(\mathcal{C}_1^{-1/2}K)(\mathcal{C}_1^{-1/2}K)^*$ on $\H$.
	Since $K$ has finite rank, so does $(\mathcal{C}_1^{-1/2}K)(\mathcal{C}_1^{-1/2}K)^*$.
	We conclude that $(\mathcal{C}_1^{-1/2}\mathcal{C}^{1/2})(\mathcal{C}_1^{-1/2}\mathcal{C}^{1/2})^* - I\in L_2(\mathcal{H})_\R$, so that \ref{item:properties_equivalence_3_2}$\Rightarrow$\ref{item:properties_equivalence_3_1}. 

	The equivalence of \ref{item:properties_equivalence_3_3} and \ref{item:properties_equivalence_3_4} follows from \cref{item:properties_equivalence_2}. 
	
	Finally, we show \ref{item:properties_equivalence_3_4} and \ref{item:properties_equivalence_3_5} are equivalent. Note that \ref{item:properties_equivalence_3_4} implies $\mathcal{C}>0$ by the already proven equivalence \ref{item:properties_equivalence_3_2}$\Leftrightarrow$\ref{item:properties_equivalence_3_4}. Also \ref{item:properties_equivalence_3_5} implies $\mathcal{C}>0$. Indeed, $\mathcal{C}\geq 0$ by \cref{item:properties_equivalence_2}, and $\mathcal{C}=\mathcal{C}_1^{1/2}\X\mathcal{C}_1^{1/2}$ is injective as a composition of injective maps, by \cref{item:properties_equivalence_1}. Thus, by \Cref{lemma:positive_is_injective_nonnegative}, $\mathcal{C}>0$ if \ref{item:properties_equivalence_3_5} holds.
	We therefore assume that $\ran{K}\subset\ran{\mathcal{C}_1^{1/2}}$, $\X\geq 0$ and $\mathcal{C}>0$, and show that $\X$ is invertible if and only if $\ran{\mathcal{C}_1^{1/2}}=\ran{\mathcal{C}^{1/2}}$.
	Since $\X\geq 0$, $\X^{1/2}$ exists. By \ref{item:properties_equivalence_1}, $\mathcal{C}=\mathcal{C}_1^{1/2}\X\mathcal{C}_1^{1/2}=(\mathcal{C}_1^{1/2}\X^{1/2})(\mathcal{C}_1^{1/2}\X^{1/2})^*$.
	Thus, both $\mathcal{C}^{1/2}$ and $\mathcal{C}_1^{1/2}\X^{1/2}$ are (possibly non-self adjoint) square roots of $\mathcal{C}$. Now, $\mathcal{C}$ is self-adjoint and positive, hence $\overline{\ran{\mathcal{C}}}=\ker{\mathcal{C}}^\perp=\mathcal{H}$ by \Cref{lemma:kernel_range}. By \Cref{lemma:nonsymmetric_root_relation} applied with $A\leftarrow \mathcal{C}^{1/2}$ and $B\leftarrow \mathcal{C}_1^{1/2}\X^{1/2}$, there exists a Hilbert space isomorphism $Q\in\B(\H)$ such that $\mathcal{C}_1^{1/2}\X^{1/2}=\mathcal{C}^{1/2}Q$. From this we conclude two facts.
	On the one hand, if $\ran{\mathcal{C}^{1/2}}=\ran{\mathcal{C}_1^{1/2}}$, so that $\mathcal{C}_1^{-1/2}\mathcal{C}^{1/2}$ is boundedly invertible by \Cref{lemma:equivalent_cm_condition}, then this implies that $\X^{1/2} = \mathcal{C}_1^{-1/2}\mathcal{C}^{1/2}Q$ is boundedly invertible. Hence $\X$ is boundedly invertible as the composition of boundedly invertible operators.
	On the other hand, if $\X$ and hence $\X^{1/2}$ is boundedly invertible, then $\ran{\mathcal{C}^{1/2}}=\ran{\mathcal{C}^{1/2}Q}=\ran{\mathcal{C}_1^{1/2}\X^{1/2}}=\ran{\mathcal{C}_1^{1/2}}$. We conclude that \ref{item:properties_equivalence_3_4} and \ref{item:properties_equivalence_3_5} are equivalent.

	\ref{item:properties_equivalence_4}\quad Suppose that $\mathcal{C}\geq0$. By \cref{item:properties_equivalence_2,item:properties_equivalence_1}, $\ran{K}\subset\ran{\mathcal{C}_1^{1/2}}$ and $\X\coloneqq I-(\mathcal{C}_1^{-1/2}K)(\mathcal{C}_1^{-1/2}K)^*$ is a well-defined and nonnegative operator. By injectivity of $\mathcal{C}_1^{-1/2}$, we have that $\mathcal{C}_1^{-1/2}K$ and $K$ have the same rank. Thus, by \Cref{lemma:range_of_square_of_finite_rank}, $(\mathcal{C}_1^{-1/2}K)(\mathcal{C}_1^{-1/2}K)^*$ and $K$ have the same rank. We then diagonalise the nonnegative and self-adjoint operator $I-\X=(\mathcal{C}_1^{-1/2}K)(\mathcal{C}_1^{-1/2}K)^*$ as $\sum_{i=1}^{\rank{K}}d_i^2e_i\otimes e_i$, where $d_i^2\geq d_{i+1}^2>0$ and $(e_i)_i$ is an orthonormal sequence in $\H$. By nonnegativity of $\X$, we have $1-d_i^2 = \langle \X e_i,e_i\rangle\geq 0$, that is, $d_i^2\in(0,1]$, for each $i\leq \rank{K}$.
	
	Furthermore, $\X$ is invertible if and only if $d_i^2\not=1$ for each $i\leq\rank{K}$ by \Cref{lemma:inverse_of_self_adj_hilbert_schmidt_perturbation}\ref{item:inverse_of_self_adj_hilbert_schmidt_perturbation_1} applied with $\delta_i\leftarrow -d_i^2$ for $i\leq \rank{K}$ and $\delta_i\leftarrow 0$ otherwise. Using \ref{item:properties_equivalence_3_5} of \cref{item:properties_equivalence_3}, it follows that the equivalent statements of \cref{item:properties_equivalence_3} hold if and only if $(d_i^2)_i\subset(0,1)$.

	Suppose the additional assumptions $\mathcal{C}>0$ and  $\ran{K}\subset\ran{\mathcal{C}_1}$ hold. The latter assumption implies $\ran{\mathcal{C}_1^{-1/2}K}\subset\ran{\mathcal{C}_1^{1/2}}$, which in turn implies $\ran{(\mathcal{C}_1^{-1/2}K)(\mathcal{C}_1^{-1/2}K)^*}\subset\ran{\mathcal{C}_1^{1/2}}$. Thus, $(e_i)_{i=1}^{\rank{K}}\subset\ran{\mathcal{C}_1^{1/2}}$. Hence, the assumption $\mathcal{C}>0$ and the fact that $\mathcal{C}=\mathcal{C}_1^{1/2}\X\mathcal{C}_1^{1/2}$ from \cref{item:properties_equivalence_1} show $1-d_i^2=\langle \X e_i,e_i\rangle=\langle \X\mathcal{C}_1^{1/2}\mathcal{C}_1^{-1/2}e_i,\mathcal{C}_1^{1/2}\mathcal{C}_1^{-1/2}e_i\rangle=\langle \mathcal{C}\mathcal{C}_1^{-1/2}e_i,\mathcal{C}_1^{-1/2}e_i\rangle>0$, showing $d_i^2<1$ for each $i\leq\rank{K}$.
\end{proof}

\propRangeOfK*

\begin{proof}[Proof of \Cref{prop:range_of_K}]
	\ref{item:range_of_K_1}\quad Suppose that $\mathcal{C}=\mathcal{C}_1-KK^*$ for some $K\in\B(\R^r,\H)$ and $\mathcal{C}\in\mathcal{E}(m_1,\mathcal{C}_1)$. By the implication \ref{item:properties_equivalence_3_1}$\Rightarrow$\ref{item:properties_equivalence_3_2} and \ref{item:properties_equivalence_3_1}$\Rightarrow$\ref{item:properties_equivalence_3_5} of \Cref{lemma:properties_equivalence}\ref{item:properties_equivalence_3}, we have that $\mathcal{C}>0$, hence $\mathcal{C}$ is injective, and that $\X\coloneqq I-(\mathcal{C}_1^{-1/2}K)(\mathcal{C}_1^{-1/2}K)^*$ is a well-defined nonnegative, self-adjoint, and invertible operator. 
	We diagonalise $\X$ as $I-\sum_{i=1}^{\rank{K}}d_i^2e_i\otimes e_i$ by \Cref{lemma:properties_equivalence}\ref{item:properties_equivalence_4}, where $(e_i)_i\subset\mathcal{H}$ is orthonormal and $(d_i^2)_i\subset(0,1)$ is nonincreasing.
	By \Cref{lemma:properties_equivalence}\ref{item:properties_equivalence_1}, $\mathcal{C}=\mathcal{C}_1^{1/2}\X\mathcal{C}_1^{1/2}$, which is the composition of three injective maps. 
	Using \Cref{lemma:inverse_of_self_adj_hilbert_schmidt_perturbation} with $\delta_i\leftarrow -d_i^2$ for $i\leq \rank{K}$ and $\delta_i\leftarrow 0$ otherwise, the inverse of $\mathcal{C}$ on $\ran{\mathcal{C}}$ is given by
	\begin{align*}
		\mathcal{C}^{-1} = \mathcal{C}_1^{-1/2}\X^{-1}\mathcal{C}_1^{-1/2} = \mathcal{C}_1^{-1/2}\left(I+\sum_{i=1}^{\rank{K}}\frac{d_i^2}{1-d_i^2}e_i\otimes e_i\right)\mathcal{C}_1^{-1/2} = \mathcal{C}_1^{-1/2}(I+ZZ^*)\mathcal{C}_1^{-1/2},
	\end{align*}
	where $Z\coloneqq \sum_{i=1}^{\rank{K}}\sqrt{\frac{d_i^2}{1-d_i^2}}e_i\otimes \varphi_i$ for any choice of ONB $(\varphi_i)_i$ of $\R^r$. Since $(d_i^2)_i\in(0,1)$, we have $\rank{Z}=\rank{K}$.

	Conversely, suppose that $\mathcal{C}$ is injective and $\mathcal{C}^{-1} = \mathcal{C}_1^{-1/2}(I+ZZ^*)\mathcal{C}_1^{-1/2}$ on $\ran{\mathcal{C}}$ for some $Z\in\B(\R^r,\H)$. Since $I+ZZ^*\geq I$, $I+ZZ^*$ is invertible. Thus, $\mathcal{C}^{-1}$ is the composition of three injective operators, and we can invert $\mathcal{C}^{-1}$ on $\ran{\mathcal{C}^{-1}}=\H$ to obtain
	\begin{align*}
		\mathcal{C}=(\mathcal{C}^{-1})^{-1} = \left( \mathcal{C}_1^{-1/2}(I+ZZ^*)\mathcal{C}_1^{-1/2} \right)^{-1} = \mathcal{C}_1^{1/2}(I+ZZ^*)^{-1}\mathcal{C}_1^{1/2}.
	\end{align*}
	By \Cref{lemma:range_of_square_of_finite_rank}, $\rank{ZZ^*}=\rank{Z}$, and we diagonalise $ZZ^*=\sum_{i=1}^{\rank{Z}}b_i^2g_i\otimes g_i$ for $b_i^2\geq b_{i+1}^2>0$ and $(g_i)_i$ an orthonormal sequence in $\H$. Then, by \Cref{lemma:inverse_of_self_adj_hilbert_schmidt_perturbation} applied with $\delta_i\leftarrow b_i^2$ for $i\leq\rank{Z}$ and $\delta_i\leftarrow 0$ otherwise, it follows that $(I+ZZ^*)^{-1}=I-\sum_{i=1}^{\rank{Z}}\frac{b_i^2}{1+b_i^2}g_i\otimes g_i$ and
	\begin{align*}
		\mathcal{C} = \mathcal{C}_1^{1/2}\left(I-\sum_{i=1}^{\rank{Z}}\frac{b_i^2}{1+b_i^2}g_i\otimes g_i\right)\mathcal{C}_1^{1/2} = \mathcal{C}_1 - \mathcal{C}_1^{1/2}\left(\sum_{i=1}^{\rank{Z}}\frac{b_i^2}{1+b_i^2}g_i\otimes g_i\right)\mathcal{C}_1^{1/2}.
	\end{align*}
	We see that $\mathcal{C}=\mathcal{C}_1-KK^*$ with $K \coloneqq \mathcal{C}_1^{1/2}\sum_{i=1}^{\rank{K}}\frac{b_i}{1+b_i^2}g_i \otimes \varphi_i$ for any choice of ONB $(\varphi_i)_i$ of $\R^r$. Hence, $\ran{K} \subset\ran{\mathcal{C}_1^{1/2}}$ and $\rank{K}=\rank{Z}$. It remains to show that $\mathcal{C}\in\mathcal{E}(m_1,\mathcal{C}_1)$. By the implication \ref{item:properties_equivalence_3_5}$\Rightarrow$\ref{item:properties_equivalence_3_1} of \Cref{lemma:properties_equivalence}\ref{item:properties_equivalence_3}, it suffices to show that $\X\coloneqq I-(\mathcal{C}_1^{-1/2}K)(\mathcal{C}_1^{-1/2}K)^*$ is nonnegative and invertible. We have $\mathcal{C}_1^{-1/2}KK^*\mathcal{C}_1^{-1/2} = \sum_{i=1}^{\rank{K}}\frac{b_i^2}{1+b_i^2}g_i\otimes g_i$ on $\ran{\mathcal{C}_1^{1/2}}$. Now, $\ran{\mathcal{C}_1^{1/2}}\subset\mathcal{H}$ densely and $\sum_{i=1}^{\rank{K}}\frac{b_i^2}{1+b_i^2}g_i\otimes g_i\in\B(\H)$. Thus, \Cref{lemma:adjoint_of_densely_defined_operators}\ref{item:adjoint_of_densely_defined_operators_1} implies $(\mathcal{C}_1^{-1/2}K)^*\supset K^*\mathcal{C}_1^{-1/2}$ and hence $\mathcal{C}_1^{-1/2}K(\mathcal{C}_1^{-1/2}K)^* = \sum_{i=1}^{\rank{K}}\frac{b_i^2}{1+b_i^2}g_i\otimes g_i$. It follows that
		$\X = I - \sum_{i=1}^{\rank{K}}\frac{b_i^2}{1+b_i^2}g_i\otimes g_i.$
		\Cref{lemma:inverse_of_self_adj_hilbert_schmidt_perturbation}\ref{item:inverse_of_self_adj_hilbert_schmidt_perturbation_1}-\ref{item:inverse_of_self_adj_hilbert_schmidt_perturbation_2}, applied with $\delta_i\leftarrow \frac{-b_i^2}{1+b_i^2}$ for $i\leq\rank{K}$ and $\delta_i\leftarrow 0$ otherwise, implies that $\X$ is nonnegative and invertible, since $\frac{-b_i^2}{1+b_i^2}>-1$ for all $i$. 

	\ref{item:range_of_K_2}\quad Suppose that $\mathcal{C}=\mathcal{C}_1-KK^*$ for some $K\in\B(\R^r,\H)$ with $\ran{K}\subset\ran{\mathcal{C}_1}$ and $\mathcal{C}\in\mathcal{E}(m_1,\mathcal{C}_1)$. 
	We first show that $\mathcal{C}$ is injective and $\ran{\mathcal{C}}=\ran{\mathcal{C}_1}$. 
	By \cref{item:range_of_K_1}, $\mathcal{C}$ is injective and $\mathcal{C}^{-1}=\mathcal{C}_1^{-1/2}(I+ZZ^*)\mathcal{C}_1^{-1/2}$ on $\ran{\mathcal{C}}$ for some $Z\in\B(\R^r,\H)$ with $\rank{Z}=\rank{K}$. By the implication \ref{item:properties_equivalence_3_1}$\Rightarrow$\ref{item:properties_equivalence_3_5} in \Cref{lemma:properties_equivalence}\ref{item:properties_equivalence_3}, it follows that $\X\coloneqq I-(\mathcal{C}_1^{-1/2}K)(\mathcal{C}_1^{-1/2}K)^*$ is a well-defined, nonnegative and invertible operator, and by the implication \ref{item:properties_equivalence_3_1}$\Rightarrow$\ref{item:properties_equivalence_3_2} that $\mathcal{C}>0$. Using \Cref{lemma:properties_equivalence}\ref{item:properties_equivalence_4}, we diagonalise $\X=I - \sum_{i=1}^{\rank{K}}d_i^2 e_i\otimes e_i$ where $(d_{i}^2)_{i=1}^{\rank{K}}\subset(0,1)$ is nonincreasing and $(e_i)_{i=1}^{\rank{K}}\subset\ran{\mathcal{C}_1^{1/2}}$ is an orthonormal sequence in $\H$. It follows that $\X$ maps $\ran{\mathcal{C}_1^{1/2}}$ onto itself. Hence $\ran{\mathcal{C}}=\ran{\mathcal{C}_1^{1/2}\X\mathcal{C}_1^{1/2}}=\ran{\mathcal{C}_1}$, where we use $\mathcal{C}=\mathcal{C}_1^{1/2}\X\mathcal{C}_1^{1/2}$ from \Cref{lemma:properties_equivalence}\ref{item:properties_equivalence_1}.

	Next, we show that we may write $\mathcal{C}^{-1}=\mathcal{C}_1^{-1}+UU^*$ on $\ran{\mathcal{C}}=\ran{\mathcal{C}_1}$ for some $U\in\B(\R^r,\H)$ satisfying $\rank{U}=\rank{K}$.
	Let $h\in\ran{\mathcal{C}_1}=\ran{\mathcal{C}}$. Since $h\in\ran{\mathcal{C}_1}\subset\ran{\mathcal{C}_1^{1/2}}$, $\mathcal{C}_1^{-1/2}h\in\ran{\mathcal{C}_1^{1/2}}$. Since $h\in\ran{\mathcal{C}}$, $(I+ZZ^*)\mathcal{C}_1^{-1/2}h\in\ran{\mathcal{C}_1^{1/2}}$. Thus, $ZZ^*\mathcal{C}_1^{-1/2}h=(I+ZZ)^*\mathcal{C}_1^{-1/2}h-\mathcal{C}_1^{-1/2}h\in\ran{\mathcal{C}_1^{1/2}}$. This shows we may write $\mathcal{C}^{-1}=\mathcal{C}_1^{-1/2}(I+ZZ^*)\mathcal{C}_1^{-1/2}=\mathcal{C}_1^{-1}+\mathcal{C}_1^{-1/2}ZZ^*\mathcal{C}_1^{-1/2}$ on $\ran{\mathcal{C}}=\ran{\mathcal{C}_1}$. With $U\coloneqq \mathcal{C}_1^{-1/2}Z$ it then holds that $U\in\B(\R^r,\H)$, and $\rank{U}=\rank{Z}=\rank{K}$ by injectivity of $\mathcal{C}_1^{-1/2}$. By \Cref{lemma:adjoint_of_densely_defined_operators}\ref{item:adjoint_of_densely_defined_operators_1}, we have $(\mathcal{C}_1^{-1/2}Z)^*=Z^*\mathcal{C}_1^{-1/2}$ on $\ran{\mathcal{C}_1^{1/2}}\supset\ran{\mathcal{C}_1}$. Consequently, $\mathcal{C}^{-1} = \mathcal{C}_1^{-1}+UU^*$ on $\ran{\mathcal{C}_1}$. This proves the `only if' direction of the statement in \cref{item:range_of_K_2}.

	For the converse implication, assume that $\mathcal{C}$ is injective, $\ran{\mathcal{C}}=\ran{\mathcal{C}}_1$ and $\mathcal{C}^{-1}=\mathcal{C}_1^{-1}+UU^*$ on $\ran{\mathcal{C}_1}$ for some $U\in\B(\R^r,\H)$.
	With $Z\coloneqq \mathcal{C}_1^{1/2}U$ it holds that $\rank{U}=\rank{Z}$ by injectivity of $\mathcal{C}_1^{1/2}$, and $\mathcal{C}^{-1}=\mathcal{C}_1^{-1/2}(I+\mathcal{C}_1^{1/2}UU^*\mathcal{C}_1^{1/2})\mathcal{C}_1^{-1/2}=\mathcal{C}_1^{-1/2}(I+ZZ^*)\mathcal{C}_1^{-1/2}$ on $\ran{\mathcal{C}_1}=\ran{\mathcal{C}}$. By \cref{item:range_of_K_1}, $\mathcal{C}\in\mathcal{E}(m_1,\mathcal{C}_1)$ and $\mathcal{C}=\mathcal{C}_1-KK^*$ for some $K\in\B(\R^r,\H)$ with $\rank{K}=\rank{Z}=\rank{U}$. It remains to show that $\ran{K}\subset\ran{\mathcal{C}_1}$. 
	As in the proof of the `only if' statement in \cref{item:range_of_K_1}, we can diagonalise $ZZ^*=\sum_{i=1}^{\rank{Z}}b_i^2g_i\otimes g_i$ and write $\mathcal{C}=\mathcal{C}_1-KK^*$ with $K\coloneqq \mathcal{C}_1^{1/2}\sum_{i=1}^{\rank{Z}}\frac{b_i}{1+b_i^2}g_i\otimes \varphi_i$ for any choice of ONB $(\varphi_i)_i$ of $\R^r$. Since $\ran{Z}=\ran{\mathcal{C}_1^{1/2}U}\subset\ran{\mathcal{C}_1^{1/2}}$, we have $\ran{ZZ^*}\subset\ran{\mathcal{C}_1^{1/2}}$, and hence $g_i\subset\ran{\mathcal{C}_1^{1/2}}$ for each $i\leq \rank{Z}$. Thus, $\ran{K}\subset\Span{\mathcal{C}_1^{1/2}g_i,\ i\leq\rank{Z}}\subset\ran{\mathcal{C}_1}$.

	\ref{item:range_of_K_3}\quad The `if' direction holds by \cref{item:range_of_K_2} and the implication \ref{item:properties_equivalence_3_1}$\Rightarrow$\ref{item:properties_equivalence_3_2} of \Cref{lemma:properties_equivalence}\ref{item:properties_equivalence_3}.
	The `only if' direction follows from \cref{item:range_of_K_2} and \Cref{lemma:properties_equivalence}\ref{item:properties_equivalence_4}. To explain the details in the `only if' direction, we assume that $K\in\B(\R^r,\H)$ is such that $\ran{K}\subset\ran{\mathcal{C}_1}$ and $\mathcal{C}=\mathcal{C}_1-KK^*>0$. Then by \Cref{lemma:properties_equivalence}\ref{item:properties_equivalence_4} it holds for $\X\coloneqq I-(\mathcal{C}_1^{-1/2}K)(\mathcal{C}_1^{-1/2}K)^*$ that $\X=I-\sum_{i=1}^{\rank{K}}d_i^2 e_i\otimes e_i$ with $(d_i^2)_{i=1}^{\rank{K}}\subset(0,1)$ and that the equivalent properties of \Cref{lemma:properties_equivalence}\ref{item:properties_equivalence_3} hold. That is, $\mathcal{C}\in\mathcal{E}(m_1,\mathcal{C}_1)$. It now follows from \cref{item:range_of_K_2} that $\mathcal{C}$ is injective, $\ran{\mathcal{C}}=\ran{\mathcal{C}_1}$ and $\mathcal{C}^{-1}=\mathcal{C}_1^{-1}+UU^*$ on $\ran{\mathcal{C}_1}$ for some $U\in\B(\R^r,\H)$ with $\rank{U}=\rank{K}$.
\end{proof}

\correspondence*
\begin{proof}[Proof of \Cref{cor:correspondence}]
	Let $r\in\N$. \Cref{item:problem_equivalence} follows from \cref{item:equivalent_approximation_classes}: since $\mathscr{C}_r^{-1} \coloneqq\{\mathcal{C}^{-1}:\ \mathcal{C}\in\mathscr{C}_r\}= \mathscr{P}_r$, we have
	\begin{align*}
		\min\{\mathcal{L}(\mathcal{C}_\pos\Vert\mathcal{P}^{-1}):\ \mathcal{P}\in\mathscr{P}_r\}&=  \min\{\mathcal{L}\bigr(\mathcal{C}_\pos\Vert(\mathcal{C}^{-1})^{-1}\bigr):\ \mathcal{C}^{-1}\in\mathscr{P}_r\}\\
		&= \min\{\mathcal{L}(\mathcal{C}_\pos\Vert\mathcal{C}):\ \mathcal{C}\in\mathscr{C}_r\}.
	\end{align*}
	\Cref{item:equivalent_approximation_classes} follows directly from \Cref{prop:range_of_K}\ref{item:range_of_K_3} applied with $(m_1,\mathcal{C}_1)\leftarrow(0,\mathcal{C}_\pr)$ and the definitions \eqref{eqn:class_approx_covariance} and \eqref{eqn:class_approx_precision}.
	\end{proof}

\problemsWellDefined*

\begin{proof}[Proof of \Cref{cor:problems_well_defined}]
	The second statement follows from the first statement, since $\mathscr{P}_r^{-1}=\mathscr{C}_r$ by \Cref{cor:correspondence}\ref{item:equivalent_approximation_classes}, and since $\mathcal{L}\in\mathscr{L}$ is finite on $\mathcal{E}\times \mathcal{E}$ by definition \eqref{eqn:definition_loss_class} and by \Cref{lemma:finite_loss}\ref{item:properties_spectral_f}.
	The first statement follows from \Cref{prop:range_of_K}\ref{item:range_of_K_2}-\ref{item:range_of_K_3} applied with $(m_1,\mathcal{C}_1)\leftarrow(0,\mathcal{C}_\pr)$ and the definition \eqref{eqn:class_approx_covariance} of $\mathscr{C}_r$.

\end{proof}

\propertiesG*
\begin{proof}[Proof of \Cref{lemma:properties_g}]
	By \Cref{cor:correspondence}\ref{item:equivalent_approximation_classes} and \Cref{cor:problems_well_defined}, $(\mathcal{C}_{\pr}^{-1}+UU^*)^{-1}\in\mathcal{E}$ for $U\in\B(\R^r,\H)$. 
	By \Cref{lemma:expansions_feldman_hajek}\ref{item:covariance_mix_2} applied with $\mathcal{C}_1\leftarrow (\mathcal{C}_\pr^{-1}+UU^*)^{-1}$ and $\mathcal{C}_2\leftarrow \mathcal{C}_\pos$, $\mathcal{C}_\pos^{1/2}(\mathcal{C}_\pr^{-1}+UU^*)\mathcal{C}_\pos^{1/2} - I$ is densely defined and there exists an ONB $(e_i)_i$ of $\H$ and eigenvalue sequence $(\gamma_i)_i\in\ell^2( (-1,\infty))$ such that 
	\begin{align*}
		\mathcal{C}_\pos^{1/2}(\mathcal{C}_\pr^{-1}+UU^*)\mathcal{C}_\pos^{1/2} - I \subset ( (\mathcal{C}_\pr^{-1}+UU^*)^{1/2}\mathcal{C}_\pos^{1/2})^*( (\mathcal{C}_\pr^{-1}+UU^*)^{1/2}\mathcal{C}_\pos^{1/2})-I = \sum_{i}^{}\gamma_i e_i\otimes e_i,
	\end{align*}
	and by comparing with the expansion in \Cref{lemma:expansions_feldman_hajek}\ref{item:covariance_mix_1}, $\sum_{i}^{}\gamma_i e_i\otimes e_i$ has the same eigenvalues as $R(\mathcal{C}_\pos\Vert (\mathcal{C}_\pr^{-1}+UU^*)^{-1})$, counting multiplicities.
	Using \eqref{eqn:posterior_preconditioned_Hessian}, the leftmost operator can be written as
	\begin{align*}
		\mathcal{C}_\pos^{1/2}(\mathcal{C}_\pr^{-1}+UU^*)\mathcal{C}_\pos^{1/2} - I &= \mathcal{C}_\pos^{1/2}\mathcal{C}_\pr^{-1}\mathcal{C}_\pos^{1/2}-I+\mathcal{C}_\pos^{1/2}UU^*\mathcal{C}_\pos^{1/2} \\
		&\subset (\mathcal{C}_\pr^{-1/2}\mathcal{C}_\pos^{1/2})^*\mathcal{C}_\pr^{-1/2}\mathcal{C}_\pos^{1/2}-I + \mathcal{C}_\pos^{1/2}UU^*\mathcal{C}_\pos^{1/2}\\
		&= \mathcal{C}_\pos^{1/2}UU^*\mathcal{C}_\pos^{1/2} - \mathcal{C}_\pos^{1/2}H\mathcal{C}_\pos^{1/2} = g(U).
	\end{align*}
	Since $\mathcal{C}_\pos^{1/2}(\mathcal{C}_\pr^{-1}+UU^*)\mathcal{C}_\pos^{1/2} - I$ is densely defined, the above continuous extension is unique, which shows that $g(U)=\sum_{i}^{}\gamma_ie_i\otimes e_i$. 
	By \Cref{prop:bayesian_feldman_hajek}, $\rank{\mathcal{C}_\pos^{1/2}H\mathcal{C}_\pos^{1/2}}=\rank{H}$.
	Thus, 	
	\begin{align*}
		\rank{g(U)} \leq \rank{\mathcal{C}_\pos^{1/2}H\mathcal{C}_\pos^{1/2}}+\rank{\mathcal{C}_\pos^{1/2} UU^*\mathcal{C}_\pos^{1/2}}\leq \rank{H}+r.
	\end{align*}
	Furthermore, $\ran{g(U)}\subset\ran{\mathcal{C}_\pos^{1/2}}=\ran{\mathcal{C}_\pr^{1/2}}$. 
	For $i\in\N$ such that $\gamma_i\not=0$, this implies $e_i\in\ran{\mathcal{C}_\pr^{1/2}}$. 
	By \Cref{lemma:extension_of_finite_basis_in_dense_subspace}, we can extend $(e_i)_{i:\gamma_i\not=0}$ to a sequence in $\ran{\mathcal{C}_\pr^{1/2}}$ which is an ONB of $\mathcal{H}$. Replacing $(e_i)_i$ with this sequence, we still have $g(U)=\sum_{i}^{}\gamma_i e_i\otimes e_i$ and now $e_i\subset \ran{\mathcal{C}_\pr^{1/2}}$ for all $i$.
\end{proof}

\differentiabilityOfDecomposition*
\begin{proof}[Proof of \Cref{lemma:differentiability_of_decomposition}]
	Let $U\in\B(\R^r,\H)$. 
	We first show that the linear map $\B(\R^r,\H)\rightarrow L_2(\H)_\R$ given by $V\mapsto \mathcal{C}_\pos^{1/2} (UV^*+VU^*)\mathcal{C}_\pos^{1/2}$ is bounded, and then identify this map as the Fr\'echet derivative of $g$ at $U$.
	Let $V\in\B(\R^r,\H)$. By \Cref{lemma:kernel_range}, $\dim{\ker{V^*}^\perp}=\rank{V}\leq r$. Then there exists an ONB $(e_i)_i$ of $\H$ for which $\Span{e_i,\ i\leq r}$ contains $\ker{V^*}^\perp$. 
	We have,
	\begin{align*}
		\norm{UV^*}_{L_2(\H)}^2 = \sum_{i}^{}\norm{UV^*e_i}^2 = \sum_{i=1}^{r}\norm{UV^*e_i}^2 \leq \sum_{i=1}^{r}\norm{U}^2\norm{V^*}^2 = r\norm{U}^2\norm{V}^2,
	\end{align*}
	where we use the definition of the $L_2(\H)$-norm, the inclusion $\ker{V^*}^\perp\subset\Span{e_1,\ldots,e_r}$, the definition of the operator norm, and $\norm{V^*}=\norm{V}$ by \cite[Proposition VI.1.4(b)]{conway_course_2007}. 
	This also shows $\norm{VU^*}_{L_2(\H)}\leq \sqrt{r}\norm{V}\norm{U}$ and $\norm{VV^*}_{L_2(\H)}\leq \sqrt{r}\norm{V}^2$.
	Thus, using the triangle inequality and the fact $\norm{TS}_{L_2(\H)}=\norm{ST}_{L_2(\H)}\leq \norm{T}\norm{S}_{L_2(\H)}$ for any $T,S\in L_2(\H)$,
	\begin{align*}
		\norm{\mathcal{C}_\pos^{1/2}(UV^*+VU^*)\mathcal{C}_\pos^{1/2}}_{L_2(\H)} \leq \norm{\mathcal{C}_\pos^{1/2}}^2\norm{UV^*+VU^*}_{L_2(\H)} \leq 2\sqrt{r}\norm{\mathcal{C}_\pos^{1/2}}^2  \norm{U}\norm{V}.
	\end{align*}
	It follows that $V\mapsto \mathcal{C}_\pos^{1/2} (UV^*+VU^*)\mathcal{C}_\pos^{1/2}$ is bounded.
	We have by \eqref{eqn:definition_g},
	\begin{align*}
		g(U+V)-g(U)
		= \mathcal{C}_\pos^{1/2}( (U+V)(U+V)^*-UU^*)\mathcal{C}_\pos^{1/2}
		=\mathcal{C}_\pos^{1/2}(VV^*+UV^*+VU^*)\mathcal{C}_\pos^{1/2}.
	\end{align*}
	Using once more the fact $\norm{TS}_{L_2(\H)}=\norm{ST}_{L_2(\H)}\leq \norm{T}\norm{S}_{L_2(\H)}$ for any $T,S\in L_2(\H)$, and the bound $\norm{VV^*}_{L_2(\mathcal{H})}\leq \sqrt{r}\norm{V}^2$ proven above, it follows that
	\begin{align*}
		\norm{g(U+V)-g(U) - \mathcal{C}_\pos^{1/2}(UV^*+VU^*)\mathcal{C}_\pos^{1/2}}_{L_2(\H)}
		&=\norm{\mathcal{C}_\pos^{1/2}(VV^*)\mathcal{C}_\pos^{1/2}}_{L_2(\H)}\\
		&\leq \norm{\mathcal{C}_\pos^{1/2}}^2\norm{VV^*}_{L_2(\H)}
		\leq  \sqrt{r}\norm{\mathcal{C}_\pos^{1/2}}^2\norm{V}^2.
	\end{align*}
	Dividing by $\norm{V}$ and letting $\norm{V}\rightarrow 0$, this shows that $g$ is differentiable and has the stated derivative.

	To show differentiability of $F_f$, let $x=(x_i)_i\in\ell^2( (-1,\infty))$, $y=(y_i)_i\in\ell^2(\R)$ and define $c_x\in \ell^2(\R)$ by $(c_x)_i= f'(x_i)$. By the assumption $f\in\mathscr{F}$, $f'$ is Lipschitz continuous in some neighbourhood $(-a,a)$ of $0$, with $a>0$ and with Lipschitz constant $M_0$. Let us take $N_x$ so large that $\abs{x_i}<a/2$ for $i>N_x$. Let $\varepsilon>0$ be arbitrary. By differentiability of $f$, we can choose $\delta_{x,\varepsilon}>0$ such that $x_i+z\in(-1,\infty)$ and $\abs{(f(x_i+z)-f(x_i)-f'(x_i)z)}<\varepsilon\abs{z}$ for $\abs{z}<\delta_{x,\varepsilon}$ and $i=1,\ldots, N_x$. We then have for $\norm{y}<\min(\delta_{x,\varepsilon},a/2,\varepsilon)$,
	\begin{align*}
		\frac{1}{\norm{y}}\abs{F_f(x+y)-F_f(x)-\langle c_x,y\rangle} &= \frac{1}{\norm{y}}\abs{\sum_{i}^{}f(x_i+y_i)-f(x_i)-f'(x_i)y_i}\\
		&\leq \sum_{i=1}^{N_x}\frac{1}{\abs{y_i}}\abs{f(x_i+y_i)-f(x_i)-f'(x_i)y_i}\\
		&+\frac{1}{\norm{y}}\sum_{i>N_x}^{}\abs{f(x_i+y_i)-f(x_i)-f'(x_i)y_i}.
	\end{align*}
	As $\abs{y_i}<\delta_{x,\varepsilon}$ for $i=1,\ldots,N_x$, the first term is bounded from above by $N_x\varepsilon$. For the second term, by the mean value theorem, for each $i>N_x$ we can find $c_i\in [x_i-\abs{y_i},x_i+\abs{y_i}]\subset(-a,a)$ such that,
	\begin{align*}
		\frac{1}{\norm{y}}\sum_{i>N_x}^{}\abs{f(x_i+y_i)-f(x_i)-f'(x_i)y_i} &=
		\frac{1}{\norm{y}}\sum_{i>N_x}^{}\abs{f'(c_i)y_i-f'(x_i)y_i} \\
		&\leq \frac{1}{\norm{y}}\sum_{i>N_x}^{}M_0\abs{c_i-x_i}\abs{y_i} \\
		&\leq \frac{1}{\norm{y}}\sum_{i>N_x}^{}M_0\abs{y_i}^2 \leq M_0\norm{y}\leq M_0\varepsilon,
	\end{align*}
	where we used the Lipschitz continuity of $f'$ in $(-a,a)$ in the first inequality, and the fact that $c_i\in [x_i-\abs{y_i},x_i+\abs{y_i}]$ in the second inequality.
	Therefore, 
	\begin{align*}
		\frac{1}{\norm{y}}\abs{F_f(x+y)-F_f(x)-\langle c_x, y\rangle} \leq (N_x+M_0)\varepsilon,
	\end{align*}
	from which we conclude that $F_f'(x)$ exists and $F_f'(x) = c_x$.
\end{proof}

\spectralDerivative*
\begin{proof}[Proof of \Cref{prop:spectral_derivative}]
	We need to relate the Fr\'echet differentiability of the composition $\G\circ \Lambda^m$, being the composition of the eigenvalue map $\Lambda^m$ on $L_2(\mathcal{Z})_\R$ as defined in \Cref{sec:notation} and a symmetric function $\G$, to Fr\'echet differentiability of $\G$ itself. To do so, we use \cite[Theorem 1.1]{Lewis1996}, which states this for the case that $\Lambda^m$ is defined on the space of symmetric matrices instead of $L_2(\mathcal{Z})$. Therefore, we identify this space of symmetric matrices with $L_2(\mathcal{Z})_\R$. The details of this identification are described in the following.

	Any statement regarding differentiability in this proof should be understood as Fr\'echet differentiability.
	Let us write Sym($m$) for the symmetric matrices on $\R^m$ endowed with the trace-inner product: $\langle A,B\rangle_{\text{Sym(}m\text{)}} = \tr{BA}$ for symmetric matrices $A,B\in\R^{m\times m}$.
	Let $(\varphi_i)_i$ be the standard basis of $\R^m$. Define $\Phi:L_2(\mathcal{Z})_\R\rightarrow \text{Sym(}m\text{)}$ by $\Phi(X)=\sum_{i,j=1}^{m}\langle Xe_j,e_i\rangle \varphi_i\otimes \varphi_j$ and let $X\in L_2(\mathcal{Z})_\R$. 
	Then $\Phi$ is an isomorphism of Hilbert spaces and by linearity of $\Phi$ we have $\Phi'(X)(X_2)=\sum_{i,j=1}^{m}\langle X_2e_j,e_i\rangle \varphi_i\otimes \varphi_j$ for all $X_2\in L_2(\mathcal{Z})_\R$.
	Furthermore, $\Lambda^m\circ\Phi^{-1}$ is the eigenvalue map on Sym($m$), where the eigenvalues are ordered in a nonincreasing way. We note that $\Phi(X)=\sum_{i=1}^{m}(\Lambda^m\circ\Phi^{-1})_i(\Phi(X))\varphi_i\otimes \varphi_i$. 
	Because $\Omega$ and $\G$ are symmetric by hypothesis, we may apply \cite[Theorem 1.1]{Lewis1996}, which states that $\G\circ(\Lambda^m\circ\Phi^{-1})$ is differentiable in $\Phi(X)$ if and only if $\G$ is differentiable in $\Lambda^m\circ\Phi^{-1}(\Phi(X))=\Lambda^m(X)$, in which case the derivative is given by 
	\begin{align*}
		(\G\circ\Lambda^m\circ\Phi^{-1})'(\Phi(X)) = \sum_{i=1}^{m}\G'\left(\Lambda^m\circ\Phi^{-1}(\Phi(X))\right)_i \varphi_i\otimes \varphi_i = \sum_{i=1}^{m}\G'(\Lambda^m(X))_i\varphi_i\otimes \varphi_i.
	\end{align*}
	By the chain rule, $\G\circ \Lambda^m$ is differentiable in $X$ if and only if $\G\circ\Lambda^m\circ\Phi^{-1}$ is differentiable in $\Phi(X)$. 
	Thus, by the above display, $\G\circ\Lambda^m$ is differentiable in $X$ if and only if $\G$ is differentiable in $\Lambda^m(X)$. Another application of the chain rule, the expression for $\Phi'$ and the previous equation then finish the proof by showing that
	\begin{align*}
		\langle (\G\circ\Lambda^m)'(X),X_2\rangle_{L_2(\mathcal{Z})} &= \langle (\G\circ\Lambda^m\circ\Phi^{-1}\circ\Phi)'(X),X_2\rangle_{L_2(\mathcal{Z})} \\
		&= \Big\langle (\G\circ\Lambda^m\circ\Phi^{-1})'(\Phi(X)),\Phi'(X)(X_2)\Big\rangle_{\text{Sym(}m\text{)}}\\
		&= \Big\langle\sum_{i=1}^{m}\G'(\Lambda^m(X))_i \varphi_i\otimes \varphi_i,\sum_{k,j=1}^{m} \langle X_2e_j,e_k\rangle \varphi_k\otimes \varphi_j\Big\rangle_{\text{Sym(}m\text{)}}\\
		&= \sum_{i=1}^{m}\G'(\Lambda^m(X))_i \langle X_2e_i,e_i\rangle\\
		&= \Big\langle\sum_{i=1}^{m}\G'(\Lambda^m(X)) e_i\otimes e_i,X_2\Big\rangle_{L_2(\mathcal{Z})},
	\end{align*}
	for any $X_2\in L_2(\mathcal{Z})_\R$.
\end{proof}

\differentiabilityFAfterLambda*
\begin{proof}[Proof of \Cref{prop:differentiability_F_after_Lambda}]
	Any statement regarding differentiability in this proof should be understood as Fr\'echet differentiability.
	For $Y\in\mathcal{W}$, we have $\ker{Y}^\perp=\ker{Y^*}^\perp=\overline{\ran{Y}}=\ran{Y}$, by \Cref{lemma:kernel_range}, $Y=Y^*$ and the finite dimensionality of $\mathcal{Z}$.
	Let $m\coloneqq \dim{\mathcal{Z}}$ and extend $(e_i)_i$ to an ONB of $\mathcal{H}$. Note that the $m^2$ operators $e_i\otimes e_j$, $i,j\leq m$, span the space $\mathcal{W}$. Therefore, $\dim{\mathcal{W}}=m^2<\infty$.
	Because finite-dimensional spaces are closed, we can define $P_\mathcal{Z}:\mathcal{H}\rightarrow\mathcal{Z}$, the orthogonal projector onto $\mathcal{Z}$. Furthermore, we let $P_m:\ell^2(\R)\rightarrow\R^m$ be the orthogonal projector onto the first $m$ coordinates of an $\ell^2$ sequence. Thus, $P_m^*$ is the natural embedding of $\R^m$ into $\ell^2(\R)$. Since \Cref{prop:spectral_derivative} is a statement on $L_2$ spaces of finite dimension, we identify $\mathcal{W}$ with $L_2(\mathcal{Z})_\R$ via $\Psi:L_2(\mathcal{H})_\R\rightarrow L_2(\mathcal{Z})_\R$, where $\Psi(Y)\coloneqq P_\mathcal{Z}\restr{Y}{\mathcal{Z}}$. 

	We first prove the result for the case in which $\Lambda$ orders the eigenvalues in nonincreasing absolute value. With $\Lambda^m$ denoting the eigenvalue map on $L_2(\mathcal{Z})_\R$ as defined in \Cref{sec:notation}, we then have $\Lambda(Y)_i = \Lambda^m(\Psi(Y))_i$ for all $i\leq m$ and $Y\in \mathcal{W}$. This implies that 
	\begin{align}
		\nonumber
		P_m\Lambda(Y)&= \Lambda^m(\Psi(Y)),\quad Y\in \mathcal{W},\\
		\label{eqn:eigenvalue_map_relation_subspace}
		\Lambda(Y)&= P_m^*\Lambda^m(\Psi(Y)),\quad Y\in \mathcal{W},
	\end{align}
	since any $Y\in\mathcal{W}$ has at most $m$ nonzero eigenvalues.
	Let $\G\coloneqq FP_m^*$. If $\pi$ is a permutation on $\{1,\ldots,m\}$ and $x\in\R^m$, then $P_m^* x = (x_1,\ldots,x_m,0,\ldots)$ and $P_m^* (x_{\pi(i)})_i = (x_{\pi(1)},\ldots,x_{\pi(m)},0,\ldots)$, and by symmetry of $F$,
	\begin{align*}
		\G( (x_{\pi(i)})_{i=1}^m) = F( P_m^*(x_{\pi(i)})_i ) = F( P_m^*x) = \G(x)
	\end{align*}
	showing that $\G$ is a symmetric function on the symmetric domain $\R^r$.
	Furthermore, by definition of $\G$ and \eqref{eqn:eigenvalue_map_relation_subspace}
	\begin{align}
		\label{eqn:composition_map_relation}
		\restr{(F\circ\Lambda)}{\mathcal{W}}(Y) = (F\circ \Lambda)(Y) = (\G\circ \Lambda^m)(\Psi(Y)),\quad Y\in \mathcal{W}.
	\end{align}
	By hypothesis, $F$ is differentiable at $\Lambda(X)$, with $X\in\mathcal{W}$. The idea of the proof is to first use \Cref{prop:spectral_derivative} to conclude differentiability of $\G\circ \Lambda^m$ at $\Psi(X)$ and then to use \eqref{eqn:composition_map_relation} and the chain rule to obtain differentiability of $\restr{(F\circ \Lambda)}{\mathcal{W}}$ at $X$ in the $L_2(\mathcal{H})$ norm topology. 

	In order to apply \Cref{prop:spectral_derivative}, we need to show that $\G$ is differentiable in $\Lambda^m(\Psi(X))$.
	By the hypothesis on $F$, $F$ is differentiable at $\Lambda(X)$ for $X\in\mathcal{W}$. 
	Furthermore, $P_m^*$ is linear, hence differentiable, and $(P_m^*)'(x)(y)=P_m^*y$ for $x,y\in\R^m$. Then, by \eqref{eqn:eigenvalue_map_relation_subspace} and the chain rule, the composition $F\circ P_m^*$ is differentiable at $\Lambda^m(\Psi(X))$, and it holds for any $y\in\R^m$,
	\begin{align*}
		\langle (FP_m^*)'(\Lambda^m(\Psi(X))),y\rangle
		&=\langle F'(P_m^*\Lambda^m(\Psi(X))),(P_m^*)'(\Lambda^m(\Psi(X)))y\rangle\\
		&=\langle F'(P_m^*\Lambda^m(\Psi(X))),P_m^*y\rangle\\
		&=\langle P_mF'(P_m^*\Lambda^m(\Psi(X))),y\rangle\\
		&=\langle P_mF'(\Lambda(X)),y\rangle,
	\end{align*}
	where we use the chain rule in the first step, the expression for the derivative $(P_m^*)'$ in the second step, the definition of the adjoint in the third step, and \eqref{eqn:eigenvalue_map_relation_subspace} in the final step.
	That is, $\G$ is differentiable at $\Lambda^m(\Psi(X))$ and 
	\begin{align}
		\G'(\Lambda^m(\Psi(X))) = P_mF'(\Lambda(X))\in\R^m.
		\label{eqn:projected_F_derivative}
	\end{align}
	We may now apply \Cref{prop:spectral_derivative} to conclude that $\G\circ\Lambda^m$ is differentiable at $\Psi(X)$. To obtain an expression for the derivative, notice that by the fact that $e_i\in\ran{\mathcal{Z}}$ for $i\leq m$ and by the hypothesised diagonalisation of $X$, we have $\Psi(X) = \sum_{i=1}^{m}\Lambda^m(\Psi(X))_ie_i\otimes e_i$, where the rank-1 operators $e_i\otimes e_i$ are now understood to act only on $\mathcal{Z}$. With \Cref{prop:spectral_derivative} we thus also obtain the expression for the derivative
	\begin{align*}
		(\G\circ\Lambda^m)'(\Psi(X)) = \sum_{i=1}^{m} \G'(\Lambda^m(\Psi(X)))_i e_i\otimes e_i = \sum_{i=1}^{m}\left(P_mF'(\Lambda(X))\right)_ie_i\otimes e_i\in L_2(\mathcal{Z})_\R,
	\end{align*}
	where for the second equation we use \eqref{eqn:projected_F_derivative}. By definition of $P_m$, $\left(P_mF'(\Lambda(X))\right)_i=F'(\Lambda(X))_i$ for $i\leq m$. Hence,
	\begin{align}
		\label{eqn:projected_composition_derivative}
		(\G\circ\Lambda^m)'(\Psi(X)) = \sum_{i=1}^{m}F'(\Lambda(X))_ie_i\otimes e_i.
	\end{align}
	Because $\Psi$ is linear, hence differentiable, the chain rule and \eqref{eqn:composition_map_relation} show that $\restr{(F\circ\Lambda)}{\mathcal{W}}$ is differentiable at $X$. To obtain the expression of the derivative, we use \eqref{eqn:composition_map_relation}, the chain rule, the fact $\Psi'(X)(Y)=\Psi(Y)$ for $Y\in\mathcal{W}$ and \eqref{eqn:projected_composition_derivative} to find
	\begin{align*}
		\langle \restr{(F\circ\Lambda)}{\mathcal{W}}'(X),Y\rangle_{L_2(\mathcal{H})} 
		&= \langle (\G\circ\Lambda^m\circ\Psi)'(X),Y\rangle_{L_2(\mathcal{H})} \\
		&=  \langle (\G\circ\Lambda^m)'(\Psi(X)),\Psi'(X)(Y)\rangle_{L_2(\mathcal{Z})}\\
		&= \langle (\G\circ\Lambda^m)'(\Psi(X)),\Psi(Y)\rangle_{L_2(\mathcal{Z})} \\
		&=  \Big\langle\sum_{i=1}^{m}F'(\Lambda(X))_ie_i\otimes e_i, \Psi(Y) \Big\rangle_{L_2(\mathcal{Z})}.
	\end{align*}
	Since, for $Y\in\mathcal{W}$, it holds that $\ran{Y}\subset\mathcal{Z}$ and $\ker{Y}^\perp=\ran{Y}$, we have $\ker{Y}^\perp\subset\mathcal{Z}$. Thus, we have $\ran{\Psi(Y)}=Y(\mathcal{Z})=Y(\ker{Y}^\perp)=\ran{Y}\subset\mathcal{Z}$ as subspaces of $\mathcal{H}$. For $i\leq m$ it holds that $e_i\in\mathcal{Z}$, hence $\langle e_i\otimes e_i,\Psi(Y)\rangle_{L_2(\mathcal{Z})} = \langle e_i\otimes e_i,Y\rangle_{L_2(\mathcal{H})}$, where on the right hand side we interpret $e_i\otimes e_i$ as acting on all of $\mathcal{H}$. For $i>m$, we have $e_i\in\mathcal{Z}^\perp$, so that $\ran{Y}\subset\mathcal{Z}$ implies that $\langle e_i\otimes e_i,Y\rangle_{L_2(\mathcal{H})} = 0$. Thus,
	\begin{align*}
		\langle \restr{(F\circ\Lambda)}{\mathcal{W}}'(X),Y\rangle_{L_2(\mathcal{H})} 
		&= \Big\langle\sum_{i=1}^{m}F'(\Lambda(X))_ie_i\otimes e_i, Y \Big\rangle_{L_2(\mathcal{\mathcal{H}})}\\
		&= \Big\langle\sum_{i=1}^{\infty}F'(\Lambda(X))_ie_i\otimes e_i, Y \Big\rangle_{L_2(\mathcal{\mathcal{H}})}.
	\end{align*}
	This concludes the proof for the case that $\Lambda$ orders the eigenvalues in a nonincreasing way.
	
	Finally, let us denote by $\tilde{\Lambda}$ an eigenvalue map on $L_2(\mathcal{H})_\R$ which can assign any fixed but arbitrary ordering on the eigenvalues. Let $X\in\mathcal{W}$, $X=\sum_{i=1}^{}\tilde{\Lambda}(X)_i e_i\otimes e_i$ be given and assume that $F$ is differentiable at $\tilde{\Lambda}(X)$. 
	Given $X$, there exists a permutation $\pi:\N\rightarrow\N$ such that, for the eigenvalue map $\Lambda$ from the previous part of the proof, $\tilde{\Lambda}(X)_i = \Lambda(X)_{\pi(i)}$ for all $i\in\N$. 
	Let $P_{\pi}:\ell^2(\R)\rightarrow\ell^2(\R)$ denote the permutation operator $(P_{\pi} x)_i = x_{\pi(i)}$, $i\in\N$, so $P_{\pi}^*=P_{\pi}^{-1}$. Then $\tilde{\Lambda}(X) = (P_{\pi}\Lambda)(X)$ and $X=\sum_{i}^{}\Lambda_{\pi(i)}(X)e_i\otimes e_i = \sum_{i}^{}\Lambda(X)_i e_{\pi^{-1}(i)}\otimes e_{\pi^{-1}(i)}.$
	By the previous part of the proof, $\restr{(F\circ\Lambda)}{\mathcal{W}}$ is differentiable at $X$. Because $F\circ \tilde{\Lambda}=F\circ\Lambda$ by symmetry of $F$, differentiability of $\restr{(F\circ\tilde{\Lambda})}{\mathcal{W}}$ at $X$ follows, and
	\begin{align}
		\label{eqn:unordered_derivative}
		(F\circ\tilde{\Lambda})'(X) = (F\circ \Lambda)'(X) = \sum_{i}^{}F_i'(\Lambda(X))e_{\pi^{-1}(i)}\otimes e_{\pi^{-1}(i)}.
	\end{align}
	Since $F$ is symmetric, $F\circ P_{\pi}= F$. Hence, for $x\in\ell^2(\R)$,
	\begin{align*}
		F'(x) = (F\circ P_{\pi})'(x) = P_{\pi}^*F'(P_{\pi} x) = P_{\pi}^{-1}F'(P_{\pi} x).
	\end{align*}
	Thus, $F_i'(\Lambda(X)) = F_{\pi^{-1}(i)}'(P_{\pi}\Lambda(X)) = F_{\pi^{-1}(i)}'(\tilde{\Lambda}(X))$. From \eqref{eqn:unordered_derivative} it now follows that
	\begin{align*}
		(F\circ\tilde{\Lambda})'(X) = \sum_{i}^{}F_{\pi^{-1}(i)}'(\tilde{\Lambda}(X)) e_{\pi^{-1}(i)}\otimes e_{\pi^{-1}(i)} = \sum_{i}^{}F_i'(\tilde{\Lambda}(X))e_i\otimes e_i.
	\end{align*}
\end{proof}

\differentiabilityJ*
\begin{proof}[Proof of \Cref{prop:differentiability_J}]
	Let $U,V\in\B(\R^r,\mathcal{H})$. Define 
	\begin{align*}
		\mathcal{Z}\coloneqq \ran{\mathcal{C}_\pos^{1/2}UU^*\mathcal{C}_\pos^{1/2}}+\ran{\mathcal{C}_\pos^{1/2}(UV^*+VU^*)\mathcal{C}_\pos^{1/2}}+\ran{\mathcal{C}_\pos^{1/2}VV^*\mathcal{C}_\pos^{1/2}}+\ran{\mathcal{C}_\pos^{1/2}H\mathcal{C}_\pos^{1/2}}
	\end{align*}
	and $\mathcal{W}\coloneqq \{X\in L_2(\mathcal{H})_\R:\ \ran{X}\subset\mathcal{Z}\}\subset L_2(\mathcal{H})_\R$. Then $\dim{\mathcal{Z}}<\infty$ since $U$, $V$ and $H$ are finite-rank, and $\ran{g(U+tV)}\subset \mathcal{Z}$ for all $t\in\R$ by definition \eqref{eqn:definition_g} of $g$, hence $g(U+tV)\in\mathcal{W}$ for all $t\in\R$. Thus, $F_f\circ\Lambda\circ g(U+tV) = \restr{(F_f\circ\Lambda)}{\mathcal{W}}\circ g(U+tV)$ for all $t\in\R$. By \Cref{lemma:differentiability_of_decomposition} and \Cref{prop:differentiability_F_after_Lambda}, $\restr{(F_f\circ\Lambda)}{\mathcal{W}}$ is Fr\'echet differentiable. By \Cref{lemma:differentiability_of_decomposition}, $g$ is Fr\'echet differentiable. In particular, $g$ is Gateaux differentiable at $U$ in the direction $V$. Hence, by the chain rule, $J_f$ is Gateaux differentiable at $U$ in the direction $V$. To compute the derivative, we recall that $(e_j\otimes e_k)_{j,k}$ is an ONB of $L_2(\H)$. 
	The Gateaux derivative of $J_f$ at $U$ in the direction $V$ is
	\begin{align*}
		J_f'(U)(V) &= (\restr{(F_f\circ\Lambda)}{\mathcal{W}}\circ g)'(U)(V) \\
		&= \Big\langle \restr{(F_f\circ\Lambda)}{\mathcal{W}}'(g(U)),g'(U)(V)\Big\rangle_{L_2(\mathcal{H})}\\
		&= \Big\langle \sum_{i}^{}f'\left(\Lambda_i(g(U))\right)e_i\otimes e_i, g'(U)(V)\Big\rangle_{L_2(\mathcal{H})}\\
		&= \sum_{j,k}^{}\Big\langle \sum_{i}^{}f'\left(\Lambda_i(g(U))\right)e_i\otimes e_i, e_j\otimes e_k\Big\rangle_{L_2(\mathcal{H})}\langle g'(U)(V),e_j\otimes e_k\rangle_{L_2(\mathcal{H})}\\
		&= \sum_{i}^{}f'\left(\Lambda_i(g(U))\right)\langle g'(U)(V), e_i\otimes e_i\rangle_{L_2(\mathcal{H})}.
	\end{align*}
	The second equation follows by the chain rule. The third equation follows from the expression for the derivative in \Cref{prop:differentiability_F_after_Lambda} applied with $X\leftarrow g(U)$ and the expression for $F_f'$ in \Cref{lemma:differentiability_of_decomposition}. The fourth and fifth equations use the property that $(e_j\otimes e_k)_{j,k}$ is an ONB of $L_2(\H)$. 
	Using the formula for $g'(U)(V)$ from \Cref{lemma:differentiability_of_decomposition},
	\begin{align*}
		J_f'(U)(V) &= \sum_{i}^{}f'\left(\Lambda_i(g(U))\right)\langle \mathcal{C}_\pos^{1/2}(UV^*+VU^*)\mathcal{C}_\pos^{1/2}, e_i\otimes e_i\rangle_{L_2(\mathcal{H})}\\
		&= \sum_{i}^{}f'\left(\Lambda_i(g(U))\right)\langle \mathcal{C}_\pos^{1/2}(UV^*+VU^*)\mathcal{C}_\pos^{1/2}e_i, e_i\rangle\\
		&= 2\sum_{i}^{}f'\left(\Lambda_i(g(U))\right)\langle \mathcal{C}_\pos^{1/2}e_i,VU^* \mathcal{C}_\pos^{1/2}e_i\rangle.
	\end{align*}
\end{proof}

\coercivity*
\begin{proof}[Proof of \Cref{lemma:coercivity}]
	Let $f\in\mathscr{F}$ and let $(U_n)_n$ be a sequence in $\B(\R^r,\mathcal{V})$ such that $\norm{U_n}\rightarrow\infty$. Then, by \Cref{lemma:cstar_property}, $\norm{U_nU_n^*}=\norm{U_n}^2\rightarrow\infty$. Since $f(x)\rightarrow\infty$ for $\norm{x}\rightarrow\infty$ and $f$ is bounded from below, it is enough to show that there is an eigenvalue $\alpha_n$ of $g(U_n)$ with $\abs{\alpha_n}\rightarrow\infty$. Since $g(U_n)$ is self-adjoint and compact by definition \eqref{eqn:definition_g}, $\norm{g(U_n)} = \max_i\abs{\Lambda_i(g(U_n))}$ by \Cref{lemma:norm_of_compact_self_adjoint}, and we must therefore show that $\norm{g(U_n)}\rightarrow\infty$. For this, it is enough to find a bounded sequence $h_n\in\mathcal{H}$, such that $\norm{g(U_n)h_n}\rightarrow\infty$.

	For any $h\in\mathcal{H}$, we have by the triangle inequality, \eqref{eqn:definition_g} and \Cref{prop:bayesian_feldman_hajek}
	\begin{align}
		\norm{g(U)h} &= \norm{\mathcal{C}_\pos^{1/2} U_nU_n^*\mathcal{C}_\pos^{1/2}h - \mathcal{C}_\pos^{1/2}H\mathcal{C}_\pos^{1/2}h} \nonumber \\
		&\geq \norm{\mathcal{C}_\pos^{1/2} U_nU_n^*\mathcal{C}_\pos^{1/2}h} - \norm{\mathcal{C}_\pos^{1/2}H\mathcal{C}_\pos^{1/2}h} \nonumber \\
		&\geq \norm{\mathcal{C}_\pos^{1/2} U_nU_n^*\mathcal{C}_\pos^{1/2}h} - 1. \label{eqn:lower_bound_g_U}
	\end{align}
	Let us write $m\coloneqq \dim{\mathcal{V}}$. For each $n$, let us diagonalise $U_nU_n^* = \sum_{j=1}^{m}\beta_{n,j}\psi_{n,j}\otimes\psi_{n,j}$, where $(\psi_{n,j})_{j=1}^m$ forms an ONB of $\mathcal{V}$ and $\beta_{n,j}\geq 0$. Define $j_n\coloneqq \argmax_{j\leq m}\beta_{n,j}$, so that $\beta_{n,j_n}$ is the largest eigenvalue of $U_nU_n^*$. As $U_nU_n^*$ is self-adjoint, $\beta_{n,j_n}=\norm{U_nU_n^*}$ by \Cref{lemma:norm_of_compact_self_adjoint}, showing $\beta_{n,j_n}\rightarrow\infty$. 
	Let $\varepsilon>0$. By density of $\ran{\mathcal{C}^{1/2}_\pos}$, for each $j\leq m$ we can choose $k_{j}\in\H$ which satisfies $\norm{\psi_{1,j}-\mathcal{C}_\pos^{1/2}k_{j}}\leq \varepsilon$. 
	Let us decompose $\psi_{n,j_n}=\sum_{j=1}^{m}\langle \psi_{n,j_n},\psi_{1,j}\rangle \psi_{1,j}$ and define $h_n\coloneqq \sum_{j=1}^{m}\langle \psi_{n,j_n},\psi_{1,j}\rangle k_j$. Note that $\norm{h_n}\leq C$ with $C\coloneqq m\max_j\norm{k_j}$ by the Cauchy--Schwarz inequality. By further application of the Cauchy--Schwarz inequality,
	\begin{align*}
		\norm{\psi_{n,j_n}-\mathcal{C}_{\pos}^{1/2}h_n}^2 &= \sum_{j}^{}\langle \psi_{n,j_n},\psi_{1,j}\rangle^2\norm{\psi_{1,j}-\mathcal{C}_{\pos}^{1/2}k_j}^2 \\
		&+ 2\sum_{i\not=j}^{}\langle \psi_{n,j_n},\psi_{1,i}\rangle \langle \psi_{n,j_n},\psi_{1,j}\rangle\langle \psi_{1,i}-\mathcal{C}_{\pos}^{1/2}k_i,\psi_{1,j}-\mathcal{C}_{\pos}^{1/2}k_j\rangle\\
		&\leq m\varepsilon^2 + 2m(m-1)\varepsilon^2 = m(2m-1)\varepsilon^2.
	\end{align*}
	It follows that for $\varepsilon$ small enough, there exists $c>0$ such that
	\begin{align*}
		\langle \psi_{n,j_n},\mathcal{C}_\pos^{1/2}h_n\rangle &=  \frac{1}{2}\left( \norm{\psi_{n,j_n}}^2+\norm{\mathcal{C}_\pos^{1/2}h_n}^2-\norm{\psi_{n,j_n}-\mathcal{C}_\pos^{1/2}h_n}^2 \right)\\
		&\geq \frac{1}{2}(1+0-m(2m-1)\varepsilon^2)>c.
	\end{align*}
	By the Cauchy--Schwarz inequality, the bound $\norm{h_n}\leq C$, the fact that $\mathcal{C}_\pr^{1/2}$ is self-adjoint, the given diagonalisation of $U_nU_n^*$ and the previous lower bound,
	\begin{align*}
		\norm{\mathcal{C}_\pos^{1/2} U_n U_n^*\mathcal{C}_{\pos}^{1/2}h_n} &\geq C^{-1}{\langle \mathcal{C}_\pos^{1/2} U_n U_n^*\mathcal{C}_\pos^{1/2} h_n,h_n\rangle}=C^{-1}\langle U_nU_n^*\mathcal{C}_\pos^{1/2}h_{n},\mathcal{C}_\pos^{1/2}h_{n}\rangle\\
		&= C^{-1}\sum_{j}^{}\beta_{n,j} \abs{\langle \psi_{n,j},\mathcal{C}_{\pos}^{1/2}h_n\rangle}^2 \geq C^{-1}\beta_{n,j_n}\abs{\langle \psi_{n,j_n},\mathcal{C}_{\pos}^{1/2}h_n\rangle}^2 \\
		&\geq c^2C^{-1}\beta_{n,j_n}\rightarrow\infty.
	\end{align*}
	Combining this with \eqref{eqn:lower_bound_g_U}, we thus find a bounded sequence $(h_n)_n$ with $\norm{g(U_n)h_n}\rightarrow\infty$.
	Finally, by \Cref{prop:differentiability_J}, $J_f$ is differentiable. The conclusion follows because a coercive, differentiable function on a finite-dimensional space $\B(\R^r,\mathcal{V})$ has a global minimum, that is attained only among its stationary points. 
\end{proof}

\optCovarianceAndPrecision*
\begin{proof}[Proof of \Cref{thm:opt_covariance_and_precision}]
	Let $f\in\mathscr{F}$. By \Cref{prop:differentiability_J}, $J_f$ is Gateaux differentiable.
	It follows from \cite[Theorem 12.4.5 (i)]{bogachevRealFunctionalAnalysis2020} that local minimisers of $J_f$ are stationary points, i.e.\ have Gateaux derivative equal to 0.
	The idea of the proof is to find among all stationary points of $J_f$ the stationary points that minimise $J_f$, and use the coercivity of $J_f$ over finite dimensional subspaces of $\mathcal{B}(\R^r,\mathcal{H})$ to conclude that these stationary points are global minimisers. We then relate these minimisers to the solutions of \Cref{prob:optimal_precision,prob:optimal_covariance}.

	\textbf{Step 1}: characterisation of the stationary points of $J_f$.

	Let $U\in\B(\R^r,\H)$. Let $(\gamma_i)_i$ and $(e_i)_i$ be as in \Cref{lemma:properties_g}, so that $e_i\in\ran{\mathcal{C}_\pr^{1/2}}=\ran{\mathcal{C}_\pos^{1/2}}$ and $g(U) = \sum_{i}^{}\gamma_ie_i\otimes e_i$.
	By \Cref{prop:differentiability_J}, the Gateaux derivative $J_f'(g(U))\in L_2^*\simeq L_2$ at $U\in\B(\R^r,\mathcal{H})$ is given by
	\begin{align*}
		J_f'(U)(V)= 2\sum_{i}^{}f'(\gamma_i)\langle \mathcal{C}_\pos^{1/2}e_i,VU^* \mathcal{C}_\pos^{1/2}e_i\rangle,\quad V\in\B(\R^r,\mathcal{H}).
	\end{align*}
	Thus, $U$ is a stationary point of $J_f$ if and only if for all $V\in\B(\R^r,\H)$,
	\begin{align*}
		\sum_{i}^{}f'(\gamma_i)\langle \mathcal{C}_\pos^{1/2}e_i,VU^* \mathcal{C}_\pos^{1/2}e_i\rangle =0.
	\end{align*}
	Since $f\in\mathscr{F}$, it follows from \Cref{lemma:finite_loss}\ref{item:properties_spectral_f} that $f'(\gamma_i)=0$ if and only if $\gamma_i=0$. 
	For an arbitrary fixed $j$, if $\gamma_j\not=0$, this implies that $U^*\mathcal{C}_\pos^{1/2}e_j=0$ for a stationary point $U$, as otherwise there exists $\varphi\in \R^r$ such that $\langle \varphi,U^*\mathcal{C}_\pos^{1/2}e_j\rangle_{\R^r}\not=0$ and the choice $V=\mathcal{C}_\pos^{-1/2}e_j\otimes \varphi$ furnishes a contradiction with $U$ being stationary. Indeed, in this case,
	\begin{align*}
		\sum_{i}^{}f'(\gamma_i)\langle \mathcal{C}_\pos^{1/2}e_i, (\mathcal{C}_\pos^{-1/2}e_j\otimes\varphi)U^*\mathcal{C}_\pos^{1/2}e_i\rangle 
		&= \sum_{i}^{}f'(\gamma_i)\langle \mathcal{C}_\pos^{1/2}e_i,\mathcal{C}_\pos^{-1/2}e_j\rangle \langle U^*\mathcal{C}_\pos^{1/2}e_i,\varphi\rangle \\
		&= f'(\gamma_j)\langle \varphi,U^*\mathcal{C}_\pos^{1/2}e_j\rangle
		\not=0.
	\end{align*}
	Hence, if $U$ is a stationary point, then $\gamma_i U^*\mathcal{C}_\pos^{1/2}e_i=0$ for all $i$. Conversely, if $\gamma_iU^*\mathcal{C}_\pos^{1/2}e_i=0$ for all $i$, then for each $i$ it holds that $U^*\mathcal{C}_\pos^{1/2}e_i=0$ or $\gamma_i=0$, showing $f'(\gamma_i)\langle \mathcal{C}_\pos^{1/2}e_i,VU^*\mathcal{C}_\pos^{1/2}e_i\rangle=0$ for all $i$ and all $V$. Hence $U$ is a stationary point of $J_f$ if and only if $\gamma_i U^*\mathcal{C}_\pos^{1/2}e_i=0$ for all $i$.

	Precomposing $g(U)=\sum_{i}^{}\gamma_i e_i\otimes e_i$ with $\mathcal{C}_\pos^{1/2}UU^*\mathcal{C}_\pos^{1/2}$ and using \eqref{eqn:definition_g}, 
	\begin{align*}
		\mathcal{C}_\pos^{1/2}(UU^*-H)\mathcal{C}_\pos UU^*\mathcal{C}_\pos^{1/2} &= \left(\sum_{i}^{}\gamma_i e_i\otimes e_i\right) \mathcal{C}_\pos^{1/2}UU^*\mathcal{C}_\pos^{1/2}\\
		&=  \left(\sum_{i}^{} e_i\otimes(\gamma_iU^*\mathcal{C}_\pos^{1/2}e_i)\right) U^*\mathcal{C}_\pos^{1/2},
	\end{align*}
	where the second equation follows since $(u\otimes v)AA^*w=u\langle AA^*w,v\rangle=u\langle A^*w,A^*v\rangle=(u\otimes (A^*v))A^*w$ for suitable $u$, $v$, $w$, and $A$.
	Thus, if $U\in\B(\R^r,\mathcal{H})$ is a stationary point of $J_f$, then
	\begin{align}
		\label{eqn:stationary_point_condition}
		(\mathcal{C}_\pos^{1/2}H\mathcal{C}_\pos^{1/2})( \mathcal{C}_\pos^{1/2}UU^*\mathcal{C}_\pos^{1/2}) = (\mathcal{C}_\pos^{1/2}UU^*\mathcal{C}_\pos^{1/2})^2.
	\end{align}
	Conversely, if \eqref{eqn:stationary_point_condition} holds, then $\sum_{i}^{}\langle \gamma_iU^*\mathcal{C}_\pos^{1/2}e_i,U^*\mathcal{C}_\pos^{1/2}h\rangle e_i=0$ for all $h\in\H$,
	because \eqref{eqn:stationary_point_condition} is equivalent to $\mathcal{C}_\pos^{1/2}(UU^*-H)\mathcal{C}_\pos UU^*\mathcal{C}_\pos^{1/2}=0$, and 
	\begin{align*}
		\mathcal{C}_\pos^{1/2}(UU^*-H)\mathcal{C}_\pos UU^*\mathcal{C}_\pos^{1/2}h&= \left(\sum_{i}^{} e_i\otimes(\gamma_iU^*\mathcal{C}_\pos^{1/2}e_i)\right) U^*\mathcal{C}_\pos^{1/2}h\\
		&= \sum_{i} \langle \gamma_i U^* \mathcal{C}_\pos^{1/2}e_i, U^*\mathcal{C}_\pos^{1/2}h\rangle e_i.
	\end{align*}
	Since $(e_i)_i$ is an ONB, this implies $\gamma_i\langle \mathcal{C}_\pos^{1/2}UU^*\mathcal{C}_\pos^{1/2}e_i,h\rangle=0$ for all $h\in\mathcal{H}$ and for all $i$. 
	If $\gamma_i\not=0$ for some $i$, then taking $h=\mathcal{C}_\pos^{1/2}UU^*\mathcal{C}_\pos^{1/2}e_i$ we get $\norm{\mathcal{C}_\pos^{1/2}UU^*\mathcal{C}_\pos^{1/2}e_i}=0$. Therefore, $\mathcal{C}_\pos^{1/2}e_i\in\ker{\mathcal{C}_\pos^{1/2}UU^*}=\ker{UU^*}=\ker{U^*}$ by injectivity of $\mathcal{C}_\pos^{1/2}$ and \Cref{lemma:kernel_of_square}, showing $\gamma_i U^*\mathcal{C}_\pos^{1/2}e_i=0$ for all $i$. Thus, $U$ satisfies \eqref{eqn:stationary_point_condition} if and only if $U$ is a stationary point.

	By injectivity of $\mathcal{C}_\pos^{1/2}$ and \Cref{lemma:range_of_square_of_finite_rank}, we have 
	\begin{equation*}
	\rank{\mathcal{C}_\pos^{1/2}UU^*\mathcal{C}_\pos^{1/2}}=\rank{\mathcal{C}_\pr^{1/2}U(\mathcal{C}_\pr^{1/2}U)^*}=\rank{\mathcal{C}_\pos^{1/2}U}=\rank{U}.
	\end{equation*}
	Hence $\mathcal{C}_\pos^{1/2}UU^*\mathcal{C}_\pos^{1/2}$ is a non-negative and self-adjoint operator of rank at most $r$. 
	In particular, it has $k$ many nonzero eigenvalues for some $k\leq r$. 
	Suppose $U$ satisfies \eqref{eqn:stationary_point_condition}, so that the corresponding $k$ eigenpairs are also eigenpairs of $\mathcal{C}_\pos^{1/2}H\mathcal{C}_\pos^{1/2}$.
	By \Cref{prop:bayesian_feldman_hajek}, there exists a nondecreasing sequence $(\lambda_i)_i\in\ell^2( (-1,0])$ with exactly $\rank{H}$ nonzero entries and ONBs $(w_i)_i$ and $(v_i)_i$ of $\mathcal{H}$ with $w_i,v_i\in\ran{\mathcal{C}_\pr^{1/2}}$ and $v_i=\sqrt{1+\lambda_i}\mathcal{C}_\pos^{-1/2}\mathcal{C}_\pr^{1/2}w_i$ for all $i$, such that
		$\mathcal{C}_\pos^{1/2}H\mathcal{C}_\pos^{1/2} = \sum_{i}^{}(-\lambda_i)v_i\otimes v_i$.
	It follows that there exists a set of $k$ distinct indices $\{i_1,\ldots,i_k\}\subset\{1,\ldots,\rank{H}\}$ such that 
	\begin{align}
		\label{eqn:stationary_point_condition_2}
		\mathcal{C}_\pos^{1/2}UU^*\mathcal{C}_\pos^{1/2} &= \sum_{j=1}^{k}(-\lambda_{i_j})v_{i_j}\otimes v_{i_j}.
	\end{align}
	Conversely, if $U$ satisfies \eqref{eqn:stationary_point_condition_2} for some distinct indices $\{i_1,\ldots,i_k\}\subset\{1,\ldots,\rank{H}\}$, then by using \eqref{eqn:stationary_point_condition_2} and $\mathcal{C}_\pos^{1/2}H\mathcal{C}_\pos^{1/2} = \sum_{i}^{}(-\lambda_i)v_i\otimes v_i$ and direct computation, $U$ also satisfies \eqref{eqn:stationary_point_condition}. Thus, $U$ is a stationary point if and only if there exists an index set $\mathcal{I}\subset\{1,\ldots,\rank{H}\}$ of cardinality at most $r$ and containing distinct indices such that
	\begin{align*}
		UU^* &= \sum_{j\in\mathcal{I}}^{}(-\lambda_{j})\mathcal{C}_\pos^{-1/2}v_{j}\otimes \mathcal{C}_\pos^{-1/2}v_{j}.
	\end{align*}
	In particular, by \Cref{lemma:range_of_square_of_finite_rank}, the stationary points $U$ of $J_f$ satisfy $U\in\B(\R^r,\mathcal{V})$,
	where,
	\begin{align*}
		\mathcal{V}&\coloneqq\Span{\mathcal{C}_\pos^{-1/2}v_i,i=1,\ldots,\rank{H}}\subset\H.
	\end{align*}

	\textbf{Step 2}: computation of the stationary points of $J_f$ with minimal value of $J_f$.

	Since $v_i=\sqrt{1+\lambda_i}\mathcal{C}_\pos^{-1/2}\mathcal{C}_\pr^{1/2}w_i$ for all $i$, it holds by \eqref{eqn:bayesian_cov_pencil} that $v_i = \sqrt{1+\lambda_i}^{-1}\mathcal{C}_\pos^{1/2}\mathcal{C}_\pr^{-1/2}w_i$ for all $i$.
	Thus, if $U$ is a stationary point such that the above expression of $UU^*$ holds for the index set $\mathcal{I}$, then we have by \eqref{eqn:definition_g} and \eqref{eqn:posterior_preconditioned_Hessian}
	\begin{align*}
		g(U) &= \mathcal{C}_\pos^{1/2}UU^*\mathcal{C}_\pos^{1/2} - \mathcal{C}_\pos^{1/2}H\mathcal{C}_\pos^{1/2}\\
		&= 
		\sum_{j\in\mathcal{I}}^{}(-\lambda_{j})\frac{\mathcal{C}_\pos^{1/2}\mathcal{C}_{\pr}^{-1/2}w_{j}}{\sqrt{1+\lambda_{j}}}\otimes \frac{\mathcal{C}_\pos^{1/2}\mathcal{C}_{\pr}^{-1/2}w_{j}}{\sqrt{1+\lambda_{j}}}
		+ \sum_{i}^{}\lambda_{i} \frac{\mathcal{C}_\pos^{1/2}\mathcal{C}_{\pr}^{-1/2}w_{i}}{\sqrt{1+\lambda_{i}}}\otimes \frac{\mathcal{C}_\pos^{1/2}\mathcal{C}_{\pr}^{-1/2}w_{i}}{\sqrt{1+\lambda_{i}}},
	\end{align*}
	and hence the eigenvalues of $g(U)$ form the set $\{0\}\cup \{\lambda_i:\ i\not\in\mathcal{I}\}\subset(-1,0]$. Since $f\in\mathscr{F}$, it holds that $f(0)=0$, and it then follows from \eqref{eqn:J_expression} that $J_f(U)=\sum_{i\not\in \mathcal{I}}^{}f(\lambda_i)$. 
	Furthermore, $f$ is decreasing on $(-1,0]$. Let $\hat{U}$ be a stationary point corresponding to the index set $\mathcal{I}$. Then $\hat{U}$ minimises $J_f$ among its stationary points if and only if the sequence $(\lambda_i)_{i\in \mathcal{I}}$ contains the $r$ most negative elements of $(\lambda_i)_i$. In turn, this is the case if and only if $\mathcal{I}=\mathcal{I}^\opt$ where $\mathcal{I}^\opt$ is any index set for which $\{\lambda_i:\ i\in \mathcal{I}^\opt\}=\{\lambda_1,\ldots,\lambda_r\}$. The set $\mathcal{I}^\opt$ is uniquely defined if and only if $\lambda_r<\lambda_{r+1}$, in which case $\mathcal{I}^\opt=\{1,\ldots,r\}$.
	Thus, using once more $v_i=\sqrt{1+\lambda_i}^{-1}\mathcal{C}_\pos^{1/2}\mathcal{C}_\pr^{-1/2}w_i$, $\hat{U}$ satisfies
	\begin{align}
		\hat{U}\hat{U}^* = \sum_{i=1}^{r}\frac{-\lambda_{i}}{1+\lambda_i}\mathcal{C}_{\pr}^{-1/2}w_{i}\otimes\mathcal{C}_{\pr}^{-1/2}w_{i},
		\label{eqn:optimal_U}
	\end{align}
	 and $J_f(\hat{U})=\sum_{i>r}^{}f(\lambda_i)$. 
	 Now, $\hat{U}\hat{U}^*$ is uniquely defined if and only if either $\lambda_{r+1}=0$ or $\mathcal{I}^\opt$ is uniquely defined.
	 Hence $\hat{U}\hat{U}^*$ is unique if and only if either $\lambda_{r+1}=0$ or $\lambda_r<\lambda_{r+1}$.

	 \textbf{Step 3}: identification of the stationary points with minimal value of $J_f$ as the minimisers of $J_f$.
	
	 To prove that the minima of $J_f$ are precisely the stationary points $\hat{U}$ defined in \textbf{Step 2}, we first show for fixed $U\in\B(\R^r,\H)$ that $\restr{J_f}{\B(\R^r,\mathcal{V}+\ran{U})}'(U)=0$ implies $J_f'(U)=0$. Using \Cref{lemma:properties_g}, we can diagonalise $g(U)=\sum_{i}^{}\gamma_i e_i\otimes e_i$ with $(e_i)_i\subset\ran{\mathcal{C}_\pos^{1/2}}$. Using the fact that $\mathcal{C}_{\pos}^{1/2}H\mathcal{C}_{\pos}^{1/2}=\sum_i (-\lambda_i) v_i\otimes v_i$, and using the definition of $\mathcal{V}$, it follows that $\mathcal{C}_\pos^{1/2}\mathcal{V}=\ran{\mathcal{C}_\pos^{1/2}H\mathcal{C}_\pos^{1/2}}$, and by \Cref{lemma:range_of_square_of_finite_rank}, $\ran{\mathcal{C}_\pos^{1/2}UU^*\mathcal{C}_\pos^{1/2}}=\ran{\mathcal{C}_\pos^{1/2}U}$. Thus, for each $j$ for which $\gamma_j\not=0$, the identity $\gamma_j e_j = g(U)e_j = \mathcal{C}_\pos^{1/2}UU^*\mathcal{C}_\pos^{1/2}e_j-\mathcal{C}_\pos^{1/2}H\mathcal{C}_\pos^{1/2}e_j$ implies that $e_j\in\ran{\mathcal{C}_\pos^{1/2}UU^*\mathcal{C}_\pos^{1/2}}+\ran{\mathcal{C}_\pos^{1/2}H\mathcal{C}_\pos^{1/2}} \subset \ran{\mathcal{C}_\pos^{1/2}U}+\mathcal{C}_\pos^{1/2}\mathcal{V}$. Hence $\mathcal{C}_\pos^{-1/2}e_j\in\ran{U}+\mathcal{V}$. Now, by the expression of the derivative of $J_f$ in \Cref{prop:differentiability_J}, $\restr{J_f}{\B(\R^r,\mathcal{V}+\ran{U})}'(U)=0$ implies $2\sum_{i}^{}f'(\gamma_i)\langle \mathcal{C}_\pos^{1/2}e_i,VU^* \mathcal{C}_\pos^{1/2}e_i\rangle=0$ for all $V\in\B(\R^r,\mathcal{V}+\ran{U})$. For any $\varphi\in\R^r$, $V\coloneqq \mathcal{C}_\pos^{-1/2}e_j\otimes \varphi\in\B(\R^r,\mathcal{V}+\ran{U})$ and hence $0=2\sum_{i}^{}f'(\gamma_i)\langle \mathcal{C}_\pos^{1/2}e_i,\mathcal{C}_\pos^{-1/2}e_j\rangle\langle \varphi,U^*\mathcal{C}_\pos^{1/2}e_i\rangle=2f'(\gamma_j)\langle\varphi,U^*\mathcal{C}_\pos^{1/2}e_j\rangle$. Since $\gamma_j\not=0$, $f'(\gamma_j)\not=0$ by \Cref{lemma:finite_loss}\ref{item:properties_spectral_f}. Thus, $U^*\mathcal{C}_\pos^{1/2}e_j=0$. We conclude that $\gamma_iU^*\mathcal{C}_\pos^{1/2}e_i=0$ for all $i$. As was shown in \textbf{Step 1} of the proof, it then holds that $J_f'(U)=0$.

	 By \Cref{lemma:coercivity}, $\hat{U}$ minimises $J_f$ over $\B(\R^r,\mathcal{V})$ for the space $\mathcal{V}$ defined above.
	 Furthermore, if $\tilde{U}\in\B(\R^r,\H)$ with $\ran{\tilde{U}}\not\subset\mathcal{V}$, then $\tilde{U}$ is not a stationary point of $J_f$, because a necessary condition for $\tilde{U}$ to be a stationary point of $J_f$ is that $\ran{\tilde{U}}\subset\mathcal{V}$, by \textbf{Step 1}. By the previous paragraph, $\tilde{U}$ is not a stationary point of $J_f$ restricted to $\B(\R^r,\mathcal{V}+\ran{\tilde{U}})$. Since $\hat{U}\in\B(\R^r,\mathcal{V})\subset\B(\R^r,\mathcal{V}+\ran{\tilde{U}})$, since $\tilde{U}\in\mathcal{B}(\R^r,\mathcal{V}+\ran{\tilde{U}})$ and since $J_f$ is coercive over $\B(\R^r,\mathcal{V}+\ran{\tilde{U}})$ by \Cref{lemma:coercivity}, it follows that $J_f(\tilde{U})>J_f(\hat{U})$. Thus, $\hat{U}$ is a global minimiser of $J_f$. 

	 \textbf{Step 4}: identification of the solutions of \Cref{prob:optimal_precision,prob:optimal_covariance}.

	Since $J_f(\hat{U})=\mathcal{L}_f(\mathcal{C}_\pos\Vert (\mathcal{C}_\pr^{-1}+\hat{U}\hat{U}^*)^{-1})$ by \eqref{eqn:definition_Jf}, it follows that the operator $\mathcal{P}^\opt_r$ in \eqref{eqn:optimal_precision} solves \Cref{prob:optimal_precision}.
	By \Cref{cor:correspondence}\ref{item:problem_equivalence}, $(\mathcal{P}^\opt_r)^{-1}$ solves \Cref{prob:optimal_covariance}. It remains to show that $\mathcal{C}^\opt_r$ defined in \eqref{eqn:optimal_covariance} satisfies $\mathcal{C}^\opt_r=(\mathcal{P}^\opt_r)^{-1}$. Since taking the inverse is a bijective operation, uniqueness of $\mathcal{C}^\opt_r$ is then implied by uniqueness of $\mathcal{P}^\opt_r$. 	
	Now, by \eqref{eqn:optimal_precision}, \Cref{lemma:inverse_of_self_adj_hilbert_schmidt_perturbation} with $\delta_i\leftarrow \lambda_i$, and \eqref{eqn:optimal_covariance},
	\begin{align*}
		(\mathcal{P}^\opt_r)^{-1} 
		&= \left( \mathcal{C}_{\pr}^{-1}+\sum_{i=1}^{r}\frac{-\lambda_i}{1+\lambda_i}\mathcal{C}_{\pr}^{-1/2}w_i\otimes \mathcal{C}_{\pr}^{-1/2}w_i \right)^{-1} \\
		&=  \left( \mathcal{C}_{\pr}^{-1/2}\left(I-\sum_{i=1}^{r}\frac{-\lambda_i}{1+\lambda_i}w_i\otimes w_i\right)\mathcal{C}_{\pr}^{-1/2}\right)^{-1} \\
		&= \mathcal{C}_{\pr}^{1/2}\left(I+\sum_{i=1}^{r}\lambda_iw_i\otimes w_i\right)\mathcal{C}_{\pr}^{1/2}\\
		&=  \mathcal{C}_{\pr} - \sum_{i=1}^{r}(-\lambda_i)\mathcal{C}_{\pr}^{1/2}w_i\otimes \mathcal{C}_{\pr}^{1/2}w_i
		= \mathcal{C}^\opt_r.
	\end{align*}
\end{proof}

\bibliographystyle{abbrv}
\bibliography{references}

\end{document}